\documentclass[12pt]{amsart}

\usepackage{amsmath,graphicx}
\usepackage{adjustbox}
\usepackage{amssymb}
\usepackage{tikz, tikz-cd}
\usepackage{fullpage}
\usepackage[all]{xy}
\usepackage[margin=1in]{geometry}
\usepackage{amsthm}
\usepackage{amsfonts}
\usepackage{bbm}
\usepackage{eufrak}
\usepackage{comment}
\usepackage{amsmath}
\usepackage{amsthm,diagbox}
\usepackage{eufrak}
\usepackage{array} 
\usepackage{color}
\usepackage[utf8]{inputenc}
\usepackage[english]{babel}
\usepackage{hyperref}
\usepackage{mathtools}
\newcolumntype{L}{>{$}l<{$}}

\DeclareMathSymbol{\shortminus}{\mathbin}{AMSa}{"39}

\newtheorem{theorem}{Theorem}[section]
\newtheorem{lemma}[theorem]{Lemma}

\newtheorem{corollary}[theorem]{Corollary}

\newtheorem{proposition}[theorem]{Proposition}
\theoremstyle{definition}  
\newtheorem{definition} [theorem] {Definition} 

\newtheorem{example} [theorem] {Example}
\newtheorem{remark} [theorem] {Remark}
\newtheorem{question} [theorem] {Question}

\theoremstyle{definition}

\newcommand{\C}{{\mathbb{C}}}
\newcommand{\F}{{\mathbb{F}}}

\newcommand{\Q}{{\mathbb{Q}}}

\newcommand{\R}{{\mathbb{R}}}
\newcommand{\Z}{{\mathbb{Z}}}

\newcommand{\surj}{\twoheadrightarrow}
\newcommand{\inj}{\hookrightarrow}

\newcommand{\lk}{\ell\text{k}}

\newcommand{\spin}{\text{spin}}

\newcommand{\T}{\mathbf{T}}
\newcommand{\s}{\mathfrak{s}}
\newcommand{\PD}{\text{PD}}

\DeclareMathOperator{\HFL}{HFL}
\DeclareMathOperator{\HFK}{HFK}

\DeclareMathOperator{\HF}{HF}

\DeclareMathOperator{\CKh}{CKh}
\DeclareMathOperator{\Kh}{Kh}

\DeclareMathOperator{\AKh}{AKh}
\DeclareMathOperator{\SFH}{SFH}

\DeclareMathOperator{\AHI}{AHI}

\DeclareMathOperator{\rank}{rank}
\DeclareMathOperator{\LLL}{L}
\DeclareMathOperator{\gr}{gr}
\DeclareMathOperator{\Cone}{Cone}

\title{Floer homology, clasp-braids and detection results}
\author[Binns]{Fraser Binns}
\author[Dey]{Subhankar Dey}
\address[]{Department of Mathematics, Princeton University}
\email{fb1673@princeton.edu}
\address[]{Department of Mathematical Sciences, Durham University}
\email{subhankar.dey@durham.ac.uk}
\thanks{FB was supported by the Simons Grant {\em New structures in low-dimensional topology}. SD was supported by an individual
research grant of the DFG, project number 505125645.}

\begin{document}

\begin{abstract}

 Martin showed that link Floer homology detects braid axes. In this paper we extend this result to give a topological characterisation of links which are almost braided from the point of view of link Floer homology. The result is inspired by work of Baldwin-Sivek and Li-Ye on nearly fibered knots. Applications include that Khovanov homology detects the Whitehead link and L7n2, as well as infinite families of detection results for link Floer homology and annular Khovanov homology.
\end{abstract}

\maketitle

\section{Introduction}
Link Floer homology is a vector space valued invariant of oriented links in $S^3$ due to Ozsv\'ath-Szab\'o~\cite{HolomorphicdiskslinkinvariantsandthemultivariableAlexanderpolynomial}. We are broadly interested in determining the topological information that link Floer homology encodes. There are a number of results in this direction. Notably, link Floer homology detects the Thurston norm~\cite[Theorem 1.1]{ozsvath2008linkFloerThurstonnorm} and whether or not $L$ is fibered~\cite{ni2007knot,ni2006note,ghiggini2008knot}. The most relevant prior result for the purposes of this paper is Martin's result that link Floer homology detects whether or not a given link component is a braid axis~\cite{martin2022khovanov}.

To state Martin's result precisely, recall that a knot $K$ in $S^3$ is \emph{fibered} if its complement is swept out by a family $\{\Sigma_t\}_{t\in S^1}$ of Seifert surfaces for $K$ with disjoint interiors. A link $L$ is \emph{braided} with respect to $K$ if $L$ can be isotoped so that it intersects each page $\Sigma_t$ transversely. Now recall that if a knot $K$ is a component of a link $L$ then the link Floer homology of $L$, $\widehat{\HFL}(L)$, carries a grading called the \emph{Alexander grading for $K$}, which we denote $A_K$. Martin's result can be stated as follows:

\begin{proposition}[Proposition 1~\cite{martin2022khovanov}]
    Let $L$ be an $n$-component non-split link with a component $K$. $\widehat{\HFL}(L)$ is of rank at least $2^{n-1}$ in the maximal non-trivial $A_K$ grading and equality holds if and only if $K$ is fibered and $L-K$ is braided with respect to $K$.
\end{proposition}
 In this paper our main goal is to give a version of Martin's result for links $L$ with a component $K$ such that the rank of the link Floer homology is of next to minimal rank in the maximal non-trivial $A_K$ grading. In particular we prove the following theorem:

\newtheorem*{thm:mainbraid}{Theorem~\ref{thm:mainbraid}}
\begin{thm:mainbraid}
  Let $L$ be an $n$-component non-split link in $S^3$ with a component $K$. The link Floer homology of $L$ in the maximal non-trivial $A_K$ grading is of rank at most $2^n$ if and only if one of the following holds:

    \begin{enumerate}
        \item $K$ is fibered and $L-K$ is braided with respect to $K$.
        \item $K$ is nearly fibered and $L-K$ is braided with respect to $K$.
        \item $K$ is fibered and $L-K$ is a clasp-braid closure with respect to $K$.
        \item $K$ is fibered and $L$ is a stabilized clasp-braid closure with respect to $K$.
        \item $K$ is fibered with fibration $\{\Sigma_t\}_{t\in S^1}$ and there exists a component $K'$ of $L$ such that $K'$ is isotopic to a simple closed curve in $\Sigma_t\setminus (L-K')$ and $L\setminus(K\cup K')$ is braided with respect to $K$.
    \end{enumerate}

\end{thm:mainbraid}

We defer the definitions of the terms involved in the statement of this Theorem to Section~\ref{sec:almostbraids}. For now we give only the following descriptions:

\emph{Nearly fibered knots} are knots which are ``nearly" fibered from the point of view of link Floer homology. Examples include $5_2$ and the Pretzel knots  $P(-3,3,2n+1)$. These knots were first studied by Baldwin-Sivek, who gave a classification of genus $1$ nearly fibered knots~\cite{baldwin2022floer}. They used this classification to obtain botany results for link Floer homology and \emph{Khovanov homology} -- a categorified link invariant due to Khovanov~\cite{khovanov2000categorification}. For example, they showed that link Floer homology and Khovanov homology detect the knot $5_2$. Li-Ye gave a topological description of nearly fibered knots of arbitrary genus which we shall make use of in due course~\cite{li2022seifert}.

A link $L$ is ``braided" with respect to a nearly fibered knot if after appropriate cuts and excisions $L$ looks like a braid. A ``clasp-braid" is a tangle obtained from a  braid by replacing a small ball containing two parallel strands with a ``clasp" -- see Figure~\ref{fig:nearlybraided}. A ``stabilization" of a clasp-braid is a tangle that can be obtained from a clasp-braid by adding parallel strands, subject to certain conditions.

Finally we note that in case $(5)$ $K'$ does not bound a disk in $\Sigma_t\setminus(L-K')$ as $L$ is assumed to be non-split.

\begin{center}
 
     \begin{figure}[h]
 \includegraphics[width=5cm]{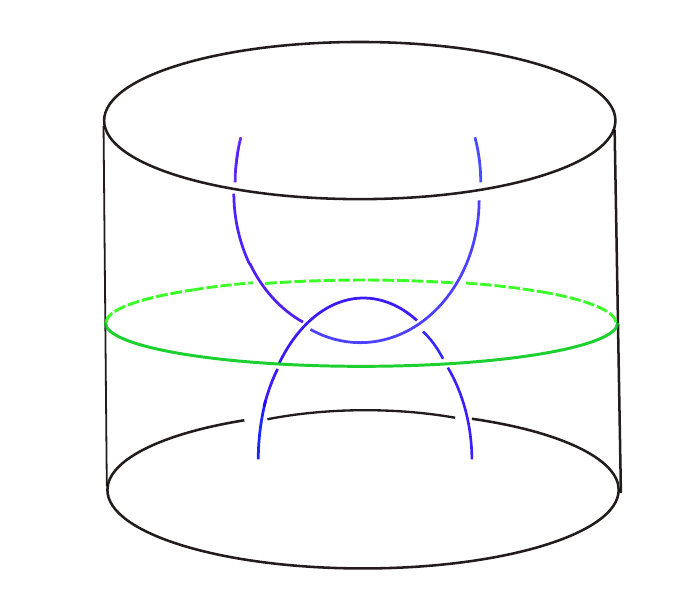}
    \caption{A clasp. The green curve indicates the core of the suture, while the blue curves indicate the tangle $T$.}\label{fig:nearlybraided}
   \end{figure}
\end{center}

The proof of Theorem~\ref{thm:mainbraid} makes extensive use of Juh\'asz's \emph{sutured Floer homology} theory, a generalization of link Floer homology to manifolds with appropriately decorated boundaries -- namely Gabai's \emph{sutured manifolds}. Specifically, given an $n$ component link $L$, the idea is to use sutured techniques to decompose the exterior of $L$ into smaller sutured manifolds which can be classified, and then determine the ways in which these pieces can be glued back together. This strategy was used by Baldwin-Sivek~\cite{baldwin2022floer} and Li-Ye~\cite{li2022seifert} in their work on nearly fibered knots, as well as in a generalization of Baldwin-Sivek's work to certain ``nearly fibered links" due to Cavallo-Matkovi\v{c}~\cite{cavallo2023nearly}.

The proof of Theorem~\ref{thm:mainbraid} passes through a classification of sutured link exteriors of small rank, Theorem~\ref{thm:mailinknrankclass}. This result is a generalization of Ni's classification of links in $S^3$ with link Floer homology of minimal rank~\cite[Proposition 1.4]{ni2014homological} and Kim's classification of links in $S^3$ with link Floer homology of next to minimal rank~\cite[Theorem 1]{kim2020links}. Since the statement of Theorem~\ref{thm:mailinknrankclass} is long and somewhat technical, we defer it to Section~\ref{sec:linkext}.

Theorem~\ref{thm:mainbraid} can be sharpened in a number of ways: the hypothesis that $L$ is not split can be readily removed -- see Corollary~\ref{rem:reducible} -- while $\widehat{\HFL}(L)$ in the maximum non-trivial $A_K$ hyperplane can be computed as a vector space equipped with $n-1$ Alexander gradings up to translation -- see Section~\ref{subsec:geography}.

Martin used her braid axis detection result as an ingredient in her proof that Khovanov homology detects $T(2,6)$~\cite[Theorem 1]{martin2022khovanov}. Theorem~\ref{thm:mainbraid} can likewise be used as an ingredient for the following Khovanov homology detection results:

\newtheorem*{thm:Khdetects}{Theorem~\ref{thm:Khdetects}}
\begin{thm:Khdetects}
    Khovanov homology detects the Whitehead link and $L7n2$.
\end{thm:Khdetects}

Indeed, we prove this result using Martin's strategy for $T(2,6)$ detection. More specifically, if $L$ is a link with the Khovanov homology type of the Whitehead link or $L7n2$, we use Lee's spectral sequence~\cite{lee2005endomorphism} to deduce the number of components of $L$ and their linking numbers and Batson-Seed's link splitting spectral sequence~\cite{batson2015link} to deduce information about the components of $L$. We then use Dowlin's spectral sequence~\cite{dowlin2018spectral} from Khovanov homology to \emph{knot Floer homology} to reduce the question to one concerning knot Floer homology. Here knot Floer homology is a version of link Floer homology due independently to Ozsv\'ath-Szab\'o~\cite{Holomorphicdisksandknotinvariants} and J.Rasmussen~\cite{Rasmussen}.

Other knot detection results for Khovanov homology include unknot detection~\cite{kronheimer2010knots} as well as a number of other small genus examples~\cite{farber2022fixed,baldwin2022floer,baldwin2018khovanov,baldwin2020khovanov}. There are larger numbers of detection results for links with at least two components~\cite{martin2022khovanov, binns2020knot,xie2020links,li2020two}.

Somewhat surprisingly, it is not the case that link Floer homology detects the Whitehead link or L7n2. In fact, the Whitehead link and L7n2 have the same link Floer homology\footnote{The link Floer homology of L7n2 was computed in~\cite[Section 12.2.1]{ozsvath2008holomorphic}. However, we note that there is a small error in the computation of the Maslov gradings.}. However, link Floer homology -- as well as knot Floer homology -- detect membership of the set consisting of the Whitehead link and L7n2, see Theorem~\ref{thm:HFKW} and Theorem~\ref{thm:HFLwhite}. On the other hand, using Theorem~\ref{thm:mainbraid} and some work of the second author, King, Shaw, Tosun and Trace~\cite{dey2023unknotted}, we can give the following infinite family of link detection results.

\newtheorem*{thm:HFLLn}{Theorem~\ref{thm:HFLLn}}
\begin{thm:HFLLn}\label{thm:HFLLn}
    Link Floer homology detects the $n$-twisted Whitehead link for $n\not\in\{-2,-1,0,1\}$.
\end{thm:HFLLn}

The \emph{$n$-twisted Whitehead links} are the links shown in Figure~\ref{L_n}. They can alternately be characterized as the index $2$-clasp-braids with respect to the unknot. Note that $L_{-2}$ and $L_{1}$ are the links L7n2 and its mirror, while $L_0$ and $L_{-1}$ are the Whitehead link and its mirror.

\begin{center}
    \begin{figure}[h]
 \includegraphics[width=5cm]{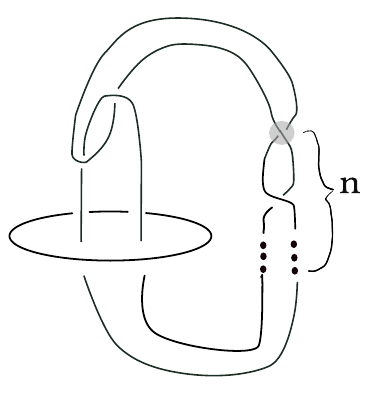}
    \caption{The $n$-twisted Whitehead links are shown here. The Whitehead link, or $L5a1$, corresponds to $n=0$, $L7n2$ corresponds to the $n=1$ case. The highlighted crossing is one we will have course to resolve. In this picture $n \geq 0$.}
\label{L_n}
\end{figure}
\end{center}

We also obtain an infinite family of \emph{annular Khovanov homology} detection results for a related family of annular links:

\newtheorem*{thm:AKH}{Theorem~\ref{thm:AKH}}
\begin{thm:AKH}
    Annular Khovanov homology detects each wrapping number 2 clasp-braid closure amongst annular knots.
\end{thm:AKH}

These annular links are those shown in Figure~\ref{annularlinks}.

\begin{center}

    \begin{figure}[h]
 \includegraphics[width=4.8cm]{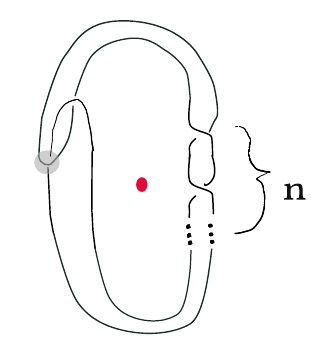}
    \caption{Wrapping number $2$ clasp-braids. The red dot indicates the annular axis, while the highlighted crossing is one which we will have course to resolve. In this picture $n \geq 0$.} \label{annularlinks}
\end{figure}
\end{center}

Annular Khovanov homology is a version of Khovanov homology for links in the thickened annulus due to Aseda-Przytycki-Sikora~\cite{asaeda_categorification_2004}. The \emph{wrapping number} of an annular link $L$ is the minimal geometric intersection number of a meridional disk and $L$.

Theorem~\ref{thm:AKH} should be compared with the result that annular Khovanov homology detects all $2$-braids -- which follows from a computation of Grigsby-Licata-Wehrli~\cite[Section 9.3]{grigsby_annular_2018} and the Grigsby-Ni's result that annular Khovanov homology detects braid closures~\cite{grigsby_sutured_2014}. Other annular Khovanov homology detection results have been given by the first author and Martin~\cite{binns2020knot}, the authors~\cite{binns2022rank}, and Baldwin-Grigsby~\cite{BG}. We note also in passing that if Khovanov homology detects a link then annular Khovanov homology detects that link when we view it as embedded in a $3$-ball in the thickened annulus, a result which follows from work of Xie-Zhang~\cite{xie2020links}.

As an intermediate step towards Theorem~\ref{thm:AKH}, we show the following;

\newtheorem*{prop:AKH2}{Proposition~\ref{prop:AKH2}}
\begin{prop:AKH2}
    Suppose $K$ is an annular knot. Suppose that the maximal non-trivial annular grading of $\AKh(K;\C)$ is two and that the $\rank(\AKh(K;\C,2))\leq 2$. Then either:
    \begin{enumerate}
        \item $K$ is a braid and $\rank(\AKh(K;\C,2))=1$.
        \item $K$ is a clasp-braid closure and $\rank(\AKh(K;\C,2))=2$.
    \end{enumerate}
\end{prop:AKH2}

The proof of this result uses a relationship between annular Khovanov homology and various versions of \emph{Instanton Floer homology}; see, for example,~\cite{kronheimer2010knots,kronheimer2011khovanov,kronheimer2011knot}. Instanton Floer invariants share many structural properties with Heegaard Floer invariants. Indeed, it is conjectured that Instanton Floer homology with complex coefficients agrees with Heegaard Floer homology~\cite[Conjecture 7.24]{kronheimer2010knots}. For the proof of Proposition~\ref{prop:AKH2} we use Xie's spectral sequence from annular Khovanov homology to annular instanton Floer homology~\cite{xie_instanton_2019} and a description of a specific summand of annular instanton Floer homology in terms of an instanton Floer homology invariant for tangles in sutured manifolds due to Xie-Zhang~\cite{xie2020links}. We then use a characterization of annular links with annular instanton Floer homology of next to minimal rank that is formally identical to that given in Theorem~\ref{thm:mainbraid} in the case that the component $K$ is an unknot.

It is natural to ask if there is a more general version of Proposition~\ref{prop:AKH2} in the style of Theorem~\ref{thm:mainbraid}.

\begin{question}\label{q:nearlybraidedannular}
   Is there a topological classification of Annular links with annular Khovanov homology of next to minimal rank in the maximum non-trivial annular grading?
\end{question}
 Note that annular links with annular Khovanov homology of minimal rank in the maximal non-trivial annular grading are braid closures by a result of Grigsby-Ni~\cite{grigsby_sutured_2014}. See Remark~\ref{rem:AKHAHI} for some further discussion.

The outline of this paper is as follows: in Section~\ref{sec:background} we review sutured manifolds, sutured Floer homology and link Floer homology. In Section~\ref{sec:almostbraids} we define the terms used in the statement of Theorem~\ref{thm:mainbraid} and discuss related notions. A classification of certain sutured link exteriors is given in Section \ref{sec:linkext}. In Section~\ref{sec:mainthmpf} we prove the main theorem, Theorem~\ref{thm:mainbraid}. The first applications of Theorem~\ref{thm:mainbraid} are given in Section~\ref{sec:HFLW}, where we determine the links with the same knot and link Floer homology as the Whitehead link, and prove Theorem \ref{thm:HFLLn}. In Section~\ref{sec:KH} we prove our two Khovanov homology detection results and in Section~\ref{sec:annular} we prove our annular Khovanov homology detection results.

\subsection*{Acknowledgements}

 We are grateful for Claudius Zibrowius' help with some of the immersed curve computations in Section~\ref{sec:HFLW} and his feedback on an earlier draft.  We would also like to thank Ilya Kofman for some helpful comments. The first author would like to thank his Ph.D. advisor, John Baldwin, as well as Gage Martin for illuminating conversations. The second author would like to acknowledge partial support of NSF grants DMS 2144363, DMS 2105525, and AMS-Simons travel grant while he was a postdoc at the University of Alabama, when the project started.

\section{Background Material}\label{sec:background}
In this section we review aspects of sutured manifold theory and sutured Floer homology, with an emphasis on those which will be relevant in subsequent sections.
\subsection{Sutured Manifolds}

Sutured Manifolds were first defined by Gabai in the 1980's for the purposes of studying taut foliations~\cite{gabai1983foliations}. In this paper we will only be interested in a subclass of sutured manifolds called \emph{balanced} sutured manifolds. We duly suppress the term ``balanced" henceforth. 

Sutured manifolds can now be defined as manifolds which can be obtained via the following construction. Take an oriented surface with non-empty boundary, $\Sigma$. Consider $\Sigma\times[-1,1]$. Let $\{\alpha_i\}_{1\leq i\leq n},\{\beta_i\}_{1\leq i\leq n}$ be collections of homologically independent curves in $\Sigma\times\{1\}$ and $\Sigma\times\{-1\}$ respectively. Attach thickened disks $D^2\times[-1,1]$ along $\partial D^2\times[-1,1]$ to neighborhoods of each $\alpha_i,\beta_i$. Denote this manifold by $Y$. $\partial Y$ comes endowed with a decomposition into $R_+\cup R_-\cup \gamma$ where $\gamma$ is given by $\partial\Sigma\times[-1,1]$, while $R_+(\gamma)$ is the interior of the components of $\partial Y\setminus \gamma$ with boundary $\partial\Sigma\times\{1\}$ while $R_-(\gamma)$ is the interior of the components of $\partial Y-s(\gamma)$ with boundary $\partial\Sigma\times\{-1\}$. $\partial Y$ also comes equipped with an orientation for which the normal vector of $R_+(\gamma)$ points out of $Y$ while the normal vector of $R_-(\gamma)$ points into $Y$. The pair $(Y,\gamma)$ is then a sutured manifold~\footnote{Technically, we should here smooth the corners of the $3$-manifold. We suppress this and all other such smoothings throughout this paper, as is customary.}.

\begin{example}
A \emph{product sutured manifold} is the product manifold $\Sigma\times[-1,1]$ endowed with sutures $\gamma=\partial\Sigma\times [-1,1]$. We call $\Sigma$ the \emph{base} of the product sutured manifold.
\end{example}

Let $R(\gamma)=R_+(\gamma)\cup R_-(\gamma)$. We will typically be interested in sutured manifolds for which $R(\gamma)$ is of minimal complexity in the following sense:
\begin{definition}
A sutured manifold $(Y,\gamma)$ is \emph{taut} if $Y$ is irreducible and $R(\gamma)$ is incompressible and Thurston norm minimizing in $H_2(Y,\gamma)$.
 \end{definition}
Another condition that is often a hypothesis on theorems concerning sutured manifolds is the following:
\begin{definition}
    A sutured manifold $(Y,\gamma)$ is \emph{strongly balanced} if each connected component $F$ of $\partial Y$ we have that $\chi(F\cap R_+(\gamma))=\chi(F\cap R_-(\gamma))$.
\end{definition}
 Sutured manifolds are of use to us primarily because they behave well under removing neighborhoods of certain surfaces. These surfaces are required to satisfy a number of conditions:
\begin{definition}
A \emph{decomposing surface} in a sutured manifold $(Y,\gamma)$ is a properly embedded surface $S$ in $Y$ such that no component of $\partial S$ bounds a disk in $R(\gamma)$ and no component of $S$ is a disk $D$ with $\partial D\subset R(\gamma)$. Moreover, for every component $\lambda$ of $S\cap \gamma$ we require that one of the following holds;\begin{enumerate}
    \item $\lambda$ is a properly embedded non-separating arc in $\gamma$ such that $|\lambda\cap s(\gamma)|=1$.
    \item $\lambda$ is a simple closed curve in a component $A$ of $\gamma$ in the same homology class as $A\cap s(\gamma)$.
\end{enumerate}
\end{definition}

If $\Sigma$ is a decomposing surface in a sutured manifold $(Y,\gamma)$ then the exterior of $\Sigma$ in $Y$, $Y'$, is naturally endowed with the structure of a sutured manifold, $(Y',\gamma')$. The operation of removing a neighborhood of $\Sigma$ from $(Y,\gamma)$, to obtain $(Y',\gamma')$ is called \emph{sutured decomposition.}

For this operation to play nicely with ``sutured Floer homology" -- an invariant we will discuss in the next subsection -- we require additional hypotheses on the surface. To explain this, let $v_0$ be a vector field on $(Y,\gamma)$, satisfying conditions as in~\cite[Section 4]{juhasz2006holomorphic}. For a surface $S$ in $(Y,\gamma)$ let $\nu_S$ denote the normal vector of $S$.
\begin{definition}
    A decomposing surface $S$ in $(M,\gamma)$ is called \emph{nice} if $S$ is open, $\nu_S$ is nowhere parallel to $v_0$ and for each component $V$ of $R(\gamma)$ the set of closed components of $S\cap V$ consists of parallel, coherently oriented, boundary coherent simple closed curves.
\end{definition}

Here a curve $c\in \Sigma$ is \emph{boundary coherent} if either $[c]\neq 0$ in $H_1(\Sigma;\Z)$, or if
$[c]=0$ in $H_1(\Sigma;\Z)$ and $c$ is oriented as the boundary of its interior. Nice decomposing surfaces $S\inj Y$ are generic in the space of embeddings of $S$ in $Y$ so we will suppress dependencies on this condition henceforth.

There are certain classes of decomposing surfaces that play a particularly important role in the sutured Floer theory.

\begin{definition}
Let $(Y,\gamma)$ be a sutured manifold. A decomposing surface $S\subset Y$ is called \emph{horizontal} if $S$ is open and incompressible, $\partial S\subset \gamma$, $\partial S$ is isotopic to $\partial R_\pm(\gamma)$, $[S]=[R_\pm(\gamma)]$ in $H_2(Y,\gamma)$, and $\chi(S)=\chi(R_\pm(\gamma))$.
\end{definition}

We will typically assume that there are no ``interesting" horizontal surfaces, i.e. that the following definition applies:
\begin{definition}
A sutured manifold $(Y,\gamma)$ is \emph{horizontally prime} if every horizontal surface in $(M,\gamma)$ is parallel to either $R_+(\gamma)$ or $R_-(\gamma)$.
    
\end{definition}

Note that if a sutured manifold is not horizontally prime then it can be decomposed along a collection of horizontal surfaces into horizontally prime pieces by a result of Juh\'asz~\cite[Proposition 2.18]{juhasz2010sutured}.

Certain types of annuli will play an important role in this work.
\begin{definition}
A \emph{product annulus} $A$ is a properly embedded annulus in a sutured manifold $(Y,\gamma)$ such that $A_\pm\subset R_\pm(\gamma)$, where $A_\pm$ are the two boundary components of $A$.
   
\end{definition}
Note that product annuli need not be decomposing surfaces. Various theorems in sutured Floer homology require additional conditions on product annuli, for example:
\begin{definition}

A product annulus $A$, properly embedded in a sutured manifold $(M, \gamma)$, is called \emph{essential} if $A$ is incompressible and  cannot be isotoped into a component of $\gamma$ with the isotopy keeping $\partial A$ in $R(\gamma)$ throughout.
    
\end{definition}

\begin{remark}
    We will be interested in undoing the operation of decomposing along a product annulus. That is, we are interested in the operation by which we take a sutured manifold $(Y,\gamma)$ with at least two sutures $\gamma_1,\gamma_2\subset \gamma$ and construct a new sutured manifold $(Y',\gamma')$ such that $(Y,\gamma)$ contains a product annulus $A$ such that decomposing $(Y',\gamma')$ along $A$ yields $(Y,\gamma)$.
    
    We can construct $(Y',\gamma')$ as follows. Pick a homeomorphism $f:\gamma_1\to\gamma_2$ such that $f$ preserves both $R_+(Y,\gamma)\cap\gamma$ and $R_-(Y,\gamma)\cap\gamma$ setwise. Then set $x\sim y$ if and only if;\begin{itemize}
        \item $x=y$ and $x,y\in Y-(\gamma_1\cup\gamma_2)$.
        \item $f(x)=y$ with $x\in\gamma_1, y\in \gamma_2$
    \end{itemize}
Apriori, $(Y',\gamma')$ is dependant on the choice of $f$. However, the group of homeomorphisms of the annulus that preserves each boundary component setwise  up to isotopy through such homeomorhisms is trivial, so in fact $(Y',\gamma')$ is independent of $f$.
    
\end{remark}

We conclude our discussion of sutured manifolds by recalling the following measure of their complexity:

\begin{definition}[Scharlemann~\cite{scharlemann1989sutured}]
Let $(Y,\gamma)$ be a sutured manifold. Given a properly embedded surface $S\subset Y$, set: 
\[
    x^s(S)=\max\{0,\chi(S\cap R_-(\gamma))-\chi(S)\}.
\]
Extend this definition to disconnected surfaces linearly and thereby to a function\\ ${x^s:H_2(Y,\partial Y;\R)\to\R}$. Call $x^s$ the \emph{sutured Thurston norm}.
\end{definition}

Equivalently we can write $x^s(\alpha)=\max\{0,\frac{1}{2}|\Sigma\cap s(\gamma)|-\chi(\Sigma):[\Sigma]=\alpha\}$.

\subsection{Sutured Floer homology}\label{subsec:suturedFloerbackground}
To each sutured manifold $(Y,\gamma)$, Juh\'asz associates a finitely generated vector space denoted $\SFH(Y,\gamma)$~\cite{juhasz2006holomorphic}. Sutured Floer homology is defined using sutured Heegaard diagrams, symplectic topology and analysis.

$\SFH(Y,\gamma)$ caries an affine grading by relative $\spin^c$ structures on $Y$. Here $\spin^c$ structures can be viewed as a equivalence classes of vector fields that agree with $v_0$. We denote these by $\spin^c(Y,\partial Y)$. There is an affine isomorphism between $\spin^c(Y,\partial Y)$ and $H^2(Y,\partial Y;\Z)$ given by tubularization, see~\cite[p639]{turaev1990euler}. Moreover, Ponincar\'e duality gives an isomorphism $H^2(Y,\partial Y;\Z)\cong H_1(Y;\Z)$. We will be particularly interested in cases in which $Y$ is the exterior of a union of surfaces and a link in $S^3$. In this case $H_1(Y;\Z)$ has a summand generated by meridians of the arc and link components.

If $(Y_1,\gamma_1)$ and $(Y_2,\gamma_2)$ are sutured manifolds, then the sutured Floer homology of the connect sum is given by:\begin{align}\label{eq:connectsum}
    \SFH(Y_1\#Y_2,\gamma_1\#\gamma_2)\cong V\otimes \SFH(Y_1,\gamma_1)\otimes \SFH(Y_2,\gamma_2)
\end{align}

Where $V$ is a rank $2$ vector space. This can be proven at the level of sutured Heegaard diagrams. The graded version of the statement is the natural one.

\begin{definition}
    Let $(Y,\gamma)$ be a sutured manifold. The \emph{sutured Floer homology polytope} is the convex hull of $\{\s\in\spin^c(M,\gamma):\SFH(M,\gamma,\s)\neq0\}$.
\end{definition}

We denote the sutured polytope of $(Y,\gamma)$ by $P(Y,\gamma)$. We describe how to compute the dimension of the sutured polytope in practice, as it will be helpful later. Suppose $(Y,\gamma)$ is a sutured manifold with sutured Heegaard diagram $(\Sigma,\alpha,\beta)$. Fix an intersection point $x\in\T_\mathbf{\alpha}\cap\T_\mathbf{\beta}$. Recall that the difference between the relative $\spin^c$-structure of $\mathbf{x}$ and any other point $\mathbf{y}\in\T_\mathbf{\alpha}\cap\T_\mathbf{\beta}$ can be thought of as an element $\epsilon(\mathbf{x},\mathbf{y})\in H_1(Y)$. Moreover $H_1(Y)\cong H_1(\Sigma)/\langle\mathbf{\alpha},\mathbf{\beta}\rangle$. To compute $\epsilon(\mathbf{x},\mathbf{y})\in H_1(\Sigma)/\langle\alpha,\beta\rangle$ pick a symplectic basis $\{z_i\}$ for $H_1(\Sigma,\partial\Sigma)$. Pick an oriented curve $\gamma$ consisting of arcs which pass from each element $x\in\mathbf{x}$ with $x\in\alpha\in\mathbf{\alpha}$ to $y\in\mathbf{y}$ with $y\in\alpha$ along $\alpha$ and arcs which pass from each element $y\in\mathbf{y}$ with $y\in\beta\in\mathbf{\beta}$ to $y\in\mathbf{y}$ with $y\in\beta$ along $\beta$. The homology class $[\gamma]\in H_1(\Sigma)$ can be read off by taking a signed count of intersections with the symplectic basis for $\Sigma$. $\epsilon(\mathbf{x},\mathbf{y})$ is then the image of $[\gamma]$ in $H_1(\Sigma)/\langle\alpha,\beta\rangle$. The dimension of the sutured Floer polytope is then given by $\dim(\langle \epsilon(\mathbf{x},\mathbf{y}):\mathbf{y}\in\T_\alpha\cap\T_\beta\rangle)$.

The sutured Floer polytope has a topological interpretation by the following Theorem:

\begin{theorem}[Friedl-Juh\'asz-Rasmussen~\cite{friedl2011decategorification}]\label{thm:FLRpolytopedetection}
    Suppose that $(M,\gamma)$ is an irreducible sutured manifold with boundary a disjoint union of tori. Let $S$ be the support of $\SFH(M,\gamma)$. Then $\max_{\s,\s'\in S}\langle \s-\s',\alpha\rangle =x^s(\alpha)$.
\end{theorem}

This is a generalization of Ozsv\'ath-Szab\'o's result that the link Floer homology -- which we shall discuss in the next Subsection -- detects the Thurston norm~\cite{ozsvath2008linkFloerThurstonnorm}.

Sutured Floer homology as an ungraded object and -- to a lesser extent -- as a graded object behave well under sutured manifold decomposition. We recall a number of results witnessing this.

To state the first, we introduce some notation. If $(M,\gamma)\rightsquigarrow(M',\gamma')$ is a surface decomposition and $e:M'\inj M$ denotes the corresponding embedding then we set\begin{align*}
    F_S:=\PD_M\circ e_*\circ(\PD_{M'})^{-1}:H^2(M',\partial M')\to H^2(M,\partial M)
\end{align*}

where here $\PD_M$ is the Poincar\'e-Lefshetz duality map on $M$. Let $T(M,\gamma)$ denote the set of trivializations of $(M,\gamma)$ that restrict to $v_0$ on the boundary. Finally, let $O_S$ denote the set of outer $\spin^c$-structures on $(M,\gamma)$ with respect to $S$ -- see~\cite[Definition 1.1]{juhasz2008floer} for a definition.

\begin{proposition}[Proposition 5.4, Juh\'asz~\cite{juhasz2010sutured}]\label{Thm:Juhaszaffine}
    Let $(M,\gamma)\overset{S}{\rightsquigarrow} (M',\gamma')$ be a nice surface decomposition of a strongly balanced sutured manifold $(M,\gamma)$. Fix $t\in T(M,\gamma), t'\in T(M',\gamma')$. Then there is an affine map $f_S:\spin^c(M',\gamma')\to\spin^c(M,\gamma)$ such that:\begin{itemize}
        \item 
    $f_S$ surjects onto $O_S$ and for any $\s\in O_S$ we have: \begin{align*}
        \SFH(M,\gamma,\s)\cong \underset{\s'\in\spin^c(M',\gamma'):f_S(\mathfrak{s'})=\s}{\bigoplus}\SFH(M',\gamma',\s')
    \end{align*}
\item If $\s_1',\s_2'\in\spin^c(M',\gamma')$ then \begin{align*}F_S(\s_1'-\s_2')=f_S(\s_1')-f_S(\s_2').\end{align*}

    \end{itemize} 
\end{proposition}

This proposition gives a measure of control of the dimension of the sutured Floer polytope under surface decomposition.

The rank of sutured Floer homology behaves even better under decomposition along horizontal surfaces:

\begin{proposition}[Proposition 8.6, Juh\'asz~\cite{juhasz2008floer}]\label{prop:juhaszhorizontaldecomp}
    Let $(Y,\gamma)$ be a sutured manifold and $(Y',\gamma')$ be the manifold obtained by decomposing $(Y,\gamma)$ along a surface $\Sigma$ such that $(Y',\gamma')$ is taut, and $[\Sigma]=0$ in $H_2(Y,\gamma)$. Suppose moreover that $\Sigma$ is open and for every component $V$ of $R(\gamma)$ the set of closed components of $\Sigma\cap V$ consists of parallel oriented boundary-coherent simple closed curves. Then $(Y',\gamma')$ has two components $(Y'_1,\gamma'_1)$ and $(Y'_2,\gamma'_2)$ and 
    
    \begin{align*} \SFH(Y,\gamma)\cong\SFH(Y',\gamma')\cong\SFH(Y_1,\gamma_1)\otimes\SFH(Y_2,\gamma_2)    
    \end{align*}

    as ungraded vector spaces.
\end{proposition}

We also have a good deal of control of the rank under decomposition along product annuli. To explain this we require a further definition.

\begin{definition}
    A sutured manifold $(Y,\gamma)$ is \textit{reduced} if $(Y,\gamma)$ contains no essential product annulus.
\end{definition}

Juh\'asz showed if $(Y,\gamma)$ is a sutured manifold such that $H_2(Y)=0$ then the dimension of the sutured polytope is maximal if $(Y,\gamma)$ is horizontally prime and contains no essential product annuli~\cite[Theorem 3]{juhasz2010sutured}. We give a minor generalization of this Theorem, which we require in Sections~\ref{sec:linkext} and~\ref{sec:mainthmpf}.

\begin{lemma}\label{cor:trivialinclusion}
  
    Let $q:C_2(Y)\surj C_2(Y,\partial Y)$ be the quotient map. Suppose ${q_*:H_2(Y)\to H_2(Y,\partial Y)}$ is trivial and the sutured manifold $(Y, \gamma)$ is taut and horizontally prime. Then either:
    \begin{enumerate}
\item $(Y,\gamma)$ is not reduced,
    \item $\dim P(Y, \gamma)= b_2(Y,\partial Y)$ or
    \item $(Y,\gamma)$ is a product sutured manifold with $Y=\Sigma\times[-1,1]$ with  $\Sigma$ an annulus or a once punctured annulus.
\end{enumerate} 

\end{lemma}

We follow Juh\'asz' proof of the $H_2(Y)=0$ case~\cite[Theorem 3]{juhasz2010sutured}.

\begin{proof}
     By~\cite[Lemma 0.7]{gabai1987foliations}, we have that any non-zero element $\alpha\in H_2(Y,\partial Y)$ there is a groomed surface decomposition $(Y,\gamma)\overset{S}{\rightsquigarrow}(Y',\gamma')$ such that $(Y,\gamma)$ is taut, $[S]=\alpha$ and $S$ is ``open" in the sense that it has a boundary. The remainder of the proof follows Juh\'asz' proof of Theorem~\cite[Theorem 3]{juhasz2010sutured} verbatim, noting that since $i_*$ is trivial the connecting homomorphism $H_2(Y,\partial Y)\to H_1(\partial Y)$ is injective and moreover that the connecting homomorphism $H_2(Y,\gamma)\to H_1(\gamma)$ is injective by naturality.
\end{proof}

We will use this theorem under the hypothesese that $(Y,\gamma)$ is horizontally prime and  $\dim P(Y, \gamma)<b_2(Y,\partial Y)/2$ to produce families of product annuli. We will be interested in applying the Lemma in two cases. The first of these is the case of link exteriors.

\begin{lemma}\label{lem:trivialinclusionlinks}
    Let $L$ be an $n$ component link. The map $q_*:H_2(X(L))\to H_2(X(L),\partial X(L))$ is trivial and $b_2(X(L),\partial X(L))=b_1(\partial X(L))/2$.
\end{lemma}

\begin{proof}
    Suppose $L, n$ are as in the statement of the Lemma. By Alexander duality: \begin{center}
        
$H_2(X(L))\cong \widetilde{H}^{0}(L)\cong\Z^{n-1}$ and $H_3(X(L))\cong \widetilde{H}^{-1}(L)\cong0$.    \end{center} Note that $H_2(\partial X(L))\cong\Z^n$. By Poincar\'e duality $H_3(X(L),\partial X(L))\cong H^0(X(L))\cong\Z$. Applying the long exact sequence of the pair $(X(L),\partial (X(L)))$, we see that $q^*$ must be trivial. The fact that $b_2(M,\partial M)=b_1(\partial M)/2$ follows from Poincar\'e duality and the half lives half dies principle, or a Mayer-Vietoris argument involving $(X(L), \nu(L))$, where $\nu(L)$ is a neighborhood of $L$.
\end{proof}

We now prove a version for complements of longitudinal surfaces in link exteriors.
\begin{lemma}\label{lem:triviallongitudeext}
      Suppose $L$ is a link and $\Sigma$ is a longitudinal surface for a component of $L$ which does not intersect $n$ of the components of $L$. If $(Y,\gamma)$ is the sutured manifolds obtained by decomposing the sutured exterior of $L$ along $\Sigma$ then $q_*:H_2(Y;\Z)\to H_2(Y,\partial Y;\Z)$ is trivial and $b_2(Y,\partial Y)=\dfrac{b_1(\partial Y)}{2}$.
\end{lemma}

\begin{proof}
    Suppose $Y, n$ are as in the statement of the Lemma. By Alexander duality: \begin{center}
        
$H_2(Y) \cong \widetilde{H}^{0}(\Sigma\cup L)\cong\Z^{n}$ and $H_3(Y)\cong \widetilde{H}^{-1}(\Sigma\cup L)\cong0$.    \end{center} Note that $H_2(\partial Y) \cong\Z^{n+1}$. By Poincar\'e duality $H_3(Y,\partial Y)\cong H^0(Y)\cong\Z$. Applying the long exact sequence of the pair $(Y,\partial Y)$, we see that $q^*$ must be trivial. The fact that $b_2(M,\partial M)=b_1(\partial M)/2$ follows from Poincar\'e duality, Alexander duality and the long exact sequence of the pair $(Y,\partial Y)$.
\end{proof}

In order to apply Lemma~\ref{cor:trivialinclusion} sufficiently many times we will need the following Lemma:

\begin{lemma}\label{Lem:canapplycor}
    Suppose $Y$ is a connected component of the complement of a subspace $S$ of $S^3$ consisting of the union of a connected surface $\Sigma$, a link $L$ with $n$ components which do not intersect $\Sigma$, and some collection of annuli $\{A_i\}_{i\in I}$ such that $\partial_\pm A_i\subset R_\pm(\gamma)$ for all $i\in I$, where $R_\pm(\gamma)$ are the two components of $\partial (X(\Sigma),\gamma)$. Suppose that $(Y,\gamma)$ is taut and horizontally prime. Suppose $(Y,\gamma)\overset{A}{\rightsquigarrow}(Y',\gamma')$ is a sutured manifold decomposition with $A$ an annulus. If $q,q'$ are the chain level quotient maps, then if $q_*:H_2(Y;\Z)\to H_2(Y,\partial Y;\Z)$ is trivial then so too is $q_*':H_2(Y';\Z)\to H_2(Y',\partial Y';\Z)$. Moreover, if $b_2(Y,\partial Y)=\dfrac{b_1(\partial Y)}{2}$ then $b_2(Y',\partial Y')=\dfrac{b_1(\partial Y')}{2}$.
\end{lemma}

\begin{proof}

We first deduce some information about different singular homology groups of $Y$. By Alexander duality, we have that both $H_3(Y;\Z)\cong \tilde{H}^{-1}(\nu(S))\cong 0$ and  

\begin{align*}
H_2(Y;\Z)\cong \tilde{H}^{0}(S)\cong \tilde{H}^0(\underset{i\in I}{\bigcup}A_i\cup\Sigma\cup L)\cong \Z^n.\end{align*}

Moreover, by Poincar\'e-Lefshetz duality $H_3(Y,\partial Y;\Z)\cong H^0(Y;\Z)\cong \Z^k$ for some $k$, while\begin{align*}
     H_2(\partial Y;\Z)\cong H^0(\partial Y;\Z)\cong\Z^{m}.
\end{align*} for some $m$. Consider the portion of the long exact sequence of the pair $(Y,\partial Y)$

\begin{center}
\begin{tikzcd}
0\arrow[r]&H_3(Y,\partial Y;\Z)\arrow[r]& H_2(\partial Y;\Z)\arrow[r]&H_2(Y;\Z)\arrow[r,"q_*"]&H_2(Y,\partial Y;\Z)
\end{tikzcd}
\end{center}

By assumption $q_*$ is trivial, so we have that $m=k+n$.

We now move on to study $Y'$. We have the following portion of the long exact sequence of the pair $(Y',\partial Y')$:

\begin{equation}\label{les}
\begin{tikzcd}
H_3(Y';\Z)\arrow[r]&H_3(Y',\partial Y';\Z)\arrow[r]& H_2(\partial Y';\Z)\arrow[r]&H_2(Y';\Z)\arrow[r]&H_2(Y',\partial Y';\Z)
\end{tikzcd}
\end{equation}

Again by Alexander duality we have that $H_3(Y';\Z)\cong \tilde{H}^{-1}(\nu(S\cup A))\cong 0$ and  \begin{align*}
H_2(Y';\Z)\cong \tilde{H}^{0}(S\cup A)\cong \tilde{H}^0(\underset{i\in I}{\bigcup}A_i\cup\Sigma\cup L\cup A)\cong \Z^n.\end{align*}

By Poincar\'e-Lefshetz duality we have that ${H_3(Y',\partial Y';\Z)\cong H^0(Y')}$ and ${H_2(\partial Y)\cong H^0(\partial Y')}$. We can thus rewrite the long exact sequence of Equation~\ref{les} as:

\begin{center}
\begin{tikzcd}
0\arrow[r]&H^0(Y';\Z)\arrow[r,"c"]& H^0(\partial Y;\Z)\arrow[r]&\Z^n \arrow[r,"q_*"]&H_2(Y',\partial Y';\Z)
\end{tikzcd}
\end{center}

We have two cases according to whether $A$ is separating or not. Suppose first that $A$ is separating.  It follows that $H^0(Y';\Z)\cong H^0(Y';\Z)\oplus \Z$. Moreover, $\partial_\pm A$ must also be separating in $\partial Y$, so that $|\partial Y'|=|\partial Y|+1$. The result follows.

Suppose $A$ is non-separating so that $H^0(Y';\Z)\cong H^0(Y;\Z)\cong\Z^k$. We have two cases; either $|\partial Y'|=|\partial Y|=k+n$ or $|\partial Y'|=|\partial Y|+1=k+n+1$. If $|\partial Y'|=|\partial Y|$ then $H^0(\partial Y';\Z)\cong H^0(\partial Y;\Z)\cong \Z^{k+n}$ and we have the desired result.

Suppose then that $|\partial Y'|=|\partial Y|+1$. We seek to obtain a contradiction. First observe that $\partial_\pm A$ must separate each of $R_\pm(\gamma)$. It follows that we can obtain two new surfaces $\Sigma_\pm$ with the same boundary as $\Sigma$ by gluing components of $R_\pm\setminus\partial_\pm A$ to a copy of $A$. Observe that at least one of $\Sigma_\pm$ has $\chi(\Sigma_\pm)\geq\chi(\Sigma)$. Relabeling if necessary we may take this surface to be $\Sigma_+$. Since $(Y,\gamma)$ is taut we have that $\chi(\Sigma_+)=\chi(\Sigma)$. It follows that $\Sigma_+$ is isotopic to $\Sigma$, since $(Y,\gamma)$ is horizontally prime. It follows that $A$ is separating, a contradiction.

To verify the condition on the Betti numbers, consider the following segment of the long exact sequence of the pair $(Y',\partial Y')$ in singular homology:

\begin{center}
\begin{tikzcd}
&H_2(Y';\Q)\arrow[r,"0"]&H_2(Y',\partial Y';\Q)\arrow[d]\\

H_1(Y',\partial Y';\Q)\arrow[d]&H_1(Y';\Q)\arrow[l]&\arrow[l] H_1(\partial Y';\Q)\\

H_0(\partial Y';\Q)\arrow[r]&H_0(Y';\Q)\arrow[r]&H_0(Y',\partial Y';\Q)\arrow[d]\\&&0
\end{tikzcd}
\end{center} 

Note that $b_2(Y',\partial Y')=b^1(Y')$ by Poincar\'e-Lefschetz duality. By the universal coefficient theorem $b^1(Y')=b_1(Y')$. By Poincar\'e-Lefschetz duality, $H_1(Y',\partial Y';\Q)\cong H^2(Y';\Q)$, while by the universal coefficeint theorem $H^2(Y';\Q)\cong H_2(Y';\Q)$. Also by Alexander duality $H_2(Y';\Q)\cong \tilde{H}^0(\underset{i\in I}{\bigcup}A_i\cup\Sigma\cup L\cup A;\Q)\cong \Q^n$.

Suppose $A$ is separating then $b_0(\partial Y')=b_0(\partial Y)+1=k+n+1$ and $b_0(Y',\partial Y')=0,b_0(Y')=k+1$. The result follows from the segment of the exact sequence given above.

Suppose $A$ is non-separating. As above we have that $|\partial Y|=|\partial Y'|$. Then $b_0(Y',\partial Y')=0,b_0(Y')=k$ and $b_0(\partial Y')=n+k$. We thus obtain the following exact sequence

\begin{center}
\begin{tikzcd}
&0\arrow[r]&H_2(Y',\partial Y';\Q)\arrow[d]\\

\Q^{n}\arrow[d]&H_2(Y',\partial Y';\Q)\arrow[l,"a"]&\arrow[l] H_1(\partial Y';\Q)\\

\Q^{n+k}\arrow[r]&\Q^k\arrow[r]&0
\end{tikzcd}
\end{center}

It follows that $a$ vanishes so that $b_1(\partial Y')=2b_1(Y',\partial Y')$.

\end{proof}

We will find it convenient to make use of the following definition:

\begin{definition}
   Suppose $(Y,\gamma)$ is balanced, taut and horizontally prime. A \emph{minimal family of essential product annuli} is a collection of annuli $\{A_i\}_{i\in I}$ such that decomposing $(Y_T,\gamma_T)$ along $\underset{i\in I}{\bigcup }A_i$ yields a sutured manifold $(Y,\gamma)$ each component $(Y_j,\gamma_j)$ of which  satisfies either:
    \begin{enumerate}
    \item $\dim P(Y_j, \gamma_j)= b_1(\partial Y_j)/2$ or
    \item $(Y_j,\gamma_j)$ is a product sutured manifold.

    \end{enumerate}

    and $|I|$ is minimal over all such families $\{A_i\}$.
\end{definition}

Note that if $\{A_i\}$ is a minimal family then each product sutured manifold obtained in the decomposition of $(Y,\gamma)$ cannot have base an annulus or disk. In particular they have strictly negative Euler characteristic.

\subsection{Link Floer Homology}
Link Floer homology is an invariant of links in $S^3$ due to Ozsv\'ath-Szab\'o~\cite{ozsvath2008holomorphic}. Link Floer homology can be be defined as the sutured Floer homology of a link exterior decorated with pairs of parallel oppositely oriented meridional sutures on each boundary component. It can be equipped with an \emph{Alexander grading} for each component, which can be thought of as taking value in $\frac{1}{2}\Z$, by evaluating the first Chern class of relative $\spin^c$ structures on $X(L)$ on Seifert surfaces for each of the components. It can then be equipped with an additional grading called the \emph{Maslov grading}. Different Maslov grading conventions exist in the literature. We use the convention that the unlink has link Floer homology with maximal non-trivial Maslov grading $0$, so that the Maslov grading is always $\Z$-valued.

Ozsv\'ath-Szab\'o showed that link Floer homology detects the Thurston norm~\cite{ozsvath2008linkFloerThurstonnorm}. Ni showed that link Floer homology detects if a link is fibered~\cite[Propositon 2.2]{ni2006note}. Martin showed that link Floer homology detects braid axes~\cite[Proposition 1]{martin2022khovanov}.

\section{Tangles, Braids, Clasp-braids and Sutured Manifolds}\label{sec:almostbraids}

In this section we review the definition of nearly fibered knots and discuss various notions of braidedness and near-braidedness that will play an important role in Section~\ref{sec:mainthmpf}.

For context, recall that Juh\'asz proved that a sutured manifold $(Y,\gamma)$ is a product if and only if $\rank(\SFH(Y,\gamma))=1$~\cite{juhasz2008floer}. It is natural to ask if there is a similar characterization of sutured manifolds satisfying the following the definition:

\begin{definition}
An \textit{almost product sutured manifold} is a sutured manifold $(Y,\gamma)$ such that $\rank(\SFH(Y,\gamma))=2$.
\end{definition}

There is currently no complete answer to this question. However, applying~\cite[Theorem 3]{juhasz2010sutured}, Baldwin-Sivek~\cite{baldwin2022floer} showed that almost product sutured manifolds that embed in $S^3$ admit a decomposition along essential product annuli into the following pieces:\begin{enumerate}
     \item product sutured manifolds
    \item exactly one of the following three pieces, up to mirroring:\begin{enumerate}
        \item A solid torus with $4$-parallel longitudinal sutures. Call this $T_4$.
        \item A solid torus with parallel oppositely oriented sutures of slope $2$.
\item The exterior of a right handed trefoil with two parallel oppositely oriented sutures of slope $2$.
\end{enumerate}
\end{enumerate}

Here the slope of a suture is measured with respect to the Seifert longitude. Baldwin-Sivek were interested in almost product sutured manifolds because they arise naturally in the study of knots that satisfy the following definition:

\begin{definition}[Baldwin-Sivek~\cite{baldwin2022floer}]
A null-homologous knot $K$ in a $3$-manifold $Y$ is called \textit{nearly fibered} if the link Floer homology of $K$ is rank $2$ in the maximal non-trivial Alexander grading.
\end{definition}

An alternative characterization of nearly fibered links can be given as follows. A null-homologous knot $K$ in a $3$-manifold $Y$ is nearly fibered if there exists a Seifert surface $\Sigma$ for $K$ such that the sutured manifold obtained by decomposing $(Y_K,\gamma_K)$ -- the exterior of $Y$ equipped with parallel meridional sutures -- along $\Sigma$ is an almost product sutured manifold.

\subsection{Tangles in Sutured Manifolds}

We now give an interpretation of the different objects referred to in Theorem~\ref{thm:mainbraid} in terms of tangles in sutured manifolds.

\begin{definition}
    Let $(Y,\gamma)$ be a sutured manifold. A \textit{tangle} is a properly embedded one dimensional sub-manifold $T\subset Y$ such that $\partial_\pm T\subset R_\pm(\gamma)$.
\end{definition}

Note that we do not require $T$ to be connected. We will be interested in studying tangles up to isotopies which preserve the condition that $\partial_\pm T\subset R_\pm(\gamma)$. In particular such isotopies are not required to fix $\partial_\pm T$. Given a tangle $T$ in a sutured manifold $(Y,\gamma)$, by a mild abuse of notation, we will call an annulus $A$ properly embedded in $(Y,\gamma)$ a \emph{product annulus} if $A\cap T=\emptyset$, $\partial_\pm A\subset R_\pm(\gamma)$, \emph{incompressible} if for any disk $(D,\partial D)\inj(Y,A)$, $(D,\partial D)$ can be be isotoped into $A$ through an isotopy of embeddings $(D_t,\partial D_t)\inj(Y,A)$ each missing $T$ and \emph{essential} if it is incompressible and if $A$ cannot be isotoped into a suture $\gamma$ through a family of annuli product annuli $A_t$ with $A_t\cap T=\emptyset$ for all $t$.

If $T$ is a tangle in a sutured manifold $(Y,\gamma)$ with a component $t$ such that $\partial t\cap R_\pm(\gamma)\neq\emptyset$ then a new tangle $T'$ can be formed by adding a new component $t'$ which runs parallel to $t$. We call this operation \textit{tangle stabilization} and the process of undoing a tangle stabilization \emph{tangle destabilization}.

In product sutured manifolds there are a particularly simple class of tangles:

\begin{definition}
    Let $(Y,\gamma)$ be a product sutured manifold with base $\Sigma$. A tangle $T$ is a \textit{braid} if it is isotopic to $\{p_i\}_{1\leq i\leq n}\times[-1,1]$ for some $n$ through an isotopy which preserves $\partial(\Sigma\times[0,1])$ as a set.
\end{definition}

Observe that any braid in a fixed product sutured manifold can be obtained from any other by a sequence of tangle stabilizations or tangle destabilizations. We extend the definition of a braid to a setting in which the underlying $3$-manifold is not necessarily a product sutured manifold.

\begin{definition}
    Let $(Y,\gamma)$ be a sutured manifold. A tangle $(T,\partial_\pm T)\subset(Y,R_\pm(\gamma))$ is a \textit{braid} if there exists a family of essential product annuli in $(Y,\gamma)$, $A_i\inj Y$, such that decomposing $(Y,\gamma)$ along $\{A_i\}$ yields a sutured manifold with a product sutured manifold component $(P,\rho)$ such that $T\subset P$ and $T$ is a braid in $(P,\rho)$.
\end{definition}

 Note that our definition generalizes the classical case of braids in a product sutured manifolds. We will be interested in studying the ``closures" of such tangles. For that we need the following preliminary definition.

\begin{definition}
    Let $L$ be a link with a component $K$. A \textit{longitudinal surface for $K$} is a surface $\Sigma$ with boundary $K$ that intersects $L-K$ transversely at a finite number of points.
\end{definition}

We can now define the notion of braid closure:

\begin{definition}
    Let $L$ be a link with a component $K$. we say that $L-K$ is \textit{braided with respect to $K$} if there   is a longitudinal surface $\Sigma$ whose image in the exterior of $L$ has maximal Euler characteristic amongst representatives of $[\Sigma]\in H_2(X(L),\partial X(L))$ such that the image of $L$ in the sutured manifold $(Y_\Sigma,\gamma_\Sigma)$ obtained by decomposing the exterior of $K$ along $\Sigma$ is a braid.
\end{definition}

This generalizes the usual notion of a braidedness in which case $K$ is required to be fibered. There are two more types of tangle referred to in the statement of Theorem~\ref{thm:mainbraid}:

\begin{definition}
    Let $T$ be a tangle in a sutured manifold $(Y,\gamma)$. We say that $T$ is a \textit{clasp-braid} if $(Y,\gamma)$ can be decomposed along an essential annulus $A$ in $(Y,\gamma)$ into a product sutured manifold $(Y_1,\gamma_1)\sqcup (Y_2,\gamma_2)$ with $T\cap Y_1$ a braid (perhaps with no strands) in $(Y_1,\gamma_1)$ and $T\cap Y_2$ a \emph{clasp}, the tangle in the product sutured manifold with base a disk, shown in Figure~\ref{fig:nearlybraided}. 
\end{definition}

 The following is a notion of ``closure" for such tangles that we will consider:

\begin{definition}
    Let $L$ be a link and $K$ a knot. We say that $L$ is \textit{a clasp-braid closure with respect to $K$} if the image of $L$ in the sutured manifold $(Y,\gamma)$ obtained by decomposing the exterior of $K$ along some longitudinal surface is a clasp-braid.
\end{definition}

There is one more type of tangle we will be interested in.
\begin{definition}
 Let $T$ be a tangle in a sutured manifold $(Y,\gamma)$. We say that $T$ is \textit{a stabilizable clasp-braid} if $(Y,\gamma)$ can be decomposed along an essential annulus into a sutured manifold $(Y_1,\gamma_1)\sqcup (Y_2,\gamma_2)$ with $T\cap Y_1$ a braid (perhaps with no strands) in $(Y_1,\gamma_1)$ and $T\cap Y_2$ a \textit{stabilizable clasp}, the tangle shown in Figure~\ref{stabilizable cusp}. A \emph{stabilized clasp-braid} is a stabilizable clasp-braid or a tangle stabilization thereof. We call the underlying sutured manifold $(Y,\gamma)$ a \emph{stabilized product sutured manifold}.
\end{definition}
 See Figure~\ref{fig:3gluing} for an example of an example of an stabilized stabilizable clasp-braid. Observe that stabilized product sutured manifolds are not taut.

\begin{center}
 
     \begin{figure}[h]
 \includegraphics[width=5.3cm]{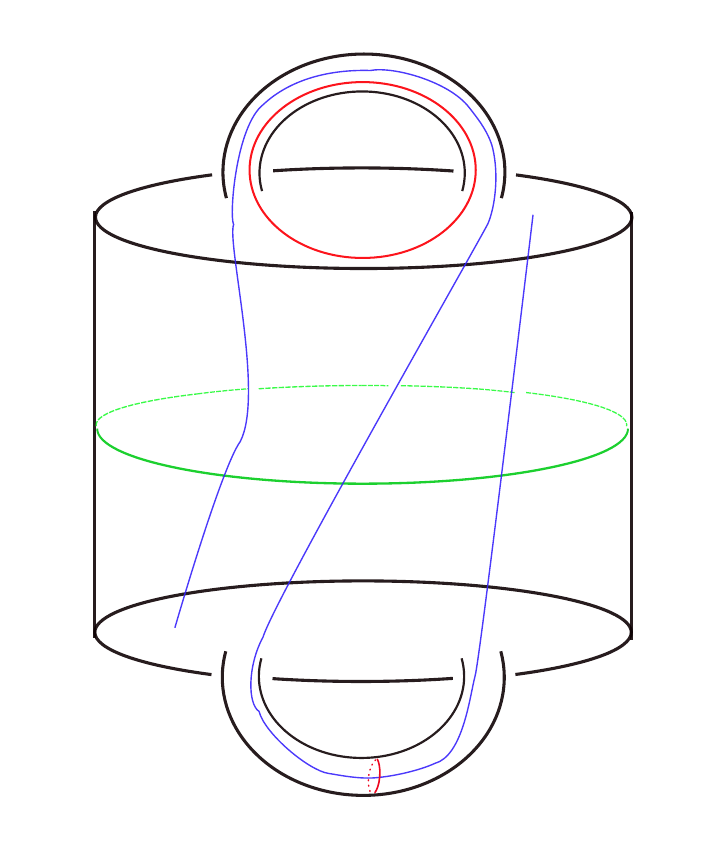}
    \caption{A stabilizable clasp. The green curve indicates the core of the suture, while the blue curves indicate the tangle $T$. The red curves indicate the core and co-core of the lower and upper stabilizing handles respectively.}\label{stabilizable cusp}
   \end{figure}
\end{center}

\begin{definition}
    Let $L$ be a link and $K$ a knot. We say that $L$ is \textit{a stabilized clasp-braid closure with respect to $K$} if $L$ can be obtained from a stabilized clasp-braid $T$ in a stabilized product sutured manifold $(Y,\gamma)$ by gluing $R_+(\gamma)$ to $R_-(\gamma)$ by a diffeomorphism $\phi$ which maps the core of the 2 dimensional stabilized handle in $R_+(\gamma)$ to the co-core of the other stabilized handle of $R_-(\gamma)$ and the co-core to the core, as unoriented curves. We moreover require that $K$ is isotopic to the image of $s(\gamma)$.
\end{definition}

Note in particular that if $L$ is a stabilized clasp-braid with respect to $K$ then $K$ is fibered.

\begin{center}
 
     \begin{figure}[h]
 \includegraphics[width=5cm]{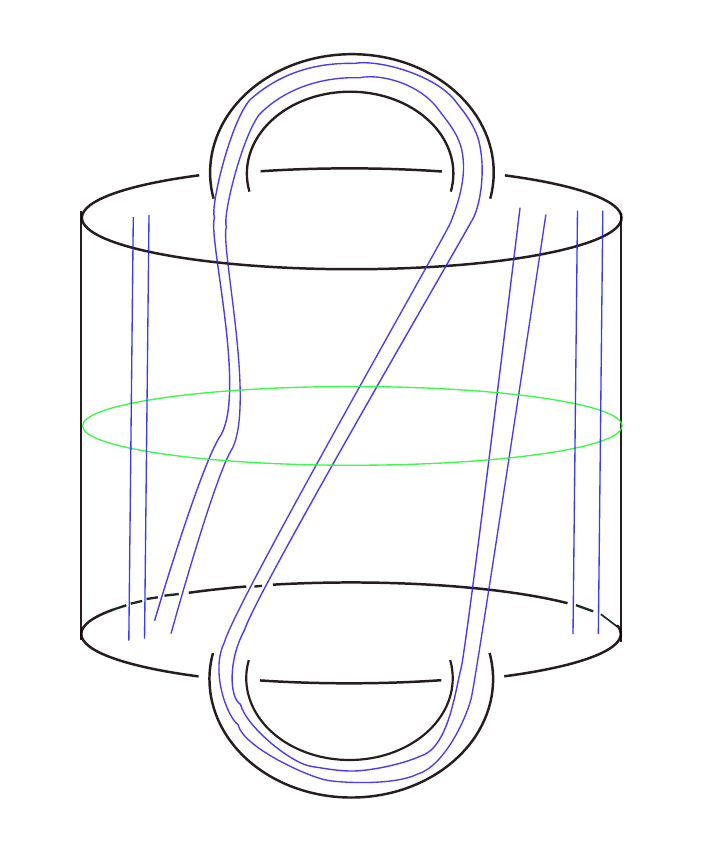}
    \caption{A stabilized clasp-braid in a stabilized sutured manifold $(Y,\gamma)$. The green curve indicates the suture.}\label{fig:3gluing}
   \end{figure}
\end{center}

We also have the following natural generalization of braid index:

\begin{definition}
    The \emph{index} of a link $L$ in the complement of a fibered knot $K$ is the minimal geometric intersection number of $L$ with a fiber surface for $K$.
\end{definition}

Observe that clasp-braid closures of index at least three are also stabilized clasp-braid closures, but not every stabilized clasp-braid closure is a clasp-braid closure. There are no stabilized clasp-braid closures of index $2$. There are no clasp-braid closures or stabilized clasp-braid closures of index $1$.

\subsection{Sutured Tangle Exteriors}\label{subsec:suturedtangleext}

We now give a description of the objects referred to in the main theorem, Theorem~\ref{thm:mainbraid} purely in terms of sutured manifolds without reference to tangles. In Section~\ref{sec:mainthmpf} we will prove Theorem~\ref{thm:mainbraid} by first working in this setting then passing to the setting of tangles in sutured manifolds.
\begin{definition}\label{sutured tangles}
   
Let $T$ be a tangle in a sutured manifold $(Y,\gamma)$. We say that $(Y_T,\gamma_T)$ is the \textit{sutured exterior for $T$} if $Y_T$ is $Y-\nu(T)$ and $\gamma_T$ consists of:\begin{itemize}
    \item $\gamma$.
    \item Pairs of parallel oppositely oriented meridional sutures on each closed component of $T$.
    \item A single meridian for each component $t$ of $T$ such that $\partial t$ has components in $R_+(\gamma)$ and $R_-(\gamma)$.
    \item A pair of parallel oppositely oriented meridians for every component $t$ of $T$ such that $\partial t$ has components in both of the components in $R_-(\gamma)$.
\end{itemize}

\end{definition}

\begin{remark}\label{rem:tanglerecovery}

Given $(Y_T,\gamma_T)$ and $(Y,\gamma)$ we wish to be able to determine $T$. To that end, let $g$ be a suture in $\gamma_T-\gamma$, which is not parallel to another suture of $\gamma_T-\gamma$. Glue in the sutured manifold $(\{z\in\C:|z|\leq 1\}\times[-1,1],\{z\in\C:|z|= 1\}\times[-1,1])$ with the tangle given by $\{(0,t):t\in[-1,1]\}$ in such a way that $(\{z\in\C:|z|=1 \}\times\{0\}$ is identified with $g$.

For every pair of parallel $g_1,g_2$ that are on a toroidal boundary component of $\partial Y_T$ not containing $\gamma$, glue in a solid torus $D^2\times S^1$ with tangle $\{0\}\times S^1$, in such a way that $g_1$ bounds a disk in $D^2\times S^1$.

Likewise, for every pair of parallel $g_1,g_2\not\in\gamma_T-\gamma$ that are not on a toroidal boundary component of $\partial Y_T$ remove $g_2$ then glue in the sutured manifold $(\{z\in\C:|z|\leq 1\}\times[-1,1],\{z\in\C:|z|= 1\}\times[-1,1])$ with the tangle given by $\{(0,t):t\in[-1,1]\}$ in such a way that $(\{z\in\C:|z|=1\}\times\{0\}$ is identified with $g_1$.

\end{remark}

The question remains as to how to recover the components of $T$ that do not give rise to sutures in $(Y_T,\gamma_T)$. We do not attempt this in general, only for the sutured exteriors of clasp-braids. We defer this until Lemma~\ref{lem:suturedclaspext}.

\subsection{The Sutured Floer Homology of Sutured Tangle Exteriors}
For Section~\ref{sec:linkext} we will require two Lemmas, slightly weaker versions of which are discussed in work of Li-Xie-Zhang~\cite[Section 1]{li2022floer}. If $(Y,\gamma)$ is a sutured manifold, $K$ is a knot in $Y$, and $\gamma_K=\gamma\cup\sigma$, where $\sigma$ is a pair of parallel curves on $\partial(\nu(K))\subset \partial Y_K$. Let $(Y_K(\sigma),\gamma)$ denote the sutured manifold obtained by capping off a component of $\sigma$ with a disk and filling in the resulting spherical boundary component with a $3$ ball. Let $V$ denote a rank $2$ vector space supported in a single affine grading. 
\begin{lemma}\label{lem:spectral}
    Suppose $(Y,\gamma)$ is a sutured manifold and $K$ is a knot in $Y$.  Then there is a spectral sequence from $\SFH(Y_K,\gamma_K)$ to $\SFH(Y_K(\sigma),\gamma)\otimes V$. This spectral sequence sends every relative $H_1(Y_K)$ grading to its image under the quotient map $H_1(Y_K)\surj H_1(Y_K)/[\mu_K]$.
\end{lemma}
Here $\mu_K$ is either component of $\sigma$, which are of course both homologous.

\begin{proof}
    Consider an admissible sutured Heegaard diagram for $(Y_K,\gamma_K)$, $(\Sigma,\alpha,\beta)$. Observe that capping off the boundary components of $\Sigma$ corresponding to $\sigma$ yields an admissible sutured Heegaard diagram for $(Y,\gamma)$ with an additional puncture. The sutured Floer homology of this manifold is exactly $\SFH(Y,\gamma)\otimes V$. The desired result follows.
\end{proof}

In the case that $Y$ is an integer homology sphere and the sutures are meridional we have a stronger statement. To state it, recall that in this case the $\spin^c$ grading on Sutured Floer homology can be viewed as a $\Q^n$ grading, where the $i$th $\Q$ grading -- which we will denote by $A_i$ -- is given by taking $\frac{\langle c_1(\s),[\Sigma_i]\rangle}{2}$, where $\Sigma_i$ a Surface Poincare dual to the $i$th component of $L$.

\begin{corollary}
    Let $Y$ be an integer homology sphere. Suppose $L$ is a link in $Y$ with a component $K$. Consider the sutured manifolds $(Y_L,\gamma_L)$ obtained by removing a neighborhood of $L$ from $Y$ and endowing it with sutures. Let $K$ be a component of $L$ with a pair of meridional sutures. Then there is a spectral sequence from $\SFH(Y_L,\gamma_L)$ that converges to ${\SFH((Y(\gamma_K))_{L-K},\gamma_{L-K})[\frac{\lk(K,L_i)}{2}]\otimes V}$.
\end{corollary}
Here $[\frac{\lk(K,L_i)}{2}]$ indicates a shift in each $A_i$ grading by $\frac{\lk(K,L_i)}{2}$. For the proof it is helpful to use an equivalent way of thinking about relative $\spin^c$ structures, as given in~\cite[Section 3.2]{ozsvath2008holomorphic}. Specifically, we can consider equivalence classes of non-vanishing vector fields which are standard on the boundary tori -- i.e. vector fields on $Y$ that have closed orbits given given by components of $L$, see~\cite[Subsection 3.6]{ozsvath2008holomorphic} for details.

\begin{proof}
From the proof of Lemma~\ref{lem:spectral} we see that it suffices to show that the Alexander grading shifts as stated. This follows exactly as in the case of links in $S^3$, as discussed in Ozsv\'ath-Szab\'o~\cite[Proposition 7.1, Subsection 3.7, Subsection 8.1]{HolomorphicdiskslinkinvariantsandthemultivariableAlexanderpolynomial}.

 To expand on this, recall from~\cite[Section 3.7]{ozsvath2008holomorphic} that given $K$ a component of a link $L$ there is a map \begin{align*}
     G_{K}:\spin^c(Y,L)\to\spin^c(Y,L-K)
 \end{align*}
 obtained by viewing relative $\spin^c$ structures on $Y$ relative to $L$ as $\spin^c$ structures on $Y$ relative to $L-K$. By~\cite[Lemma 3.13]{ozsvath2008holomorphic}, $c_1(G_K(\s))=c_1(\s)+\PD([K_1])$. It follows in turn that if $\Sigma$ is a Seifert surface for a component $K_i$ of $L$ then \begin{align*}\langle c_1(G_K(\s)),[\Sigma_i]\rangle=\langle c_1(\s),[\Sigma_i]\rangle+\langle\PD([K_1]),[\Sigma_i]\rangle,\end{align*}so that \begin{align*}\langle c_1(G_K(\s)),[\Sigma_i]\rangle=\langle c_1(\s),[\Sigma_i]\rangle+\lk(K,L_i).\end{align*}

 The result follows.
\end{proof}

In Section and~\ref{sec:linkext} and Section~\ref{sec:mainthmpf} our main tool will be the following Proposition, which will allow us to decompose the exterior of a longitudinal surface for a component $K$ of a link $L$ into smaller pieces.

\begin{proposition}\label{prop:allproductannuli}
     Suppose $L$ is a link and $\Sigma$ is a longitudinal surface for some component $K$ of $L$. Let $(Y_T,\gamma_T)$ be the sutured manifold obtained by decomposing $(Y_L,\gamma_L)$ along $\Sigma$. Suppose $(Y_T,\gamma_T)$ is balanced, taut and horizontally prime. Then there exists a finite collection of essential product annuli $\{A_i\}_{i\in I}$ such that decomposing $(Y_T,\gamma_T)$ along $\underset{i\in I}{\bigcup }A_i$ yields a sutured manifold $\bigsqcup (Y_i,\gamma_i)$, each component of which satisfies either:
    \begin{enumerate}
    \item $\dim P(Y_i, \gamma_i)= b_1(\partial Y_i)/2$ or
    \item $(Y_i,\gamma_i)$ is a product sutured manifold with base an annulus or a once punctured annulus.
    \end{enumerate}
\end{proposition}
\begin{proof}
   By a result of Juh\'asz~\cite[Proposition 2.16]{juhasz2010sutured}, there is a finite collection of disjoint incompressible product annuli $\{A_i\}_{i\in I}$ such that decomposing $(Y_T,\gamma_T)$ along $\underset{i\in I}{\bigcup }A_i$ yields a reduced sutured manifold. The result then follows from Lemma~\ref{cor:trivialinclusion}, which is applicable by Lemma~\ref{Lem:canapplycor} and Lemma~\ref{lem:triviallongitudeext}.
\end{proof}

\section{Link exteriors with Sutured Floer homology of low rank}\label{sec:linkext}

Let $L$ be a link in $S^3$. As in the previous section, we let $S^3_L$ denote the exterior of $L$. For this section we let $\gamma_L$ be some collection of sutures on $\partial S^3_L$, at least two for each component of $\partial S^3_L$, such that all but at most one component of $\partial S^3_L$ has pairs of parallel oppositely oriented meridians as sutures. We will let $\gamma_i$ denote the sutures corresponding to the $i$th component of $L$.

In this section we classify sutured manifolds $(S^3_L,\gamma_L)$ with $\rank(\widehat{\SFH}(S^3_L,\gamma_L))\leq 2^n$. This is a generalization of work of Kim, which restricted to the setting that all of the sutures were pairs of parallel oppositely oriented meridians~\cite[Theorem 1]{kim2020links}. Kim's result is in turn a generalization of work of Ni, who classified links in $S^3$ with link Floer homology or rank at most $2^{n-1}$~\cite[Proposition 1.4]{ni2014homological}. Our result is also a generalization of work of Baldwin-Sivek, who implicitly gave the $n=1$ case of this classification in~\cite[Theorem 5.1]{baldwin2022floer}.

Note that sutured Floer homology carries $n$ Alexander gradings and a Maslov grading in our setting, defined exactly as in the link Floer homology case.

\begin{theorem}\label{thm:mailinknrankclass}
    Suppose $L$ is an $n$-component link in $S^3$ with components $K_1,K_2,\dots K_n$. If $\rank(\SFH(S^3_L,\gamma_L))\leq 2^n$ then, up to relabeling and mirroring, we have:

    \begin{enumerate}
        \item $L$ is the split sum of an $n-1$ component unlink, where each $\gamma_i$ are meridians for $1\leq i\leq n-1$ and $L_n$ is one of the following:

        \begin{enumerate}
           \item An unknot exterior with two parallel sutures of slope $\infty$ or $2$.
         \item A right-hand trefoil exterior with two parallel sutures of slope $2$.
         \item An unknot exterior with four parallel sutures of slope $\infty$.
        \end{enumerate}

    \item $L$ is a Hopf link split sum an unlink, with $\gamma_{i}$ a pair of meridians for all $i$.
    \item $L$ has a split unknotted component with sutures of slope $0$.
    \item $L$ is the split sum of an unlink with meridional sutures and a Hopf link where one component has meridional sutures and the other has two slope $0$ sutures.

        \end{enumerate}
\end{theorem}

Here the slopes are measured with respect to the Seifert longitude of each component. We will use Theorem~\ref{thm:mailinknrankclass} in Section~\ref{sec:mainthmpf}.

\begin{remark}\label{rem:SFHcomp}

We describe the sutured Floer homology of the sutured manifolds in the statement of Theorem~\ref{thm:mailinknrankclass}. These computations follow from the connect sum formula for sutured Floer homology, Equation~\ref{eq:connectsum}, the sutured decomposition formula and sutured Thurston norm detection. To this end let $V$ be a rank $2$ vector space supported in Alexander multi-grading $0$.

\begin{enumerate}
    \item Suppose $L$ is the split sum of an $n-1$ component unlink and a knot $K$, which after relabeling we take to be the first component of $L$.\begin{enumerate}
        \item If $K$ is an unknot with two parallel sutures of slope infinity then $\SFH(Y_L,\gamma_L)$ is given by
    $ V^{\otimes(n-1)}$.
    \item If $K$ is an unknot with two parallel sutures of slope two then $\SFH(Y_L,\gamma_L)$ is given by
    \begin{align*} V^{\otimes(n-1)}\otimes(\F[1/2,0,0\dots 0]\oplus\F[-1/2,0,0\dots 0])\end{align*}

    where here, and for the rest of this remark, $\F^k[x_1,x_2,\dots x_n]$ indicates an $\F^k$ summand of multi-Alexander grading $[x_1,x_2,\dots x_n]$.
     \item If $K$ is $T(2,3)$ with two parallel sutures of slope two then $\SFH(Y_L,\gamma_L)$ is given by
    \begin{align*} V^{\otimes(n-1)}\otimes(\F[3/2,0,0\dots 0]\oplus\F[-3/2,0,0\dots 0])\end{align*}
    \item If $K$ is an unknot with four parallel sutures of slope infinity then $\SFH(Y_L,\gamma_L)$ is given by
    \begin{align*} V^{\otimes(n-1)}\otimes(\F[1/2,0,0\dots 0]\oplus\F[-1/2,0,0\dots 0])\end{align*}
    \end{enumerate}  

\item If $L$ is a Hopf link split sum an unlink then $\SFH(S^3_L,\gamma_L)$ is given by:\begin{align*}
    V^{\otimes(n-2)}\otimes (\F[1/2,1/2,0,\dots 0]\oplus \F[1/2,-1/2,0,\dots 0]\oplus \F[-1/2,1/2,0,\dots 0] \\ \oplus \F[-1/2,-1/2,0,\dots 0])
\end{align*}

where we label the components of $L$ such that its first two components are the Hopf link.

\item If $L$ has a split unknotted component with sutures of slope $0$ then $(S^3_L,\gamma_L)$ is not taut and $\rank(\SFH(S^3_L,\gamma_L))=0$.

\item If $L$ is a Hopf link where one component has meridional sutures and the other has a pair of slope $0$ sutures, then $\SFH(S^3_L,\gamma_L)$ is given by $V^{\otimes n}$.

\end{enumerate}

\end{remark}
We now turn to the proof of Theorem~\ref{thm:mailinknrankclass}. We begin with the case that $R(\gamma_L)$ is compressible.

\begin{lemma}\label{lem:nottaut}
     Suppose $L$ is a link in $S^3$. If $R(\gamma_L)$ is compressible then $L$ has a split unknotted component $K$ with longitudinal sutures.
\end{lemma}

\begin{proof}
   Suppose $L$ is as in the statement of the Lemma. Any compressing disk must be a Seifert surface for a component $K$ of $L$. The $\gamma_K$ must consist of longitudinal sutures by definition.
\end{proof}

The following serves as the base case of Theorem~\ref{thm:mailinknrankclass}:

\begin{proposition}\label{BaldwinSivekrank2knotexteriors}
    Suppose $K$ is a knot in $S^3$ with $\rank(\SFH(Y_K,\gamma_K))\leq 2$. Then up to mirroring we have that:

    \begin{enumerate}
    
         \item $K$ is an unknot and $\gamma_k$ consists of one of:\begin{enumerate}
              \item parallel sutures of slope $0$.
             \item two parallel sutures of slope $\infty$.
             \item two parallel sutures of slope $2$.
             \item four parallel sutures of slope $\infty$.
         \end{enumerate}
         \item $K$ is a right handed trefoil and and $\gamma_k$ consists of two parallel sutures of slope $2$.
    \end{enumerate}
\end{proposition}

\begin{proof}
    We have cases according to the value of $\rank(\SFH(Y_K,\gamma_K))$. The ${\rank(\SFH(Y_K,\gamma_K))=2}$ case follows from work of Li-Ye in \cite[Theorem 1.4]{li2022seifert} (and the remark right before), or Baldwin-Sivek, see the proof of Theorem 5.1 in~\cite{baldwin2022floer}. The $\rank(\SFH(Y_K,\gamma_K))=1$ case follows from Juh\'asz product sutured manifold detection result~\cite[Theorem 1.4, Theorem 9.7]{juhasz2008floer}.
    
    Finally, if $\rank(\SFH(Y_K,\gamma_K))=0$ then $(Y_K,\gamma_K)$ is not taut, by~\cite[Theorem 1.4]{juhasz2008floer}. Since $R(\gamma)$ is Thurston norm minimizing, it follows that $R(\gamma_K)$ is compressible, so that $K$ is the unknot and $\gamma_K$ is some collection of slope $0$ sutures. 
\end{proof}
We need two more Lemmas before we can proceed to the next case of Theorem~\ref{thm:mailinknrankclass}.

\begin{lemma}\label{lem:nondegeneratenorm}
    Suppose $(Y_L,\gamma_L)$ is taut. Let $K$ be a component of $L$. Then either:
    \begin{enumerate}
        \item The $A_K$ span of $\SFH(Y_L,\gamma_L)$ is non-zero,
        \item $K$ is the unknot and $\gamma_K$ consists of two meridians, or,
        \item $L-K$ has a component $K'$ such that $K$ is a meridian of $K'$ and each component of $\gamma_K$ has slope $0$.
    \end{enumerate}
\end{lemma}

Since we are working in a setting in which $\chi(R_\pm)$ are $0$, the property of being taut is equivalent to being irreducible and incompressible which is in turn equivalent to $L$ being non-split and not having any unknotted components with slope zero sutures.

\begin{proof}
Fix a component $K$ of $L$. Let $\Sigma$ be a Thurston norm minimizing surface Poincar\'e dual to $\mu_k$ -- a meridian of $K$. Observe that unless $(|\Sigma\cap s(\gamma_L)|,\chi(\Sigma))\in\{(0,1),(0,0),(2,1)\}$ we have that $x^s(\Sigma)=\max\{0,\frac{1}{2}|\Sigma\cap s(\gamma_L)|-\chi(\Sigma)\}>0$. If $(|\Sigma\cap s(\gamma_L)|,\chi(\Sigma))=(0,1)$ then $R(\gamma_L)$ is compressible, contradicting tautness. If $(|\Sigma\cap s(\gamma_L)|,\chi(\Sigma))=(0,0)$ then $\Sigma$ is an annulus and we see that $L$ has a component $K'$ such that $K$ is a meridian of $K'$ and $\gamma_K$ consists of longitudinal sutures. If $(|\Sigma\cap s(\gamma_L)|,\chi(\Sigma))=(2,1)$ then we are in case two of the statement of the Lemma.

In all other cases $x^s(-)$ is non-trivial in the direction Poincar\'e dual to $\mu_K$. It follows from Theorem~\ref{thm:FLRpolytopedetection} that the sutured Floer homology norm is non-degenrate in the direction Poincar\'e dual to $\mu_K$ -- i.e. we are in case $1$ of the statement of the Lemma.
\end{proof}

\begin{lemma}\label{lem:meridian}
    Suppose an $n$ component link $L$ has a component $K$ with a pair of meridional sutures and another component $K'$ which is a meridian of $K$ with longitudinal sutures. If $\rank(\SFH(Y_L,\gamma_L))\leq 2^{n}$. Then $L-K'$ is an unlink.
\end{lemma}

\begin{proof}
    Suppose $L,n$ are as in the statement of the Lemma. Decompose $(Y_L,\gamma_L)$ along the annulus cobounded by a longitude of $K'$ and a meridian of $K$. The resulting sutured manifold, $(Y,\gamma)$, has at least $6$ parallel sutures on one of its boundary components. Moreover, $Y\cong Y_{L'}$ where $L'=L-K'$. Removing these excess parallel sutures we obtain a sutured manifold of rank at most $2^{n-2}$. Note that $L'$ has only pairs of parallel meridional sutures. It follows that $L'$ is an unlink by~\cite[Proposition 1.4]{ni2014homological}.
\end{proof}

\begin{proposition}\label{prop:2componentrank}
     Suppose $L$ is a two component link and $\rank(\SFH(S^3_L,\gamma_L))\leq 4$. Then up to mirroring $(Y_L,\gamma_L)$ are given as follows:

\begin{enumerate}
    \item A Hopf link with meridional sutures.

    \item The split sum of a sutured manifold from Proposition~\ref{BaldwinSivekrank2knotexteriors} and an unknot exterior with meridional sutures.
    \item $L$ has a split unknotted component with sutures of slope $0$.
    \item $L$ is a Hopf link where one component has meridional sutures and the other has two pairs of slope $0$ sutures.
\end{enumerate}
    
\end{proposition}

\begin{proof}
    Suppose $L$ is as in the statement of the proposition. If $(S^3,\gamma_L)$ is irreducible but not taut, the result follows from Lemma~\ref{lem:nottaut}. We may thus assume that $(Y_L,\gamma_L)$ is taut so in turn that $\rank(\SFH(Y_L,\gamma_L))>0$.
    
    Suppose $S^3_L$ is reducible. Let $K_1$ and $K_2$ be components of $L$. Then the split sum formula for sutured Floer homology implies that \begin{align*}
        0 <\rank(\SFH(Y_L,\gamma_L))= 2\cdot\rank(\SFH(S^3_{K_1},\gamma_{K_1}))\cdot\rank(\SFH(S^3_{K_2},\gamma_{K_2}))\leq 4.\end{align*}
    
    Without loss of generality we then have that $0<\rank(\SFH(S^3_{K_1},\gamma_{K_1}))\leq 1$ and \\ ${0<\rank(\SFH(S^3_{K_2},\gamma_{K_2}))\leq 2}$. It follows from Proposition~\ref{BaldwinSivekrank2knotexteriors} that $K_1$ is an unknot and $\gamma_{K_1}$ is a pair of parallel meridional sutures, while $(K_2,\gamma_2)$ is one of the other pairs in the statement of Proposition~\ref{BaldwinSivekrank2knotexteriors}, as desired.

    Suppose now we are in the case that $S^3_L$ is irreducible. If both components of $L$ have meridional sutures then the result reduces to that given by Kim~\cite[Theorem 1]{kim2020links}, so we have that $L$ is the Hopf link or an unlink.
    
    Suppose a component $K$ of $L$ has non-meridional sutures. By Lemma~\ref{lem:spectral} there is a spectral sequence from $\SFH(Y_L,\gamma_L)$ to $\SFH(Y_K,\gamma_K)\otimes V$, where $V$ is a rank $2$ vector space. It follows that each component $K$ of $L$ is of the type given in Proposition~\ref{BaldwinSivekrank2knotexteriors}. We proceed by the cases given in the statement of Proposition~\ref{BaldwinSivekrank2knotexteriors}.

    Suppose $K$ is $T(2,\pm 3)$ with two sutures of slope $\pm 2$. Then Lemma~\ref{lem:spectral} implies that $\SFH(Y,\gamma)$ is supported in the lines $A_K=\pm\frac{3}{2}$. Consider a Thurston norm minimizing surface $\Sigma$ with maximal number of boundary components that is Poincar\'e dual to $[\mu_K]\in H^1(S^3_L)$. Suppose $\Sigma$ intersects $\partial\nu(K')$, where $K'$ is $L-K$. Then after decomposing $(S^3,\gamma_L)$ along $\Sigma$ and removing the excess parallel sutures we obtain a sutured manifold $(Y',\gamma')$ with $\rank(\SFH(Y',\gamma'))= 1$.  It follows that $(Y',\gamma')$ is a product sutured manifold~\cite[Theorem 1.4, Theorem 9.7]{juhasz2008floer}. It follows in turn that $\Sigma$ must be the image of a fiber surface for $T(2,3)$. However, the linking number must be zero, as the Alexander gradings do not shift under the spectral sequence, so we have a contradiction. Suppose $\Sigma$ does not intersect $K'$. Then decomposing along $\Sigma$ yields a sutured manifold $(Y',\gamma')$ with ${\rank(\SFH(Y',\gamma'))=2}$. Observe that since $\Sigma$ has maximal number of boundary components there can be no product annuli connecting $\Sigma_\pm$ to $\partial(\nu(K'))$. Decomposing along the product annuli yielded by Proposition~\ref{prop:allproductannuli} repeatedly -- using the fact that the rank of the sutured Floer homology is at most two, so the dimension of the sutured Floer polytope is at most one -- we obtain the disjoint union of a product sutured manifold and a sutured manifold $(Y',\gamma')$ with spherical boundary components containing $K'$. Capping off the boundary of $(Y',\gamma')$ with $3$-balls we obtain a knot $K$ with $\rank(\SFH(S^3_K,\gamma_K))=1$, where $\gamma_K$ are parallel meridional sutures. It follows that $K$ is the unknot, and $L$ is of the desired form.

  Suppose $K$ is an unknot knot with two sutures of slope $\pm 2$. The arguments for this case resembles the last case. Lemma~\ref{lem:spectral} implies that $\SFH(Y,\gamma)$ is supported in the lines ${A_K=\pm\frac{1}{2}}$. Consider a Thurston norm minimizing surface $\Sigma$ with maximal number of boundary components that is Poincar\'e dual to $\mu_K$ in $S^3_L$. Suppose $\Sigma$ intersects $\partial\nu(K')$, where $K'$ is $L-K$. Then after decomposing $(S^3,\gamma_L)$ along $\Sigma$ and removing the excess parallel sutures we obtain a sutured manifold $(Y',\gamma')$ with $\rank(\SFH(Y',\gamma'))=1$. It follows that $(Y',\gamma')$ is a product sutured manifold 
 by~\cite[Theorem 1.4, Theorem 9.7]{juhasz2008floer}. It follows in turn that $\Sigma$ must be the image of a fiber surface for the unknot. However, the linking number must be zero, as the gradings do not shift in the spectral sequence from $\SFH(S^3_L,\gamma_L)$ to $\SFH(S^3_K,\gamma_K)\otimes V$, so we have a contradiction. Suppose $\Sigma$ does not intersect $K'$. Then decomposing along $\Sigma$ yields a sutured manifold $(Y',\gamma')$ with $\rank(\SFH(Y',\gamma'))=2$. Observe that since $\Sigma$ has maximal number of boundary components there can be no product annuli connecting $\Sigma_\pm$ to $\partial(\nu(K'))$. Decomposing along the product annuli yielded by Proposition~\ref{prop:allproductannuli}, using the fact that the rank of the sutured Floer homology is at most two, so the dimension of the sutured Floer polytope is at most one -- we obtain the disjoint union of a product sutured manifold and a $3$ ball containing $K'$. Capping off the boundary of the $3$-ball with another $3$-ball we obtain a knot $K$ with $\rank(\SFH(S^3_K,\gamma_K))=1$, where $\gamma_K$ is a pair of parallel meridional sutures. It follows that $K$ is the unknot, and $L$ is of the desired form.

   Suppose $K$ is an unknot knot with four meridional sutures. Removing a pair of these sutures from $(Y_L,\gamma_L)$ yields a link exterior with meridional sutures $(S^3_{L},\gamma_{L}')$ with \\${\rank(\SFH(S^3_{L},\gamma_{L}'))\leq 2}$. It follows by work of Ni that $L'$ is an unlink~\cite[Proposition 1.4]{ni2014homological} and the desired result follows.

     Suppose $K$ is an unknot with sutures of slope $0$. We have separate cases corresponding to the number of pairs of parallel such sutures.
     
     Suppose $\gamma_K$ has strictly more than $3$ pairs of parallel sutures of slope zero. Removing one or two pairs of parallel sutures yields a sutured manifold of rank one, which is a product sutured manifold, a contradiction, since $L$ -- and hence $\partial Y_L$ -- has two components. 
     
     Suppose $\gamma_K$ has two pairs of parallel sutures of slope zero. Removing one pair yields a sutured manifold, $(Y_{L},\gamma_{L}')$, with sutured Floer homology of rank at most $2$. Note that $b_1(\partial Y_L)=4$, that $(Y_{L},\gamma_{L}')$ cannot be a product sutured manifold, and that for link exteriors, any horizontal surface must be a product annulus. It follows from Corollary~\ref{cor:trivialinclusion} and Lemma~\ref{lem:trivialinclusionlinks} that there is a product annulus $A$ in $(Y_L,\gamma_L')$. We have two cases according to whether both boundary components of $A$ are on the same component of $\partial Y_L$ or not. Suppose both components lie on the same component of $\partial Y_L$. Decomposing along $A$ yields a sutured manifold with an excess pair of parallel sutures. Removing these parallel sutures we obtain a sutured manifold $(Y,\gamma)$ with $\rank(\SFH(Y,\gamma))=1$. It follows that $(Y,\gamma)$ is a product sutured manifold, a contradiction since there is a connected component of $ Y$ with at least two boundary components. Suppose now that the two boundary components of $A$ lie on distinct components of $\partial Y_L$. Since the boundary component of $A$ corresponding to $K$ is of slope $0$ and the boundary component of $A$ corresponding to remaining component of $L$, $K'$, is of slope $\infty$, it follows that $K$  is a meridian of $K'$. To determine $K'$, observe that if we decompose $(Y_L,\gamma_L)$ along $A$ we obtain a $(Y_{K'},\gamma_{K'})$, where $\gamma_{K'}$ consists of four pairs of parallel oppositely oriented sutures. It follows that $\rank(\SFH(Y_{K'},\gamma_{K'}))\geq 2^3$, a contradiction.

     Suppose that $\gamma_K$ consists of a single pair of longitudinal sutures. Let $\Sigma$ be a longitudinal surface for $K$ in $X(L)$ that minimizes the Thurston norm and amongst such surfaces maximizes the number of boundary components. Consider the sutured manifold $(Y_\Sigma,\gamma_\Sigma)$ obtained by decomposing $(Y_L,\gamma_L)$ along $\Sigma$. Suppose the $A_K$ span is non-zero. Note that $0<\rank(\SFH(Y_\Sigma,\gamma_\Sigma)\leq 2$ and that $(Y_\Sigma,\gamma_\Sigma)$ contains two pairs of excess parallel sutures on the boundary component corresponding to $K$, a contradiction.
     
     Suppose the $A_K$ span is zero. Then applying Lemma~\ref{lem:nondegeneratenorm} implies the $K$ is a meridian of $K'$ and $\Sigma$ is an annulus. Decomposing along $\Sigma$ yields a sutured manifold with a pair of excess parallel sutures. Removing these yields a sutured manifold $(S^3_{K'},\gamma_{K'})$ with $\SFH(S^3_{K'},\gamma_{K'})$ of rank $2$ or $1$. It follows that $K'$ is the unknot as desired.

\end{proof}

Before concluding the proof of Theorem~\ref{thm:mailinknrankclass}, we require one more lemma:

\begin{lemma}\label{lem:degpoly}
    Suppose $(S^3_L,\gamma_L)$ is taut. $P(S^3_L,\gamma_L)$ is non-degenerate unless:\begin{enumerate}
        \item $L$ has a split unknotted component with meridional sutures.
        \item $L$ has a component $K$ with a meridional suture and another component $K'$ which is a meridian of $K$ with longitudinal sutures.
    \end{enumerate} 
\end{lemma}

\begin{proof}
    Suppose $P(S^3,\gamma_L)$ is degenerate. Then the sutured Floer homology norm is degenerate, so in turn for some second relative homology class we have that $x^s(\alpha)=0$.  That is we can find a surface $\Sigma$ with $\Sigma\in H_2(S^3_L,\partial S^3_{L})$ with $(|\Sigma\cap s(\gamma_L)|,\chi(\Sigma))\in\{(0,1),(0,0)\}$. If $(|\Sigma\cap s(\gamma_L)|,\chi(\Sigma))=(0,1)$ then $(Y_L,\gamma_L)$ is compressible, contradicting tautness. If $(|\Sigma\cap s(\gamma_L)|,\chi(\Sigma))=(0,0)$ then $\Sigma$ is an annulus and we see that $L$ has a component $K$ with a meridional suture and another component $K'$ which is a meridian of $K$ with longitudinal sutures.
\end{proof}
We can now conclude the proof of the main theorem of this section.

\begin{proof}[Proof of Theorem~\ref{thm:mailinknrankclass}]

    Suppose $L$ and $\gamma_L$ are as in the statement of the proposition. If $(S^3_L,\gamma_L)$ is not taut, the result follows from Lemma~\ref{lem:nottaut}. We may thus assume that $(Y_L,\gamma_L)$ is taut and so in turn that $\rank(\SFH(Y_L,\gamma_L))>0$.

We prove the result by induction. Note that $n=1$ and $n=2$ cases follow from Proposition~\ref{BaldwinSivekrank2knotexteriors} and Proposition~\ref{prop:2componentrank}.

Suppose $n\geq 3$. Suppose the knots $K_i$ are as in the statement of the Theorem. As in the proof of Proposition~\ref{prop:2componentrank}, we first deal with the case that $Y_L$ is reducible. If $Y_L$ is reducible then there is a separating $2$-sphere which splits $L$ into sublinks $L_1$ and $L_2$ with $|L_i|>0$. The connect sum formula for sutured Floer homology implies that \begin{align*}0<2\cdot\rank(\SFH(S^3_{L_1},\gamma_{L_1}))\cdot\rank(\SFH(S^3_{L_2},\gamma_{L_2}))\leq 2^{n}.\end{align*} Without loss of generality we have by inductive hypothesis that \begin{center}
    
$0<\rank(\SFH(S^3_{L_1},\gamma_{L_1}))\leq 2^{|L_1|-1}$, while $0<\rank(\SFH(S^3_{L_2},\gamma_{L_2}))\leq 2^{|L_2|}$.\end{center} It follows by another application of the inductive hypothesis that $L_1$ is an unlink and $\gamma_{L_1}$ is a collection of pairs of parallel meridional sutures, while $(L_2,\gamma_2)$ is one of the other pairs in the statement of Proposition~\ref{thm:mailinknrankclass}, as desired.

Note that if $P(Y_L,\gamma_L)$ is degenerate the result follows from Lemma~\ref{lem:degpoly}, Lemma~\ref{lem:meridian}.

Suppose now that $Y_L$ is irreducible and $P(Y_L,\gamma_L)$ is non-degenerate. If every component has a pair of parallel meridional sutures, the result is exactly Kim's result~\cite[Theorem 1]{kim2020links}. Suppose some component of $L$ does not have a pair of parallel meridians for sutures. Without loss of generality we can take that component to be $K_1$.

For each $i\neq 1$, Let $L_i$ denote the $n-1$ component sublink of $L$ containing $K_j$ for $j\neq i$. By inductive hypothesis each such $L_i$ satisfies $\rank (\SFH (Y_{L_i},\gamma_{L_i}))\leq 2^{n-1}$.

Suppose $\rank (\SFH (Y_{L_i},\gamma_{L_i}))=2^{n-1}$ for some $L_i$. The the spectral sequence from Lemma~\ref{lem:spectral} from $\SFH(Y_L,\gamma_L)$ to $\SFH(Y_{L_i},\gamma_{L_i})\otimes V$ collapses immediately. Observe that in each case $\rank (\SFH (Y_{L'},\gamma_{L'}))=2^{n-1}$ we have that $\dim(P(Y_{L'},\gamma_{L'}))\leq 1$, so that $\dim(P(Y_{L},\gamma_{L}))\leq 2$ by Remark~\ref{rem:SFHcomp}, a contradiction, since $\dim(P(Y_L,\gamma_L))\geq 3$ as $P(Y_L,\gamma_L))$ is non-degenerate and $L$ has at least three components.

It remains to treat the case in which $\rank(\SFH(S^3_{L_i},\gamma_{L_i}))\leq 2^{n-2}$ for all $i$. By induction and Remark~\ref{rem:SFHcomp} we have that for each $i$ either:\begin{enumerate}
    \item $L_i$ is an unlink and each $\gamma_{L_i}$ consists of pairs of curves of slope infinity.
    \item $K_1$ is an unknot with sutures of slope $0$ split sum $L_i-K_1$.

    \end{enumerate}

 It thus remains to consider the cases in which at least one $L_i$ is of type $2$. Without loss of generality we may take this to be $L_2$.

Consider the spectral sequence from $\SFH(S^3_L,\gamma_L)$ to $\SFH(S^3_{L_2},\gamma_{L_2})\otimes V$. Observe that $\SFH(S^3_{L_i},\gamma_{L_i})\otimes V$ is supported in Alexander multi-grading $\mathbf{0}$.
Since we are now assuming that $P(Y_L,\gamma_L)$ is non-degenerate. It follows that there is a generator $x\in\SFH(S^3_L,\gamma_L)$ with $A_1$ grading strictly greater than zero. Observe that every generator $y$ with $A_1(y)=A_1(x)$ does under every spectral sequence from $\SFH(S^3_L,\gamma_L)$ to $\SFH(S^3_{L_i},\gamma_{L_i})\otimes V$ for $i\neq 1$. It follows that $\rank(\SFH(S^3_L,\gamma_L,A_1=A_1(x)))\geq 2^{n-1}$ by Lemma~\ref{lem:lowerrankbound}. Likewise $\rank(\SFH(S^3_L,\gamma_L,A_1=-A_1(x)))\geq 2^{n-1}$, so that $\rank(\SFH(S^3_L,\gamma_L,A_1=\pm A_1(x)))\geq 2^{n}$ and there are no generators of Alexander grading $0$ a contradiction.

\end{proof}

\section{Longitudinal Surfaces with Sutured Floer homology of next to minimal rank.}\label{sec:mainthmpf}
The goal of this section is to prove the main result of this paper, which we recall here for the reader's convenience:

\begin{theorem}\label{thm:mainbraid}
    Suppose $K$ is a component of an $n$ component non-split link $L$ in $S^3$. The link Floer homology of $L$ in the maximal non-trivial $A_K$ grading is of rank at most $2^n$ if and only if one of the following holds:

    \begin{enumerate}
        \item $K$ is fibered and $L-K$ is braided with respect to $K$.
        \item $K$ is nearly fibered and $L-K$ is braided with respect to $K$.
        \item $K$ is fibered and $L-K$ is a clasp-braid with respect to $K$.
       \item $K$ is fibered and $L-K$ is a stabilized clasp-braid with respect to $K$.
        \item\label{case5} $K$ is fibered and there is a component $K'$ of $L$ such that $K'$ lies in a Seifert surface for $K$ and $L-K-K'$ is braided with respect to $K$.
    \end{enumerate}
\end{theorem}

We prove Theorem~\ref{thm:mainbraid} at the end of Section 5.2, assuming Proposition~\ref{thm:maintangle} which we over the course of this section. Here we discuss some of the consequences of the Theorem. While Martin's braid detection result,~\cite[Proposition 1]{martin2022khovanov}, holds in arbitrary $3$ manifolds, the above result is dependant on the fact that $L$ is a link in $S^3$. The hypothesis that $L$ is non-split is readily removed:

\begin{corollary}\label{rem:reducible}
Suppose $K$ is a component of a split $n$ component link $L$ in $S^3$. The link Floer homology of $L$ in the maximal non-trivial $A_K$ grading is of rank at most $2^{n}$ if and only if one of the following holds:\begin{enumerate}
   \item $L=L_1\sqcup L_2$ where $L_1$ is an $n'$ component link containing $K$ such that link Floer homology of $L_1$ in the maximal non-trivial $A_K$ grading is of rank at most $2^{n'}$ and $L_2$ an unlink.
    \item $L=L_1\sqcup L_2$ where $L_1$ is an $n'$ component link containing $K$ such that link Floer homology of $L_1$ in the maximal non-trivial $A_K$ grading is of rank at most $2^{n'-1}$ and $L_2$ is a Hopf link split sum an unlink.
\end{enumerate}
\end{corollary}

Here $L_1\sqcup L_2$ denotes the split sum of $L_1$ and $L_2$.

\begin{proof}

The link Floer homology of the split sum of an $n$ component link $L$ containing $K$ and an $n'$ component link $L'$ is given by:
\begin{align*}
    \widehat{\HFL}(L\sqcup L')=V\otimes\widehat{\HFL}(L)\otimes \widehat{\HFL}( L')
\end{align*}

Where $V$ is a rank two vector space. This follows from Equation~\ref{eq:connectsum}. The rank of the maximal $A_K$ grading of $\widehat{\HFL}(L)$ is at least $2^{n-1}$. This follows from Martin's braid axis detection result~\cite[Proposition 1]{martin2022khovanov}. We also have that $\rank(\widehat{\HFL}( L'))\geq2^{n'-1}$ since $\widehat{\HFL}( L')$ admits a spectral sequence to $\widehat{\HF}(S^3)\otimes V^{\otimes (n'-1)}$. If the maximal $A_K$ grading of $\widehat{\HFL}(L)$ is of rank $2^{n+n'}$ then we must have that either:\begin{itemize}
    \item $\widehat{\HFL}(L)$ is of rank $2^{n}$ in the maximal Alexander grading and $\rank(\widehat{\HFL}(L'))= 2^{n'-1}$ or,
 \item $\widehat{\HFL}(L)$ is of rank $2^{n-1}$ in the maximal $A_K$ grading and $\rank(\widehat{\HFL}(L))\leq 2^{n'}$.\end{itemize}
 
 $n'$ component links $L'$ with $\rank(\widehat{\HFL}(L'))= 2^{n'-1}$ are unlinks by a result of Ni~\cite[Proposition 1.4]{ni2014homological}, while $n'$ component links $L'$ with $\rank(\widehat{\HFL}(L'))= 2^{n'}$ are Hopf links split sum unlinks by a result of Kim~\cite[Theorem 1]{kim2020links} -- see also Section~\ref{sec:linkext}. If $K$ is a component of $L$ such that $\widehat{\HFL}(L)$ is of rank $2^{n-1}$ in the maximal $A_K$ grading then $K$ is a braid axis~\cite[Proposition 1]{martin2022khovanov}. The result follows from these observation together with Theorem~\ref{thm:mainbraid}.
\end{proof}

\begin{remark}
 In Case~\ref{case5} of Theorem~\ref{thm:mainbraid} we can say a little more under the assumption that $L=K\cup K'$ and $K$ is of genus one. i.e. $K$ is either a trefoil or the figure eight knot. Essential simple closed curves in the once puncture torus are parameterized by pairs of coprime integers $(p,q)$, where $p$ and $q$ can be determined as follows. Fix a pair of oriented curves $\{x,y\}$ representing a symplectic basis for $H_1(\Sigma)$, where $\Sigma\inj S^3$ is a Seifert surface for $K$. If $K'$ is a knot in $\Sigma$ then $p=[K']\cap[x]$, $q=[K']\cap[y]$. Observe that any tuple $(p,q)$ determines a link in $\Sigma$ by taking $p$ parallel copies of $x$ and $q$ parallel copies of $y$ and then taking the oriented resolution of every intersection. The classification of the isotopy type of the link consisting of $K$ and $K'$ is thus -- a priori -- coarser than the quotient of $\{(p,q):\gcd(p,q)=1\}$ by the action of the monodromy on this set. Recall that with respect to appropriate bases, the monodromy of the $T(2,3)$ is given by $\begin{pmatrix}
 1 & -1 \\1 & 0\end{pmatrix}$ while the monodromy of the figure eight knot is given by Arnold's cat map, $\begin{pmatrix}2 & 1 \\1 & 1\end{pmatrix}$. Indeed, the classification is strictly coarser. For example, the curves $(1,0)$ and $(0,1)$ in a fiber surface of the figure eight knot are not related by the action of the monodromy of the figure eight. Nevertheless the resulting links are isotopic. See the two leftmost figures in~\cite[Figure 2]{dey2023unknotted}.
\end{remark}

Applying Proposition~\ref{Thm:Juhaszaffine} to the sutured exteriors of the links mentioned in the statement of Theorem~\ref{thm:mainbraid} allows us to determine $\widehat{\HFL}(L)$ in the maximal $A_K$ grading, complete with the remaining Alexander gradings, up to affine isomorphism. We discuss this in Subsection~\ref{subsec:geography}.

In Section~\ref{sec:annular} we will apply a version of Theorem~\ref{thm:mainbraid} for instanton Floer homology in the case that $K$ is an unknot. For the sake of comparison we state the following corollary, which follows directly from Theorem~\ref{thm:mainbraid}:

\begin{corollary}\label{cor:annularlinkfloer}
    Suppose $U$ is an unknotted component of an $n$ component non-split link $L$ in $S^3$ such that the link Floer homology of $L$ in the maximal non-trivial $A_U$ grading is of rank at most $2^n$. Then one of the following holds:

    \begin{enumerate}
        \item $L-U$ is braided with respect to $U$.
        \item There is an unknotted component $U'$ of $L-U$ such that $L\setminus(U'\cup U)$ is braided with respect to $U$ and $U'$ is isotopic to a curve in a longitudinal disk for $U$.
         \item $L-U$ is a stabilized clasp-braid with respect to $U$.
        \item $L-U$ is a clasp-braid with respect to $U$.
    \end{enumerate}
\end{corollary}

The case in which $L$ is split can be deduced as in Remark~\ref{rem:reducible}.

 Suppose $L$ is a link with a component $K$ and $\Sigma$ is a maximal Euler characteristic, maximal boundary component longitudinal surface for a component $K$ of $L$. We will prove Theorem~\ref{thm:mainbraid} by way of the following Proposition:

\begin{proposition}\label{thm:maintangle}
     Suppose $(Y_T,\gamma_T)$ is irreducible and $T$ has $n$ closed components. Then $\rank(\SFH(Y_T,\gamma_T))\leq 2^{n+1}$ if and only if:

    \begin{enumerate}
        \item $T$ is a braid and $(Y,\gamma)$ is a product sutured manifold.
        \item $T$ is a braid and $(Y,\gamma)$ is an almost product sutured manifold.
        \item $T$ is a clasp-braid and and $(Y,\gamma)$ is a product sutured manifold.
        \item\label{case:casestabilized} $T$ is a stabilized clasp-braid in a stabilized product sutured manifold.
        \item There is a component $t$ of $T$ such that $T-t$ is a braid, $(Y,\gamma)$ is product sutured manifold and $t$ is isotopic to asimple closed curve in $R_+(\gamma_{T-t})$.
    \end{enumerate}
\end{proposition}

Here we follow notation introduced earlier in this paper, see Definition~\ref{sutured tangles}, for instance. A version of this Proposition in which applies in the reducible is given in Proposition~\ref{prop:braid2ormore}. We prove Proposition~\ref{thm:maintangle} by induction over the course of the rest of this Section.

\subsection{Computations and Geography Results}\label{subsec:geography}
 In this section we collect some results concerning the geography problem for link Floer homology. Some of these computations will be useful in subsequent subsections. More specifically, we compute the maximal non-trivial $A_K$ grading of $\widehat{\HFL}$ equipped with $A_{K'}$ gradings with $K'\neq K$ for some of the links in the statement of Theorem~\ref{thm:mainbraid}.

The following affine vector spaces will appear repeatedly:\begin{align}
    B_n(A_1,\dots A_n)=&\begin{cases}\F&\text{ if }A_i\in\{0,1\}\text{ for all }i.\\0&\text{otherwise.}
    \end{cases}
\end{align}

Recall that given a link $L$ with a component $K$, we let $(Y_T,\gamma_T)$ denote the sutured tangle exterior obtained by decomposing the exterior of $L$ along an implicit maximal Euler characteristic longitudinal surface. Here, unlike in other sections, we do not remove excess parallel sutures. We note that if we were to remove these excess parallel sutures, the sutured Floer homology would have fewer $\F^2$ summands.

Recall that $\SFH(Y,\gamma)$ admits a relative $H_1(Y)$-grading.

\begin{lemma}\label{lem:computefiberedbraideddecomposed}
      Suppose $L$ is an $n$ component link with a fibered component $K$ and that $L-K$ is braided with respect to $K$. As a relatively $H_1(Y_T)$ graded vector space we have that: 
      
      \begin{align}
  \SFH(Y_T,\gamma_T,\mathbf{x}) \cong&\begin{cases}\F&\text{ if }\mathbf{x}=\underset{1\leq i\leq n-1}\sum a_i[\mu_i],\text{ where }a_i\in\{0,1\}\text{ for all }i\\0&\text{otherwise.}
    \end{cases}
\end{align}

Here $\{\mu_i\}_{1\leq i\leq n-1}$ are meridians of the $n$ tangle components endowed with excess parallel sutures.
\end{lemma}

\begin{figure}[ht]
    \centering
    \includegraphics[width=6.4cm]{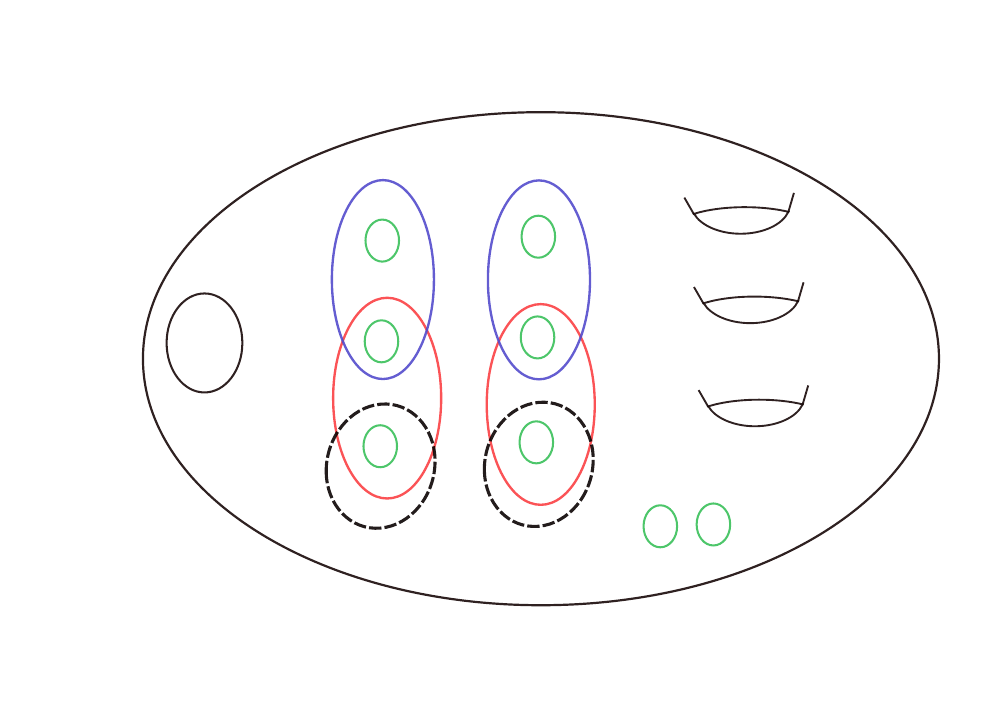}
    \caption{A sutured Heegaard diagram for a sutured braid exterior with some excess parallel sutures, as in the statement of Lemma~\ref{lem:computefiberedbraideddecomposed}. The puncture corresponds to the fibered knot $K$. The Heegaard surface can be thought of as a fiber surface for $K$, with neighborhoods of the components of $L$ removed. The green circles are boundary components corresponding to components of $T$, the red circles are the alpha curves and blue circles are the beta curves. The dotted curves are meridians $\mu_i$ of the relevant tangle components.}
    \label{fiberedbraid}
\end{figure}

\begin{proof}
  $(Y_T,\gamma_T)$ admits a sutured Heegaard diagram of the type shown in Figure~\ref{fiberedbraid}. Note that the each pair of excess parallel sutures contributes a rank two vector space in the tensor product. The relative $H_1(Y_T)$ grading can be computed from the diagram.
\end{proof}

\begin{corollary}\label{lem:computefiberedbraided}
    Suppose $L$ is an $n$ component link with a fibered component $K$ and that $L-K$ is braided with respect to $K$. Then, up to affine isomorphism, the the maximal non-trivial $A_K$ grading of $\widehat{\HFL}(L)$ is given by $B_{n-1}$.
\end{corollary}

\begin{proof}
    Proposition~\ref{Thm:Juhaszaffine} gives an affine map from $\SFH(Y_T,\gamma_T)$ to the maximal non-trivial Alexander grading of $\widehat{\HFL}(L)$. The result follows from Lemma~\ref{lem:computefiberedbraideddecomposed}, noting that the map $F_S$ in Proposition~\ref{Thm:Juhaszaffine} sends the each $[\mu_i]$ grading to the Alexander grading of the corresponding link component. 
\end{proof}

We note in passing that this corollary could have been used to reduce the case analysis for the classification of links with Khovanov and knot Floer homology of low ranks given in~\cite{binns2022rank}.

We do not explicitly compute $\SFH(Y_T,\gamma_T)$ in the case that $K$ is nearly fibred. Rather we given the following Lemma:

\begin{proposition}\label{lem:computenearlyfibereddecomposed}
    Let $L$ be an $n$ component link with a nearly fibered component $K$ such that $L-K$ is braided with respect to $K$. Then $n-1\leq \dim(P(Y_T,\gamma_T))\leq n$ and ${\rank(\SFH(Y_T,\gamma_T))= 2^n}$. Moreover, up to affine isomorphism, the maximum non-trivial $A_K$ grading of $\widehat{\HFL}(L)$ is given by $B_{n-1}\otimes\F^2[0,0\dots 0]$.
\end{proposition}

Here $\F^2[0,0\dots,0]$ denotes the vector space $\F^2$ supported in $A_i$ gradings $0$.

\begin{proof}
Let $(Y_T,\gamma_T)$ be as in the statement of the lemma. Recall that $(Y,\gamma)$ decomposes into a product sutured manifold and an unknot or trefoil exterior with appropriate sutures. Let $(Y_J,\gamma_J)$ be the non-product component of the exterior of a minimal genus Seifert surface for $K$. Note that $\dim(P(Y_J,\gamma_J))\leq 1$ since $\rank(\SFH(Y_J,\gamma_J))= 2$. A sutured Heegaard diagram for $(Y_T,\gamma_T)$ can be obtained from a sutured Heegaard diagram, $\mathcal{H}_K$, for $(Y_J,\gamma_J)$ and a sutured Heegaard diagram, $\mathcal{H}_P$, for a product sutured manifold $(P,\rho)$ by gluing a subset of the boundary components together and modifying the sutured Heegaard diagram in the image of $\mathcal{H}_P$, so as to add excess parallel sutures. The effect of the first step is to decrease $\dim(P(Y_J,\gamma_J))$ by at most one, while the effect of the second step is to increase the dimension by at least $n-2$ and at most $n-1$.  It follows that  $n-1\leq \dim(P(Y_T,\gamma_T))\leq n$.

To see that $\rank(\SFH(Y_T,\gamma_T))=2^n$, observe that $(Y_T,\gamma_T)$ has $H_2(Y_T)\cong 0$ and can be decomposed along a collection of product annuli with at least one boundary component non-trivial in $R(\gamma_T)$ into a reduced sutured manifolds with sutured Floer homology of rank $2$ together with a product sutured manifold with $n-1$ pairs of excess parallel sutures. The desired rank equality follows from~\cite[Lemma 8.9]{juhasz2008floer}. Consider the associated Heegaard diagram, $\mathcal{H}$.

  To conclude, observe again that Proposition~\ref{Thm:Juhaszaffine} gives an affine map from $\SFH(Y_T,\gamma_T)$ to the maximal non-trivial Alexander grading of $\widehat{\HFL}(L)$. The result follows from considering the Heegaard diagram $\mathcal{H}$, noting that the map $F_S$ in Proposition~\ref{Thm:Juhaszaffine} sends (the Poincar\'e dual of) the homology classes of the excess parallel sutures to the Alexander grading corresponding to the $i$th component of $L$ and kills all other gradings. 
\end{proof}

\begin{lemma}\label{lem:nearlybraidedexterior}
     Let $L$ be an $n$ component link with a fibered component $K$ such that $L-K$ is a clasp-braid with respect to $K$.  As a relatively $H_1(Y_T)$ graded vector space $\SFH(Y_T,\gamma_T)$ is given by:
    
    \begin{align*}
  \SFH(Y_T,\gamma_T,\mathbf{x}) \cong&\begin{cases}\F&\text{ if }\mathbf{x}=\underset{1\leq i\leq n-1}\sum a_i[\mu_i],\text{ where }a_i\in\{0,1\}\text{ for all }i\neq 1\text{ and }a_1=\pm 1\\
  \F^2& \text{ if }\mathbf{x}=\underset{1\leq i\leq n-1}\sum a_i[\mu_i],\text{ where }a_i\in\{0,1\}\text{ for all }i\neq 1\text{ and }a_1=0 \\
  0&\text{otherwise.}
    \end{cases}
\end{align*}

\end{lemma}

Here representatives of the $\{\mu_i\}_{1\leq i\leq n-1}$ homology classes are as indicated in Figure~\ref{fig:HDl5a1pairofpants}.
 
\begin{center}

     \begin{figure}[h]
 \includegraphics[width=6.2cm]{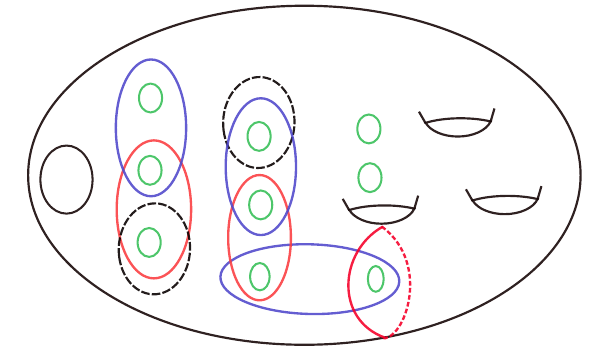}
    \caption{A sutured Heegaard diagram for the sutured manifold $(Y_T,\gamma_T)$ in the statement of Lemma~\ref{lem:nearlybraidedexterior}. We follow the conventions of Figure~\ref{fiberedbraid}, noting that we take $\mu_1$ to be the dotted black curve in the middle of the diagram, which corresponds to the component of $T$ with four parallel sutures.}\label{fig:HDl5a1pairofpants}
   \end{figure}
\end{center}

\begin{proof}

  We proceed as in Lemma~\ref{lem:computefiberedbraideddecomposed}. $(Y_T,\gamma_T)$ admits a sutured Heegaard diagram of the type shown in Figure~\ref{fiberedbraid}. Note that the each pair of excess parallel sutures contributes a rank 2 vector space in the tensor product, such as the ones to the left. The relative $H_1(Y_T)$ grading can be computed from the diagram.

\end{proof}

        \begin{corollary}\label{lem:computenearlyfibered}
       Let $L$ be an $n$ component link with a fibered component $K$ such that $L-K$ is clasp-braided with respect to $K$. Then up to affine isomorphism, the maximal $A_K$ grading for $\widehat{\HFL}(L)$ is given by $P_{n-1}\otimes E_{n-1}$.
         \end{corollary}

         \begin{proof}
   Again we note that Proposition~\ref{Thm:Juhaszaffine} gives an affine map from $\SFH(Y_T,\gamma_T)$ to the maximal non-trivial Alexander grading of $\widehat{\HFL}(L)$. The result follows from Lemma~\ref{lem:nearlybraidedexterior}, noting that the map $F_S$ in Proposition~\ref{Thm:Juhaszaffine} sends (the Poincar\'e dual of) the homology classes of the excess parallel sutures to the Alexander grading corresponding to the $i$th component of $L$.
\end{proof}

Figure~\ref{fig:gluinghandle} shows an example of a sutured manifold as described in the Lemma.

Now set:

\begin{align}
    P_n(A_1,\dots A_n)=&\begin{cases}\F&\text{ if }A_1=\pm 1\text{ and }A_i=0\text{ for all }i>1.\\\F^2&\text{ if }A_i=0\text{ for all }i\\0&\text{otherwise.}
    \end{cases}
\end{align}

and finally:

\begin{align}
    E_n(A_1,\dots A_n)=&\begin{cases}\F&\text{if }A_1=0\text{ and }A_i\in\{0,1\}\text{ for all }i\neq 1.\\0&\text{otherwise.}
    \end{cases}
\end{align}
 
\begin{lemma}\label{lem:homessentialdecomposed}
     Let $L$ be an $n$ component link. Suppose $K$ is fibered component of $L$ and $K'$ is a component of $L$ such that $L-(K\cup K')$ is braided with respect to $K$ and $K'$ is isotopic to a curve on a Seifert surface for $K$. As a relatively $H_1(Y_T)$ graded vector space we have that:
     
     \begin{align}
\SFH(Y_T,\gamma_T,\mathbf{x})\cong&\begin{cases}
            \F&\text{ if } \mathbf{x}=\underset{1\leq i\leq n-1}{\sum}a_i[\gamma_i],  a_i\in\{0,1\}\text{ for all }i\\
            0&\text{ otherwise.}
        \end{cases}
    \end{align}

\end{lemma}

Here $\gamma_i$ are the homology classes of the curves shown in Figure~\ref{fig:homessential}.

\begin{center}

    \begin{figure}[h]
 \includegraphics[width=6.2cm]{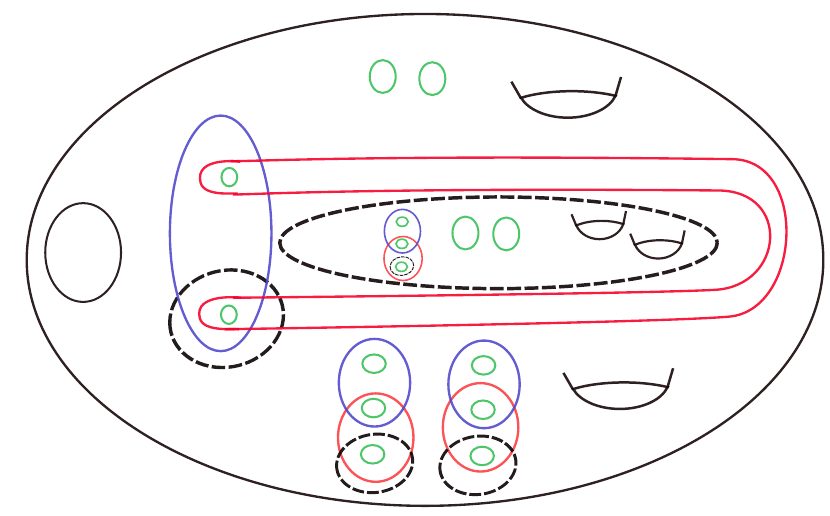}
    \caption{A sutured Heegaard diagram for the sutured manifold $(Y_T,\gamma_T)$ in the statement of Lemma~\ref{lem:nearlybraidedexterior}. We follow the conventions of Figure~\ref{fiberedbraid}, and let $[\gamma_1]$ be the homology class of a surface framed longitude of $K'$ and $\gamma_i$ be isotopic to meridians of the boundary components of the Heegaard diagram corresponding to components of $L-(K\cup K')$ with excess parallel sutures.}
\label{fig:homessential}
\end{figure}
\end{center}

\begin{proof}
    $(Y_T,\gamma_T)$ has a sutured Heegaard diagram as shown in Figure~\ref{fig:homessential}. The result follows.

\end{proof}

         \begin{corollary}
         \label{cor:HFLhomessential}
     
             Let $L$ be an $n$ component link. Suppose $K$ is fibered component of $L$ and $K'$ is a component of $L$ such that $L-(K\cup K')$ is braided with respect to $K$ and $K'$ is isotopic to a curve on a Seifert surface for $K$. Up to affine isomorphism the maximum non-trivial $A_K$ grading of $\widehat{\HFL}(L)$ is given by 
             \begin{align}
B_{n-1}[\mathbf{A}_{F([\gamma_1])}]\oplus B_{n-1}.\end{align}

        Here $[\mathbf{A}_{F([\gamma_1)]}]$ indicates a shift by the image of the homology class of the image of $\gamma_1$ written in terms of the homology classes of the  meridians $\mu_i$ of the components of $L-K$. 
        
         \end{corollary}

Observe that the component of the shift $[\mathbf{A}_{F[\gamma_1]}]$ in the $[\mu_{K'}]$ grading is exactly the Thurston-Bennequin number of $K'$ viewed as a Legendrian knot. If $L_i$ is a component of $L\setminus(K\cup K')$ the component of the shift $[\mathbf{A}_{F[\gamma_1]}]$ in the $[\mu_{L_i}]$ direction is given by $[\lk(K', L_i)]$.

    \begin{proof}
  Again Proposition~\ref{Thm:Juhaszaffine} gives an affine map from $\SFH(Y_T,\gamma_T)$ to the maximal non-trivial Alexander grading of $\widehat{\HFL}(L)$. The result follows from Lemma~\ref{lem:homessentialdecomposed}, noting that the map $F_S$ in Proposition~\ref{Thm:Juhaszaffine} sends (the Poincar\'e dual of) the homology classes of the excess parallel sutures to the Alexander grading corresponding to the $i$th component of $L$ and send $[\gamma_1]$ to an appropriate sum of the homology classes of the meridians of the components of $L-K$.
\end{proof}

\begin{lemma}\label{lem:stabilizableclaspcomp}
    Suppose $L$ is an $n$-component link with a fibered component $K$ and that $L-K$ is a stabilized clasp-braid closure. As a relatively  $H_1(Y_T)$ graded vector space we have that:
  
    \begin{align*}
  \SFH(Y_T,\gamma_T,\mathbf{x}) \cong&\begin{cases}\F&\text{ if }\mathbf{x}=\underset{1\leq i\leq n-1}\sum a_i[\mu_i],\text{ where }a_i\in\{0,1\}\text{ for all }i\neq 1\text{ and }a_1=\pm 1\\
  \F^2& \text{ if }\mathbf{x}=\underset{1\leq i\leq n-1}\sum a_i[\mu_i],\text{ where }a_i\in\{0,1\}\text{ for all }i\neq 1\text{ and }a_1=0 \\
  0&\text{otherwise.}
    \end{cases}
    \end{align*}

    Here $\mu_i$ is the dotted curve shown in Figure~\ref{fig:HDfor3gluing}.
\end{lemma}

\begin{center}

     \begin{figure}[h]
 \includegraphics[width=6.2cm]{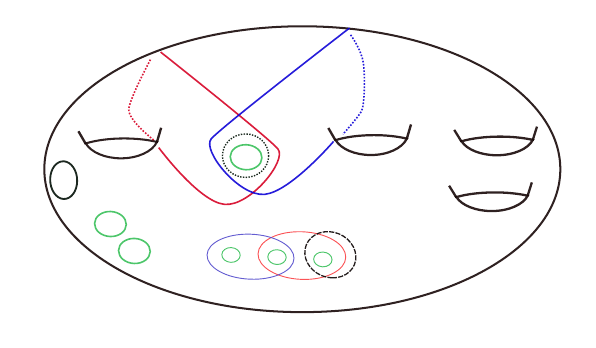}   \caption{A Heegaard diagram for the sutured manifold pictured in Figure~\ref{stabilizable cusp}. }\label{fig:HDfor3gluing}
   \end{figure}
\end{center}

\begin{proof}
Observe that $(Y_T,\gamma_T)$ has a sutured Heegaard diagram as shown in Figure~\ref{fig:HDfor3gluing}. The result follows as in the previous Lemmas.

\end{proof}

\begin{corollary}

     Suppose that $L$ is an $n$ component link with a fibered component $K$ and that $L-K$ is stabilized clasp-braid closure. Then up to affine isomorphism the link Floer homology of $L$ in the maximal non-trivial $A_K$ grading is given by:

\begin{align*}
  \begin{cases}
      \F & \text{ if }A_1=1\text{ and } A_i\in\{0,1\} \text{ for }i\neq 1.\\
      \F^2 & \text{ if }A_1=0\text{ and } A_i\in\{0,1\} \text{ for }i\neq 1.\\
      \F & \text{ if }A_1=-1\text{ and } A_i\in\{0,1\} \text{ for }i\neq 1.\\
      0&\text{ otherwise.}
  \end{cases}
\end{align*}

\end{corollary}

\begin{proof}
Once again we need only apply Proposition~\ref{Thm:Juhaszaffine} and Lemma~\ref{lem:stabilizableclaspcomp}.
\end{proof}

\color{black}

\subsection{The first base case of Theorem~\ref{thm:maintangle}}
We first fix some notation. Let $L$ be a link in $S^3$ with a component $K$. Let $(Y,\gamma)$ be a sutured manifold obtained by decomposing $(S^3_K,\mu_K)$ along a Seifert surface which may or may not be genus minimizing. Let $T$ denote the image of $L-K$ in $(Y,\gamma)$. Let $(Y_T,\gamma_T)$ be the sutured tangle exterior of $T$, as in Definition \ref{sutured tangles}. In this subsection we prove Theorem~\ref{thm:maintangle} in the case that $T$ has no closed components.

\begin{remark}\label{rmk:conditions}
Here we note the following properties of the sutured manifolds $(Y_T,\gamma_T)$, which follows from the Definition \ref{sutured tangles}:

\begin{enumerate}
    \item $\partial Y_T$ has $n+1$ connected components, where $n$ is the number of closed components of $T$.
    \item $Y_T$ is a submanifold of $S^3$.
    \item $(Y_T,\gamma_T)$ is without excess parallel sutures.
    \item Suppose $S$ is a connected surface embedded in $Y_T$ with boundary a union of sutures $\bigcup_{1\leq i\leq k} s_{\gamma_i}$. Then either:\begin{enumerate}
        \item $k=2$ and $\gamma_1$ and $\gamma_2$ are parallel in $\partial Y_t$, or,
        \item $\{\gamma_i\}=\{\gamma\}\cup\{\gamma_t:t\text{ is a compoent of }T\text{ with }\partial t\cap R_\pm(\gamma)\neq\emptyset\}$.
    \end{enumerate}

\end{enumerate}

\end{remark}

\begin{proposition}\label{prop:nocomponentlongsurface}
 Suppose that $(Y_T,\gamma_T)$ is taut and that $T$ has no closed components. If $\rank(\SFH(Y_T,\gamma_T))\leq 2$ then either:

\begin{enumerate}
    \item $(Y,\gamma)$ is  a product sutured manifold and $T$ is a braid.
    \item $(Y,\gamma)$ is an almost product sutured manifold and $T$ is a braid.
    \item  $(Y,\gamma)$ is a product sutured manifold and $T$ is a clasp-braid.
    \item $(Y,\gamma)$ is a stabilized product sutured manifold and $T$ is a stabilized clasp-braid.
\end{enumerate}

\end{proposition}

Our proof strategy is to apply Proposition~\ref{prop:allproductannuli} to decompose $(Y_T,\gamma_T)$ into simpler pieces and to then classify the ways in which these pieces can be reassembled to a sutured manifold which satisfies the conditions in Remark~\ref{rmk:conditions}. Recall that if $(Y,\gamma)$ is a connected almost product sutured manifold we can always decompose it as $(S^3_{K'},\gamma_{K'})\sqcup (P,\rho)$ where each component of $(P,\rho)$ is glued to $(S^3_{K'},\gamma_{K'})$.

We can further simplify our analysis somewhat with the following lemma:

\begin{lemma}\label{lem:cabling}
    Suppose $T$ is a tangle in a balanced sutured manifold $(Y,\gamma)$. Let $T'$ be a stabilization of $T$. Then  $\SFH(Y_{T'},\gamma_{T'})\cong\SFH(Y_T,\gamma_T)$.\end{lemma}

    \begin{proof}
        To see this note that a sutured Heegaard diagram for $(Y_{T'},\gamma_{T'})$ can be obtained by adding a small puncture to a sutured Heegaard diagram $(\Sigma,\mathbf{\alpha},\mathbf{\beta})$ for $(Y_T,\gamma_T)$ near the boundary component of $\partial\Sigma$ corresponding to $t$. The result follows.
    \end{proof}

    This will allow us to reduce some case analysis to the setting that tangles do not have parallel components connecting $R_+(\gamma)$ tp $R_-(\gamma)$. We call a tangle $T$ in a sutured manifold $(Y,\gamma)$ \emph{simplified} if $(Y_T,\gamma_T)$ is  aspherical and horizontally prime and $T'$ is not a stabilization.

We will also find the following Lemma helpful:
\begin{lemma}\label{lem:sameSFH}
    Suppose $(Y',\gamma')$ decomposes as $(Y_1,\gamma_1)\sqcup (P,\rho)$ where $(P,\rho)$ is a connected product sutured manifold with  at least one non-glued boundary component. Then $$\rank(\SFH(Y',\gamma'))=\rank(\SFH(Y_1,\gamma_1)).$$
\end{lemma}

\begin{proof}
    Consider a Heegaard diagram for $(Y',\gamma')$ and note that pseudo-holomorphic-disks cannot cross into $(P,\rho)$ because it contains a boundary component.
\end{proof}

We now investigate the ways in which sutured tangle exteriors $(Y_T,\gamma_T)$ with \\ $\rank(\SFH(Y_T,\gamma_T))\leq 2$ can be recovered from some of its possible decompositions.

\begin{lemma}\label{lem:nocomponentstep2}
      Suppose $(Y_T,\gamma_T)$ can be decomposed along a minimal family of essential product annuli $\{A_i\}_{i\in I}$ into the disjoint union of $(P,\rho)$ and $(S^3_{K'},\gamma_{K'})$, where $(S^3_{K'},\gamma_{K'})$  is given up to mirroring by one of the following two pieces:
          \begin{enumerate}
        \item A solid torus with parallel oppositely oriented sutures of slope $2$.
\item The exterior of a right handed trefoil with two parallel oppositely oriented sutures of slope $2$.

    \end{enumerate}

    Then $(Y,\gamma)$ is an almost product sutured manifold and $T$ is a braid.
  \end{lemma}

  Note that $|I|\leq 2$, and that $P$ has at most two components. No component of $(P,\rho)$ has base surface a disk or an annulus, as in these cases at least one of the product annuli is non-essential.

    \begin{proof}
        Suppose $I=\emptyset$. Then $(P,\rho)=\emptyset$ and $(Y,\gamma)$ can be recovered from $(S^3_{K'},\gamma_{K'})$ by gluing in a thickened disk along one of the boundary sutures. But in this case $H_1(Y)\cong\Z/2$ since the sutures of $\gamma_{K'}$ are of slope $2$, so that $Y$ cannot be the exterior of a surface in $S^3$, a contradiction.

 Suppose $(S^3_{K'},\gamma_{K'})$ is glued to $(P,\rho)$ along exactly one suture $\gamma_i$. Then $(Y,\gamma)$ can be recovered from $(S^3_{K'},\gamma_{K'})\cup_{\gamma_i}(P,\rho)$ by gluing in thickened disks along all but one of the sutures. However, in this case $H_1(Y)$ would have a $\Z/2$ summand since the sutures $\gamma_{K'}$ are of slope $2$, so that $Y$ cannot be the exterior of a surface in $S^3$, a contradiction.

 Suppose $(S^3_{K'},\gamma_{K'})$ is glued to $(P,\rho)$ along two sutures. If $(P,\rho)$ has two components, then we would again have that $H_1(Y)$ contains a $\Z/2$ summand since the sutures $\gamma_{K'}$ are of slope $2$, so that $Y$ cannot be the exterior of a surface in $S^3$, a contradiction.

 It follows that $(P,\rho)$ has a single connected component. $\SFH(Y,\gamma)\cong\SFH(Y_T,\gamma_T)$ by Lemma~\ref{lem:sameSFH}. Thus $(Y,\gamma)$ can be decomposed along essential product annuli $A_1,A_2$ into $(S^3_{K'},\gamma_{K'})\sqcup(P',\rho')$ where $(P',\rho')$ is a product sutured manifold and $T$ is a braid in $(P',\rho')$. It follows in turn that $(Y,\gamma)$ is an almost product sutured manifold and $T$ a braid, as desired.    
    \end{proof}

    The next reassembly Lemma will be substantially more involved. We prove the following preparatory lemma.

 \begin{lemma}\label{lem:firstsymmetricproduct}
     Suppose $(Y',\gamma')$ admits a sutured Heegaard diagram with exactly one $\alpha$ curve. Also suppose $(Y_T,\gamma_T)$ admits a sutured decomposition along a minimal family of essential product annuli $\{A_i\}_{i\in I}$ to the sutured manifold given by the disjoint union of  $(Y',\gamma')$ and $(P,\rho)$, where $(P,\rho)$ is some product sutured manifold, such that every suture of $(Y',\gamma')$ is glued. Then either:\begin{enumerate}
         \item $\rank(\SFH(Y,\gamma))=\rank(\SFH(Y_T,\gamma_T))=\rank(\SFH(Y',\gamma'))$ and $T$ is a braid in $(Y,\gamma)$, or,
         
         \item $(P,\rho)$ has a component $(P_1,\rho_1)$ with base a punctured disk and $\gamma\not\in\rho_1$.
     \end{enumerate}
 \end{lemma}

 The difference between this Lemma and Lemma~\ref{lem:sameSFH} is that here we do not require that the product sutured manifold has a boundary component which is not glued.

 \begin{proof}
     Consider a sutured Heegaard diagram $\mathcal{H}_T$ for $(Y_T,\gamma_T)$ give by gluing a sutured Heegaard diagram for $(Y',\gamma')$ to a sutured Heegaard diagram for $(P,\rho)$. A sutured Heegaard diagram for $(Y,\gamma)$ is obtained by filling in every suture $\gamma_T$ save for $\gamma$. The differential can only possibly change if a punctured disk in $\mathcal{H}_T$ bounded by an $\alpha$ curve and a $\beta$ curve is filled in with a disk.

     Suppose we are in the case that $(P,\rho)$ does not have a component $(P_1,\rho_1)$ with base space a punctured disk such that $\gamma\not\in\rho_1$. We need to check that $T$ is a braid in $(Y,\gamma)$. We check that each member of the minimal family of essential product annuli $A_i$ in $(Y_T,\gamma_T)$ is still essential in $(Y,\gamma)$. This is true unless there is a component $(P_2,\rho_2)$ of $(P,\rho)$ with base space an annulus and $\gamma\in \rho_2$. But we are exluding this by assumption, so the desired result follows.
 \end{proof}

    \begin{proposition}\label{lem:nocomponentstep1}
           Suppose $(Y_T,\gamma_T)$ can be decomposed along a minimal family of essential product annuli $\{A_i\}_{i\in I}$ into the disjoint union of a product sutured manifold $(P,\rho)$ and a solid torus with $4$-parallel longitudinal sutures, $T_4$. Then either:\begin{enumerate}
               \item  $(Y,\gamma)$ is an almost product sutured manifold and $T$ is a braid,
              \item  $(Y,\gamma)$ is a product sutured manifold and $T$ is a clasp-braid, or,
              \item $(Y,\gamma)$ is a stabilized product sutured manifold and $T$ is a stabilized clasp-braid.
           \end{enumerate}
    \end{proposition}

     Note that $|I|\leq 4$ and $P$ has at most four connected components. No component of $(P,\rho)$ has base surface a disk or an annulus, as in these cases at least one of the product annuli is non-essential.

    \begin{proof}
Our goal is to reconstruct $T$ and $(Y,\gamma)$ be regluing along the minimal family of essential product annuli.

    We first examine the case in which $T_4$ has some self gluing. Note that in order to ensure each boundary component of $\partial Y$ has a suture, adjacent sutures cannot be glued. Suppose two non-adjacent sutures are glued. Observe that one can obtain an embedded Klein bottle in $Y_T$  by considering the image of an annulus $A$ cobounding the glued sutures. It follows that there is an embedded Klein bottle in $Y_T$ and hence in $S^3$, a contradiction.

   We may thus assume that $T_4$ has no self gluing.  Moreover, at least two of $T_4$'s sutures are glued as otherwise $(Y_T, \gamma_T)$ would not satisfy the properties listed in Remark \ref{rmk:conditions}. In particular it would contains three parallel sutures. We proceed now by cases according to the number of sutures of $T_4$ that are glued to $(P,\rho)$.

   \textbf{Two of $T_4$'s sutures are glued:} Suppose that $(P,\rho)$ is not connected. Let $(P_1,\rho_1)$ and $(P_2,\rho_2)$ be the two components. Note that $\rho_1$ and $\rho_2$ each cobound a surface with either of the remaining sutures of $T_4$. Thus, since neither $(P_i,\rho_i)$ has base surface a disk or annulus, we have that at least one of these surfaces does not contain $\gamma$ and is not an annulus, contradicting point $4$ in Remark~\ref{rmk:conditions}.
   
   Suppose the two glued sutures are glued to the same connected component of $(P,\rho)$. There are two possibilities:\begin{enumerate}
       \item  Two sutures on $T_4$ that are not separated by another suture on $T_4$ are glued to $R_\pm(P,\rho)$, as shown in Figure~\ref{fig:gluinghandle}.
       \item Two sutures on $T_4$ that are separated by another suture on $T_4$ are glued to $\rho$, as shown Figure~\ref{fig:gluinghandle2}. 
   \end{enumerate}
   
   Suppose we are in the first case. We claim that the sutured manifold obtained with the sutures on the outside of the handle is identical to one with the sutures on the inside of the handle. To see this note that for a surface $\Sigma$ we define a diffeomorphism $\phi$ which interchanges any two pairs of essential simple closed curves. Consider the diffeomorphism $\phi\times\mathbbm{1}$ on $\Sigma\times[-1,1]$. Thus we may define a diffeomorphism which interchanges the core and the co-core of the $T_4$ handle we are attaching. One can then readily see that this sutured manifold is diffeomorphic to the standard one. Thus, assuming sutured clasp exteriors only arise as the exteriors of clasps -- a fact we will prove in Lemma~\ref{lem:suturedclaspext} -- we are exactly in case (2) of the statement of the Proposition.

\begin{center}   
     \begin{figure}[h]
 \includegraphics[width=5cm]{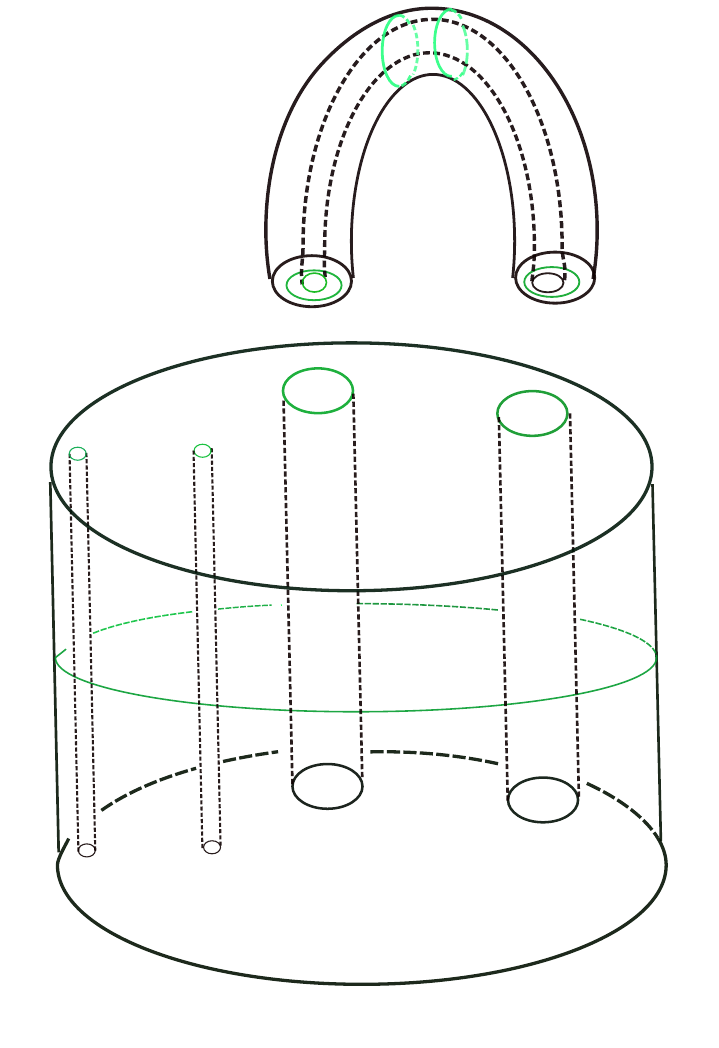}
    \caption{Gluing $T_4$ to a connected product sutured manifold. Note that in general the product sutured manifold can have genus.}\label{fig:gluinghandle}
\end{figure}
 
\end{center}

Suppose now that we are in the second case: that $T_4$ is glued to $(P,\rho)$ along two sutures that are separated by another suture in $\partial(T_4)$. Suppose one of these sutures in $T_4$ is $\gamma$ -- i.e. the suture corresponding to the knot $K$. Then in $S^3$, $K$ bounds a disk which $L$ punctures exactly once -- i.e. $L\setminus K$ has exactly one component, $L_1$, and $K$ is a meridian of $L_1$. It follows that $(Y_T,\gamma_T)$ is the exterior of $L_1$ endowed with a pair of meridional sutures. This cannot have sutured Floer homology of rank $2$, since knots in $S^3$ have sutured Floer homology of odd rank, a contradiction.

If neither of the remaining components of $T_4$'s sutures are $\gamma$ then two meridians of $T$ are homologous, a contradiction unless they are from the same component $t$ of $T$. This is impossible since they represent different homology classes in $H_1(\partial Y_T)$, apart from in the special case that $(P_i,\rho_i)$ has base surface an annulus for some $i$, which we are excluding by assumption.

\begin{center}
 
     \begin{figure}[h]
 \includegraphics[width=6.2cm]{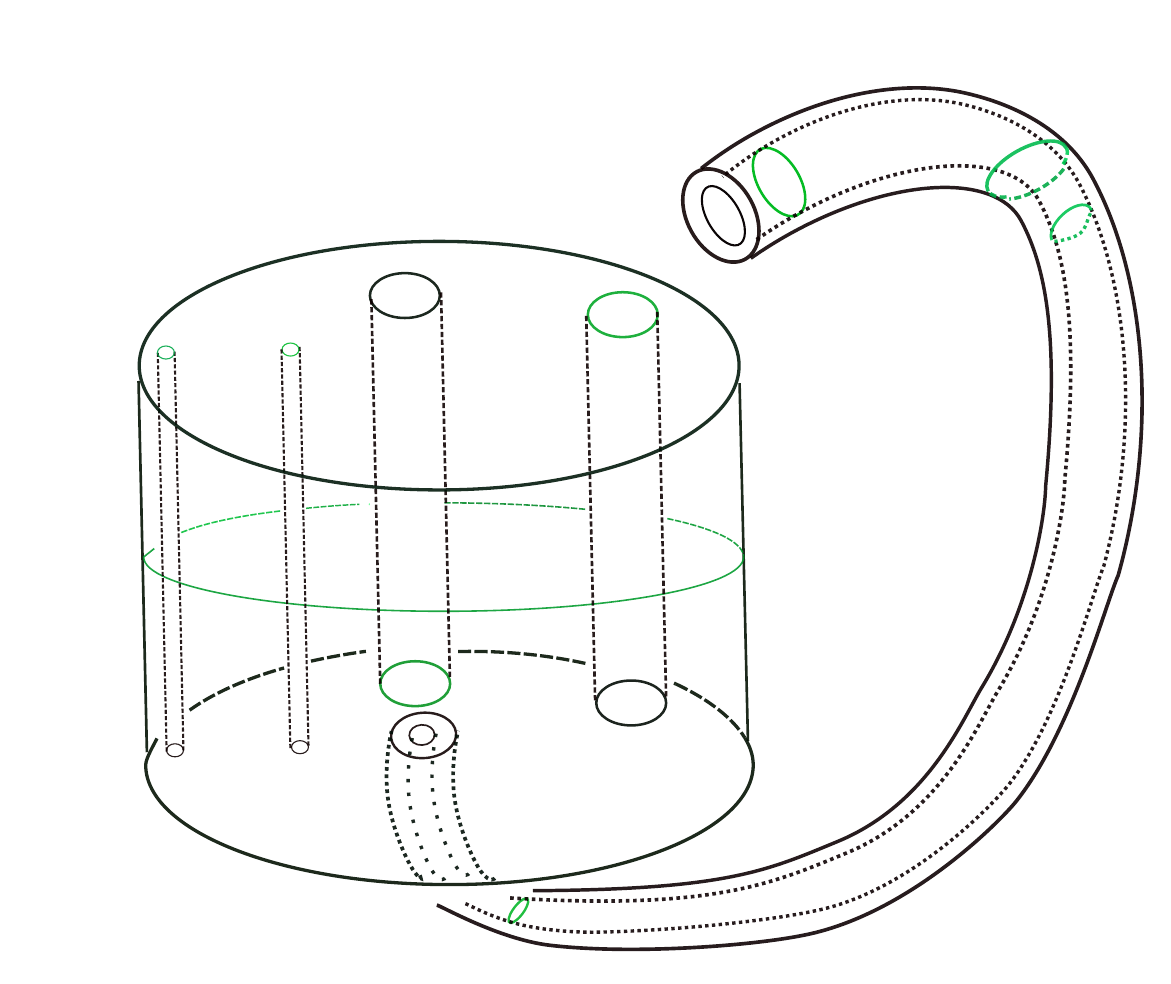}
    \caption{Gluing $T_4$ to a connected product sutured manifold. Note that in general the product sutured manifold may have genus.}\label{fig:gluinghandle2}
   \end{figure}
\end{center}

\textbf{Three of $T_4$'s sutures are glued:} Suppose $P$ is connected. Then we are in case $(3)$ of the statement of the Proposition.

Suppose that $P$ is not connected. Let $(P_1,\rho_1)$ be a connected component glued to $T_4$ along a single suture. Observe that the core of $P_1$ cobounds a surface with the remaining suture of $T_4$. Since this surface cannot be an annulus, it must contain $\gamma$, but this contradicts the assumption that $R_\pm(\gamma_T)$ are Euler characteristic maximizing.

   \textbf{All of $T_4$'s sutures are glued:} Then $\gamma_T$ contains no pairs of parallel sutures. It follows that each component $t$ of $T$ has $t\cap R_\pm(\gamma)\neq\emptyset$. Note that $T_4$ admits a sutured Heegaard diagram with exactly one $\alpha$ curve. In particular, Lemma~\ref{lem:firstsymmetricproduct} implies that $(Y,\gamma)$ is an almost product sutured manifold and $T$ is a braid in $(Y,\gamma)$  unless $(P,\rho)$ has a component $(P_1,\rho_1)$ with base surface $S$ a multi-punctured disk and $\gamma\not\in\rho_1$, so that we are in case $(1)$ of the statement of the Proposition. We may thus assume we have such a component $(P_1,\rho_1)$.

Suppose $(P,\rho)$ has at least three components. Observe that at least one component is glued along a single product annulus to $T_4$, in addition to $(P_1,\rho_1)$. Let $(P_2,\rho_2)$ be one such component. Suppose $(P_1,\rho_1)$ and $(P_2,\rho_2)$ are glued to $T_4$ along sutures $\gamma_1$ and $\gamma_2$. Observe that the curves $(s(\rho_1)\cup s(\rho_2))\cup(s(\gamma_1)\cup s(\gamma_2))$ cobound a surface in $(Y_T,\gamma_T)$. It follows that $\gamma\in \rho_2$, as else a collection of meridians of arcs in $T$ cobound a surface. We thus have a contradiction as $R_+(\gamma)$ is supposed to be Euler characteristic maximizing.

   Suppose now that $(P,\rho)$ has exactly two components, $(P_1,\rho_1)$ and $(P_2,\rho_2)$. Consider first the case in which $(Y_T,\gamma_T)$ is simplified, i.e. that $P_1$ is an annulus, so that we may take $(P,\rho)=(P_2,\rho_2)$.  It follows that we can reduce to the case that $T_4$ has three sutures glued and $(P,\rho)$ has connected base (but $R_+(\gamma_T)$  may no longer an Euler characteristic maximizing surface for $\gamma$). The non simplified case follows directly from Lemma~\ref{lem:cabling}.

\end{proof}

We now turn to proving that sutured clasp exteriors arise only as the sutured exteriors of clasps.

\begin{lemma}\label{lem:suturedclaspext}
    Suppose $(Y_T,\gamma_T)$ is a sutured clasp-braid exterior. Then $(Y,\gamma)$ is a product sutured manifold and $T$ is a clasp-braid.
\end{lemma}

\begin{proof}
    Observe that by decomposing $(S^3_L,\gamma_L)$ along a different (but isotopic, see Figure~\ref{fig:isotopiclongitudes}) longitudinal surface we can obtain a sutured manifold $(Y',\gamma')$ which is identical to $(Y_T,\gamma_T)$ except that it has a pair of parallel meridional sutures on the boundary component corresponding to $t\subset T$ with $\partial t \subseteq R_+(\gamma)$. Moreover $\rank(\SFH(Y',\gamma'))=4$. Fill in the sutures of $(Y',\gamma')$ corresponding to $t\subset T$ with boundary components in both $R_+(\gamma)$ and $R_-(\gamma)$ to obtain a sutured manifold $(Y'',\gamma'')$. Decompose $(Y'',\gamma'')$ along a maximal family of product annnuli and remove the connected components consisting of product sutured manifolds. The resulting sutured manifold, $(Y''',\gamma''')$, has $g(\partial Y''')\leq 3$ by Lemma~\ref{lem:trivialinclusionlinks}. Now observe that $(Y''',\gamma''')$ contains two pairs of parallel sutures and that removing either pair and decomposing along the annuli produced by Lemma~\ref{lem:trivialinclusionlinks} yields the disjoint union of a product sutured manifold and a sutured manifold $(Y_0,\gamma_0)$ with $g(\partial Y_0)\leq 1$. Note that $T$ is irreducible since it contains no closed components, so that $g(\partial Y_0)\neq0$. Using the arguments from the proof of Theorem~ \ref{lem:nocomponentstep1} we find that $(Y_0,\gamma_0)$ is a sutured clasp exterior. It follows that $(Y''',\gamma''')$ is a sutured clasp exterior with an extra pair of parallel sutures or a sutured clasp exterior glued to a once punctured annulus again with an extra pair of parallel sutures.

    In the first case, since in the punctured torus $R_\pm(\gamma''')$ simple closed curves representing the same homology class are isotopic, we have that the pair of parallel meridians are the expected curves --- as shown in the lower right hand side of Figure~\ref{fig:isotopiclongitudes} --- concluding the proof in this case.

    In the second case let $\gamma_1,\gamma_2$ be the sutures in $(Y''',\gamma''')$ without parallel copies. Observe that there is an essential product annulus $A$ in the sutured manifold obtained by removing from $(Y''',\gamma''')$ either of the two pairs of parallel sutures in $\partial Y'''$. We claim these two annuli, which we call $A_1$ and $A_2$, are isotopic.

   To do so we first show that $\partial_\pm A_1$ and $\partial_\pm A_2$ are isotopic in $R_\pm((Y''',\gamma_1\cup\gamma_2))$ respectively. Observe that without loss of generality $\partial_-A_1$ separates the pair of parallel meridional sutures from $\gamma_1\cup \gamma_2$ in $R_-((Y''',\gamma_1\cup\gamma_2))$ while $\partial_+A_2$ separates the pair of parallel meridional sutures from $\gamma_1\cup \gamma_2$ in $R_+((Y''',\gamma_1\cup\gamma_2))$. Now observe that the homotopy classes of $\partial_+A_1$ in $\pi_1(R_+(Y''',\gamma_1\cup\gamma_2))$ and $\partial_+A_2$ in $\pi_1(R_-(Y''',\gamma_1\cup\gamma_2))$ are determined by the fact that their images under the inclusion into $Y'''$ agree with $\partial_-A_{1}$ and $\partial_+A_{2}$ respectively. The claim follows, so we may take $A_1$ and $A_2$ to have shared boundary.
    
    We now show that the interiors of $A_1$, $A_2$ can be isotoped so that they do not intersect. To do so observe that after a small perturbation we may assume that $A_1$ and $A_2$ intersect in a finite collection of circles. We proceed to remove these intersections. There are two types of intersections: homologically essential curves in $A_1$ and homologically inessential curves in $A_1$.
    
    Suppose there is a component of $A_1\cap A_2$ that is homologically inessential in $A_1$. Pick an innermost such component $c$. Note that $c$ cannot be homologically essential in $A_2$ as else $\partial_\pm A_2$ would be null-homotopic. $c$ is thus homologically inessential in $A_2$. But then there are disks $D_i\subset A_i$ with $D_1\cap D_2=c$. Since $Y'''$ is irreducible it follows that $D_1\cup D_2$ bounds a solid ball which we can use to isotope away the intersection $c$.

    Suppose there is a component of $A_1\cap A_2$ that is homologically essential in $A_1$. Pick such a component that is closest to $\partial_+A_1=\partial_+A_2$, $c$. Observe that $c$ must be homologically essential in $A_2$ as else $\partial_+A_1=\partial_+A_2$ would be null-homotopic. Consider the annuli cobounded by $c$ and $\partial_+A_1$ in $A_1$ and $c$ and $\partial_+A_2$ in $A_2$. The union of these two annuli is an embedded torus in $Y'''$. Since $Y'''$ is atoroidal this torus bounds a solid torus, which may use to remove the curve $c$ via an isotopy of $A_1$ and $A_2$.
    
    Now, having removed all of the intersections between the interiors of $A_1$ and $A_2$ and ensuring that $\partial A_1=\partial A_2$ by an isotopy, we may consider the torus obtained as $A_1\cup A_2$. Again using the fact that $Y'''$ is atoroidal we see that $A_1$ and $A_2$ are isotopic. It follows that $A_1$ and $A_2$ are essential product annuli in $(Y''',\gamma''')$ and we can decompose $(Y''',\gamma''')$ along either one to reduce to the first case, completing the proof.

    \end{proof}

    \begin{center}
        \begin{figure}[h!]
            \centering
            \includegraphics[width=10cm]{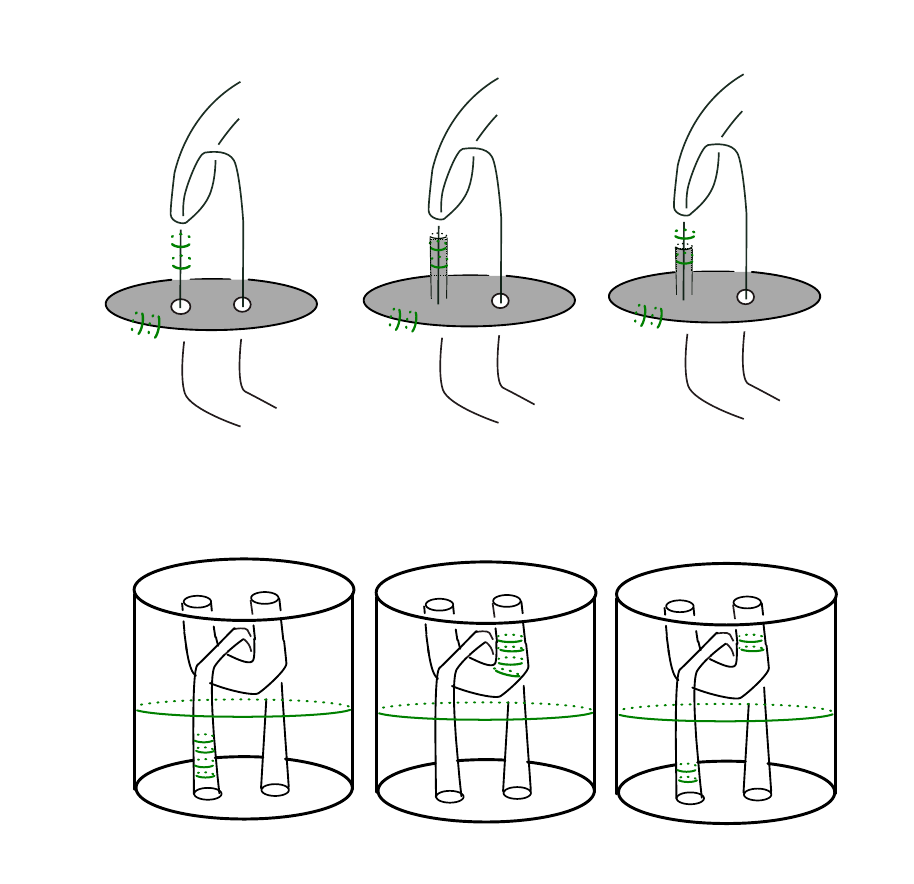}
            \caption{Three isotopic longitudinal surfaces for a component $K$ of a link $L$ together with the three sutured manifolds obtained by decomposing the exterior of $L$ along the three surfaces. Here, for simplicity, the link is a two component link. The green curves indicate the sutures.}
            \label{fig:isotopiclongitudes}
        \end{figure}
    \end{center}

\begin{proof}[Proof of Proposition~\ref{prop:nocomponentlongsurface}]
    Let $(Y_T,\gamma_T)$ be as in the statement of the Proposition. Since $(Y_T,\gamma_T)$ is taut, $\rank(\SFH(Y_T,\gamma_T))>0$ by~\cite[Proposition 9.18]{juhasz2006holomorphic}. If $\rank(\SFH(Y_T,\gamma_T))=1$ then $(Y_T,\gamma_T)$ is a product sutured manifold by a Theorem of Juh\'asz~\cite[Theorem 9.7]{juhasz2008floer}, so that we are in case $(1)$ in the statement of the Proposition.
    
    It remains to determine the sutured manifolds $(Y_T,\gamma_T)$ with $\rank(\SFH(Y_T,\gamma_T))=2$. Suppose $(Y_T,\gamma_T)$ is not horizontally prime. Then there exists a horizontal surface $\Sigma$ that is not parallel $R_\pm(\gamma_T)$. Decomposing $(Y_T,\gamma_T)$ along $\Sigma$ yields a disjoint union of two sutured manifolds $(Y_1,\gamma_1))\sqcup(Y_2,\gamma_2)$. We then have that\begin{align*}
        \SFH(Y_1,\gamma_1)\otimes\SFH(Y_2,\gamma_2)\cong\SFH(Y_T,\gamma_T)\cong\F^2.
    \end{align*}

    It follows, without loss of generality, that $\rank\SFH(Y_1,\gamma_1)=1$ so that $(Y_1,\gamma_1)$ is a product sutured manifold~\cite[Theorem 9.7]{juhasz2008floer}, contradicting the fact that $\Sigma$ is not parallel $R_\pm(\gamma)$. Thus $(Y_T,\gamma_T)$ is horizontally prime. 
  
  Suppose $(Y_T,\gamma_T)$ is reducible. Then by Equation~\ref{eq:connectsum} we have that there are sutured manifolds $(Y_1,\gamma_1)),(Y_2,\gamma_2)$ such that \begin{align*}
        \F^2\cong\SFH(Y_T,\gamma_T)\cong V\otimes\SFH(Y_1,\gamma_1)\otimes\SFH(Y_2,\gamma_2)
  \end{align*}

  where $V$ is a rank two vector space. It follows that $(Y_i,\gamma_i)$ are both product sutured manifolds. The connect sum of two such manifolds has disconnected boundary, a contradiction.
  
  Sutured manifolds $(Y_T,\gamma_T)$ with $\rank(\SFH(Y_T,\gamma_T))=2$ are by definition almost product sutured manifolds. By applying Proposition~\ref{prop:allproductannuli} repeatedly -- and noting that the resulting sutured manifolds can have rank at most two and hence sutured Floer homology of dimension at most one -- $(Y_T,\gamma_T)$ can decomposed along a family of product annuli $\{A_i\}_{i\in I}$ into a product sutured manifold and non-product pieces $\{(Y_i,\gamma_i)\}_{1\leq i\leq n}$ with $\dim (P(Y_i,\gamma_i))=b_1(\partial Y_i)/2$. If there are no non-product pieces we are in case $(1)$ of the statement of the proposition and we are done. We may this assume that $n\geq 1$. Since decomposing along product annuli does not increase the rank of sutured Floer homology by~\cite[Proposition 8.10]{juhasz2008floer}, it follows that $0<\prod_{1\leq i\leq n}(\rank(\SFH(Y_i,\gamma_i)))\leq 2$. It follows that $\SFH(Y_i,\gamma_i)=1$ for all but at most one $1\leq i\leq n$ so that $(Y_i,\gamma_i)$ is a product sutured manifold~\cite[Theorem 9.7]{juhasz2008floer} for all but at most one $1\leq i\leq n$, a contradiction unless $n=1$ and $\rank(\SFH(Y_1,\gamma_1))>1$. We can thus take $\rank(\SFH(Y_1,\gamma_1))=2$. It follows that $\dim(P(Y_1,\gamma_1))\leq 1$, so that $b_1(\partial Y_1)/2\leq 1$. Thus in turn we have that $\partial Y_1$ is a sphere or a torus. $\partial Y_1$ is not a sphere since we are assuming that $Y_T$ is irreducible. If $\partial Y_1$ a torus then, by Theorem~\ref{thm:mailinknrankclass}, $Y_1$ is one of the following three pieces up to mirroring:

    \begin{enumerate}
        \item A solid torus with $4$-parallel longitudinal sutures.
        \item A solid torus with parallel oppositely oriented sutures of slope $2$.
\item The exterior of a right handed trefoil with two parallel oppositely oriented sutures of slope $2$.
\end{enumerate}
The result now follows directly by Lemma~\ref{lem:nocomponentstep1} and Lemma~\ref{lem:nocomponentstep2}. 
\end{proof}

Figure~\ref{fig:example} gives an example of a sequence of decompositions along essential product annuli, as described in the proof above.

\begin{figure}
    \centering
    \includegraphics[width=5in]{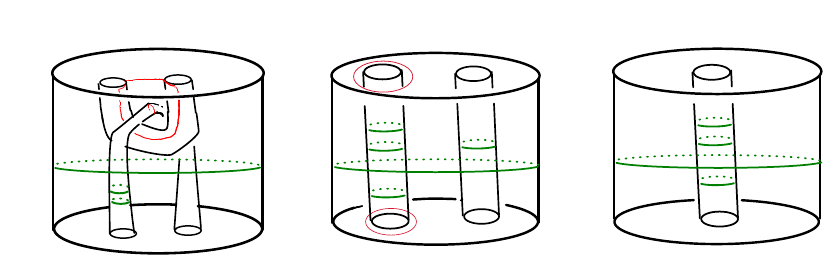}
    \caption{The left most sutured manifold is an elementary clasped sutured piece with sutured Floer homology of rank $2$. The middle sutured manifold is obtained from the first by decomposing along the product annulus with boundary given by the red curves in the first picture. The sutured manifold on the right is obtained by decomposing along the product annulus with boundary given by the red curves in the middle sutured manifold, and removing the product sutured piece. Notice that the sutured manifold on the right is a solid torus with 4 parallel longitudinal sutures. }
    \label{fig:example}
\end{figure}

We can now derive Theorem~\ref{thm:mainbraid} from Proposition~\ref{thm:maintangle}. We remind the reader that we have only proved the first base case of Proposition~\ref{thm:maintangle} thus far, and that remaining subsections are dedicated to completing the proof.

\begin{proof}[Proof of Theorem~\ref{thm:mainbraid}]
    Suppose $K$ and $L$ are as in the statement of the theorem. Let $\Sigma$ be a maximal Euler characteristic longitudinal surface for $K$. Let $(Y,\gamma)$ be the sutured complement of $\Sigma$ and $T$ be the tangle in $(Y,\gamma)$ given by the image of $L-K$ in $(Y,\gamma)$. Observe that if $T$ has $n'$ closed components then $\rank(\SFH(Y_T,\gamma_T))\leq 2^{n'+1}$. It follows that $(Y,\gamma)$ and $ T$ are as in the statement of Proposition~\ref{thm:maintangle}.

    Aside from in Case~\ref{case:casestabilized} in the statement of Proposition~\ref{thm:maintangle}, the result follows directly by regluing $R_+(\gamma)$ and $R_-(\gamma)$. Case~\ref{case:casestabilized} is more subtle; we need to show that the diffeomorphism identifies the core of the stabilizing handle in $R_+(\gamma)$ with the cocore of the stabilizing handle in $R_-$ and vice-versa. To do so first fill in all of the sutures of $(Y_T,\gamma_T)$ save for $\gamma$ and exactly one of those corresponding to a tangle component that passes through the stabilizing handles. Now reglue $R_\pm(\gamma)$ by a diffeomorphism $\phi$ which is isotopic to the diffeomorphism $R_+(\gamma)\to R_-(\gamma)$ which recovers $L$, when one forgets the basepoints corresponding to the tangle. This yields the sutured exterior of a new link $L'$ in $S^3$, one component of which is isotopic to $K$. Observe that, for $L'$, $K$ has another maximal Euler characteristic longitudinal surface, $\Sigma'$, which has genus one less than $g(R_\pm(\gamma'))$, but two more punctures. Decompose $(S^3_{L'},\gamma_{L'})$ along $\Sigma'$. After removing excess parallel sutures, the resulting sutured manifold $(Y',\gamma')$ has sutured Floer homology of rank $2$. By the arguments given in the proof of Proposition~\ref{lem:nocomponentstep1}, $(Y',\gamma')$ is a clasp-braid and $L'\setminus K$ is a clasp-braid closure with respect to $K$. It follows in turn that $\phi$ has the desired action on the components of $R_\pm(\gamma_T)$ corresponding to the stabilizing handles on the stabilized product sutured manifold.
\end{proof}

\subsection{The second base case of Theorem~\ref{thm:maintangle}}\label{subsec:2ndbasecase} We now prove Theorem~\ref{thm:mainbraid} in the case that a maximal Euler characteristic longitudinal surface for $K$ intersects all but one component of $L-K$. We continue to follow the notation of the previous subsection. However, we additionally assume that the maximal Euler characteristic longitudinal surface for $K$ that we decompose along has a maximal number of boundary components.
\begin{proposition}\label{Prop:atmostcomponentintheinterior}
  Suppose that $(Y_T,\gamma_T)$ is taut, irreducible, $T$ contains a single closed component, $t$, and $\rank(\SFH(Y_K,\gamma_K))\leq 4$. Then $(Y,\gamma)$ is a product sutured manifold ${(\Sigma\times[-1,1],\partial\Sigma\times[-1,1])}$, $T-t$ is a braid and up to isotopy, $t\subset \Sigma\times\{0\}$ and $t$ does not bound a disk in $(Y_T,\gamma_T)$.
\end{proposition}

\begin{proof}

Suppose $(Y,\gamma)$, and $T$ are as in the statement of the Proposition. Suppose $(Y_T,\gamma_T)$ is not horizontally prime. Then there exists a horizontal surface $\Sigma$ that is not parallel $R_\pm(\gamma_T)$. Decomposing $(Y_T,\gamma_T)$ along $\Sigma$ and removing a pair of excess parallel sutures yields the disjoint union of two sutured manifolds $\{(Y_i,\gamma_i)\}_{1\leq i\leq 2}$. Since, $\Sigma$ is horizontal we have that\begin{align*}
        \SFH(Y_1,\gamma_1)\otimes\SFH(Y_2,\gamma_2)\cong\F^2.
    \end{align*}

    It follows, without loss of generality, that $\rank(\SFH(Y_1,\gamma_1))=1$, so that $(Y_1,\gamma_1)$ is a product sutured manifold~\cite[Theorem 9.7]{juhasz2008floer}, contradicting the assumption that $(Y_T,\gamma_T)$ is horizontally prime.
    
    We thus have that $(Y_T,\gamma_T)$ is horizontally prime. Let $T'=T-t$. Recall from Subsection~\ref{subsec:suturedFloerbackground} that there is a spectral sequence from $\SFH(Y_T,\gamma_T)$ to ${\SFH(Y_{T'},\gamma_{T'})\otimes V}$, where $V$ is a rank $2$ vector space. This spectral sequence collapses the (relative) $[\mu_{t}]\in H_1(Y_T)$ grading, where $\mu_t$ is a meridian of $t$. It follows that we have two pairs of generators ${\{x_1,x_2\},\{y_1,y_2\}\in\SFH(Y_T,\gamma_T)}$ such that $\epsilon(x_1-x_2)$ and $\epsilon(y_1-y_2)$ contained in ${\langle\mu_{K}\rangle\subset H_1(\partial Y_K)}$.  It follows that $\dim(P(Y_T,\gamma_T))\leq 2$.

Observe that there is no essential product annulus $A$ with one component on $\partial \nu(t)$ and the other on $R_\pm(Y,\gamma)$, as this would contradict our assumption that $\Sigma$ has a maximal number of boundary components.

Thus the same spectral sequence argument that $\dim(P(Y_T,\gamma_T))\leq 2$ applies even after decomposing $(Y_T,\gamma_T)$ along essential product annuli. In turn, by repeated applications of Proposition~\ref{prop:allproductannuli}, there is a collection of essential product annuli $\{A_i\}_{i\in I}$ in $(Y_T,\gamma_T)$ such that after decomposing $(Y_T,\gamma_T)$ along $\{A_i\}_{i\in I}$ and removing product sutured manifold components we are left with a sutured manifold $(Y',\gamma')$ with $g(\partial Y') \leq 2$.

If $(Y',\gamma')$ is disconnected, then since none of the essential product annuli $\{A_i\}_{i\in I}$ we have decomposed along had boundary components on both the component of $\partial Y_T$ corresponding to $t$ and the component of $\partial Y_T$ corresponding to the longitudinal surface $\Sigma$, we have that $(Y',\gamma')$ has at least one spherical boundary component. This contradicts the assumption that $(Y_T,\gamma_T)$ is irreducible.

It follows that $(Y',\gamma)$ is connected with $g(\partial(Y'))\leq2$. Indeed, $g(\partial(Y'))=2$, as else $(Y_T,\gamma_T)$ is reducible, contradicting tautness.  If there are any spherical boundary components we can fill them without increasing the rank of the sutured Floer homology. Observe that the resulting sutured manifold $(S^3_{L'},\gamma_{L'})$ is in fact a two component link exterior, with sutures on one component that may not be parallel pairs of meridians. We have that $\rank(\SFH(S^3_{L'},\gamma_{L'}))\leq 4$. It follows from Theorem~\ref{thm:mailinknrankclass} and the fact that $(Y_T,\gamma_T)$ is irreducible that $(Y_{L'},\gamma_{L'})$ is a Hopf link with meridional sutures or a Hopf link where one component has meridional sutures and the other has a pair of parallel sutures of slope $0$. In the case that there are slope zero sutures we would find that the meridian of $T$ is homologous to $[\gamma]+\sum_{t_i\in T'}[\mu_i]$ in $(Y,\gamma)$, where $[\mu_i]$ is a meridian of the $i$th component of $T'$, contradicting the fact that $t$ and $K$ have linking number zero. It follows that $(Y_{L'},\gamma_{L'})$ is a Hopf link with meridional sutures. Gluing this to a product sutured manifold yields the desired result.

\end{proof}

The version of Proposition~\ref{Prop:atmostcomponentintheinterior} in the case that $(Y_T,\gamma_T)$ is reducible follows as in Corollary~\ref{rem:reducible}:
\begin{corollary}
Let $(Y,\gamma)$ be a sutured manifold and $T$ be a tangle with a single closed component $t$. Suppose $(Y_T,\gamma_T)$ is reducible and $\rank(\SFH(Y_T,\gamma_T))\leq 4$. Then $(Y,\gamma)$, $T-t$ are as given in the statement of Proposition~\ref{prop:nocomponentlongsurface} and $t$ is a split unknot.
\end{corollary}

\begin{proof}
    Suppose we are in the setting of the statement of the Corollary, and that $(Y_T,\gamma_T)$ is reducible. Then $(Y_T,\gamma_T)$ can be written as the connect sum of two sutured manifolds $(Y_1,\gamma_1)$ and $(Y_2,\gamma_2)$. By Equation~\ref{eq:connectsum}:$$\rank(\SFH(Y_1,\gamma_1))\cdot\rank(\SFH(Y_2,\gamma_2))\leq 2.$$ Without loss of generality we can take $\gamma\subset\gamma_1$. Since $S^3$ is irreducible, $\gamma_t$, the components of $\gamma_T$ corresponding to $t$, are contained in $\gamma_2$. Indeed, $(Y_2,\gamma_2)$ is a knot exterior with pairs of parallel meridional sutures. Observe that $0<\rank(\SFH(Y_2,\gamma_2))\leq 2$, so that $(Y_2,\gamma_2)$ is an unknot exterior with a pair of parallel meridional sutures. In particular we have that $\rank(\SFH(Y_2,\gamma_2))=1$, so that $\rank(\SFH(Y_1,\gamma_1))\leq 2$. The result follows from applying Proposition~\ref{prop:nocomponentlongsurface}.
\end{proof}

\subsection{The inductive Step in the proof of Theorem~\ref{thm:maintangle}.}
In this subsection we give the inductive step in the proof of Theorem~\ref{thm:maintangle}. We require a number of preparatory combinatorial lemmas. For these lemmas we let $V$ be a non-trivial vector space with a relative $\Z\langle \{\mu_i\}_{1\leq i\leq n},\{\nu_i\}_{1\leq i\leq k}\rangle$ grading, where $n,k\in\Z$. We require the relative $\Z\langle \{\mu_i\}_{1\leq i\leq n},\{\nu_i\}_{1\leq i\leq k}\}\rangle$ grading to satisfy the property that $gr(x-y)+gr(y-z)=gr(x-z)$.

\begin{lemma}\label{lem:lowerrankbound}
 Suppose that for each homogeneous elements $x\in V$ and $1\leq i\leq n$, there exists a homogeneous element $y\in V$ such that $gr(x-y)\neq 0\in\Z\langle \mu_i\rangle$. Then  for any homogeneous element $x$ of $V$, $\{ y:gr(x-y)\in \Z\langle \mu_i\rangle_{1\leq i\leq n}\}$ is supported in at least $2^n$ gradings, so that in particular  $\rank(\{ y:gr(x-y)\in \Z\langle \mu_i\rangle_{1\leq i\leq n}\})\geq 2^n$.
\end{lemma}

\begin{proof}
    We proceed by induction on $n$. Since $V$ is non-trivial, the base case is immediate.

    Suppose we have proven the Lemma for some $m\in\Z^{\geq 0}$. We have that there are at least $2^m$ elements $\{z_i\}_{1\leq i\leq 2^m}$ such that $gr(z_i-z_j)\in\langle \mu_i\rangle_{1\leq i\leq n}$. Consider $z_1$. By assumption there exists an element $z_{n+1}$ such that $gr(z_{m+1}-z_1)\in\langle \mu_{m+1}\rangle$. Observe that the vector space $W:=\langle x:gr(x-z_{m+1})\subset \langle \mu_i\rangle_{1\leq i\leq m}\rangle $ does not contain $z_i$ for $1\leq i\leq 2^m$. By inductive hypothesis $W$ contains at least $2^m$ elements, whence the result follows.
    
\end{proof}

\begin{lemma}\label{lem:upperranknbound}
    
Suppose that $\rank(V)\leq 2^{n+1}$ and that for each homogeneous element $x\in V$ and $1\leq i\leq n$, there exists an element $y\in V$ such that $gr(x-y)\neq 0\in\Z\langle \mu_i\rangle$. Then $V$ is supported in at most two $\{\nu_i\}_{1\leq i\leq m}$ gradings.
\end{lemma}

\begin{proof}
Suppose $V$ is supported in at least three $\{\nu_i\}_{1\leq i\leq m}$ gradings. Then there exists three generators $x_1,x_2,x_3$ with $\langle gr(x_i-x_j),\langle\nu_m\rangle_{1\leq m\leq k}\rangle\neq 0$ for all $i\neq j$. By Lemma~\ref{lem:lowerrankbound}, $\rank(\langle z:gr(z-x_i)\in \langle \mu_i\rangle_{1\leq i\leq n} \rangle)\geq 2^n$ for each $i$. Thus $\rank(V)\geq 3\cdot 2^n$, a contradiction.
    
\end{proof}

For the next lemma, we let $P(V)$ denote the polytope given by the (affine) convex hull of the support of the vector space $V$.
\begin{corollary}\label{lem:dimbound}
   Suppose that $\rank(V)\leq 2^{n+1}$ and that for each element $x\in V$ and $1{\leq i\leq n}$, there exists an element $y\in V$ such that $gr(x-y)\neq 0\in\Z\langle\mu_i\rangle$. Then ${\dim(P(V))\leq n+1}$.
\end{corollary}

\begin{proof}
    This follows from Lemma~\ref{lem:upperranknbound}. 
\end{proof}

For the inductive step it will be convenient for us to phrase Theorem~\ref{thm:maintangle} without the assumption that $(Y_T,\gamma_T)$ is irreducible.
\begin{proposition}\label{prop:braid2ormore}

 Let $(Y,\gamma)$ be a sutured manifold containing a tangle $T$. Let $T'$ be the $n$-component subtangle of $T$ consisting of closed components. Suppose $(Y_T,\gamma_T)$ is taut and irreducible $\rank(\SFH(Y_T,\gamma_T))\leq 2^{n+1}$ then one of the following three is true:
       
       \begin{enumerate}
\item $T'$ is a split unlink and
       \begin{enumerate}
           \item $(Y,\gamma)$ is  a product sutured manifold and $T - T'$ is a braid.
    \item $(Y,\gamma)$ is an almost product sutured manifold and $T - T'$ is a braid.
    \item  $(Y,\gamma)$ is a product sutured manifold and $T-T'$ is a clasp-braid.
    \item $(Y,\gamma)$ is a stabilized product sutured manifold and $T-T'$ is a stabilized clasp braid.

       \end{enumerate}

       \item There is a component $t$ of $T'$ such that $T'-t$ is a split unlink, $(Y,\gamma)$ is a product sutured manifold $(\Sigma\times[-1,1],\partial\Sigma\times[-1,1])$, $T-T'$ is a braid and up to isotopy $t\subset \Sigma\times\{0\}$ and $t$ does not bound a disk in $(Y_T,\gamma_T)$.

       \item $T'$ is the split sum of a Hopf link and an unlink, $(Y,\gamma)$ is a product sutured manifold and $T-T'$ is a braid.
            
       \end{enumerate}

\end{proposition}

  To prove this proposition we first show that the dimension of $P(Y_T,\gamma_T)$ behaves well under certain product annulus decompositions for the $(Y_T,\gamma_T)$ in the statement of Proposition~\ref{prop:braid2ormore}.
    
\begin{lemma}\label{lem:polytopedim}
Suppose $T$ is a tangle in a sutured manifold $(Y,\gamma)$ with $n$ closed components $\{t_i\}_{1\leq i \leq n}$ and that $0<\rank(\SFH(Y_T,\gamma_T))\leq 2^{n+1}$. Suppose additionally that for all $i$ either:\begin{enumerate}
    \item  the tangles $T-t_i$ are either as in the statement of Proposition~\ref{prop:braid2ormore} or,
    \item the sutured manifolds $(Y_{T-t_i},\gamma_{T-t_i})$ are not taut.
\end{enumerate}

Let $(Z_{T'},\gamma_{T'})$ be any sutured manifold obtained by decomposing $(Y_T,\gamma_T)$ along a family of essential product annuli with boundary in $R_\pm(\gamma)$. Then $\dim(P(Z_{T'},\gamma_{T'}))\leq n+1$.
  \end{lemma}
For the proof we identify closed components of $T$, $t_i$, and closed components of $T'$, $t'_i$. We will make extensive us of Lemma~\ref{lem:spectral}: that for each closed component $t_i\in T$ there are spectral sequences from $\SFH(Y_T,\gamma_T)$ to $\SFH(Y_{T-t_i},\gamma_{T-t_i})\otimes V$ collapsing the $[\mu_i]$ gradings.

  \begin{proof}
We proceed by induction. The $n=0$ case holds because $\rank(\SFH(Z_{T'},\gamma_{T'}))\leq 2$, so that $\dim(P(Z_{T'},\gamma_{T'}))\leq 1$.      
  
For the inductive step we have a number of cases depending on $(Z_{T'-t_i},\gamma_{T'-t_i})$. Suppose that $(Y_{T-t_i},\gamma_{T-t_i})$ is taut for some $i$. Then $(Y_{T-t_i},\gamma_{T-t_i})$ is of form $1$ or $2$ in the statement of Proposition~\ref{prop:braid2ormore}. 
    
    With this in mind deal with the cases:
    \begin{enumerate}
        \item  Suppose that for some $i$, $(Y_{T-t_i},\gamma_{T-t_i})$ is as in case $1b),c),d)$ or $2$ or $3$ in the statement of Proposition~\ref{prop:braid2ormore}. Then $\rank(\SFH(Z_{T'-t_i},\gamma_{T'-t_i}))=2^{n}$ so the spectral sequence from $\SFH(Z_{T'},\gamma_{T'})$ to $\SFH((Z_{T'-t_i},\gamma_{T'-t_i}))\otimes V$ collapses immediately. It follows in turn that the $\dim(P(Z_{T'},\gamma_{T'}))\leq \dim(P(Z_{T'-t_i'},\gamma_{T'-t_i'})))+1$. But we have that $\dim(P(Z_{T-t_i'},\gamma_{T-t_i'})))\leq n$ by inductive hypothesis so the claim holds.

        \item Suppose now that there is no $i$ such that $(Y_{T-t_i},\gamma_{T-t_i})$ is as in case $1b),c),d)$, $2$ or $3$ in the statement of Proposition~\ref{prop:braid2ormore}.  That is we suppose that for some $i$, which without loss of generality we take to be $1$, $(Y_{T-t_1},\gamma_{T-t_1})$ is as in case $1a)$ in the statement of the Proposition, i.e. that $\underset{i\geq 2}{\bigcup}t_{i}$ is a split unlink $T-\underset{i\geq 1}{\bigcup}t_{i}$ is a braid and $(Y,\gamma)$ is a product sutured manifold. Since $\rank(\SFH(Z_{T'-t'_{1}},\gamma_{T'-t'_1})\otimes V)= 2^{n}$, it follows from Lemma~\ref{lem:spectral} that there is a collection of generators $X=\{x_k\}_{1\leq k\leq 2^{n}}$ such that $\gr(x_{k_1},x_{k_2})\in\langle\mu_i\rangle\subset H_1(Z_{T'})$, where $\mu_i$ is the meridian of $t_i$.
        
        If all other generators $y$ of $\SFH(Z_{T'},\gamma_{T'})$ satisfy $\gr(x_j,y)\in\langle\mu_i\rangle$, for some ${1\leq j\leq 2^{n}}$ we have that $\dim(P(Z_{T'},\gamma_{T'}))=1$ and we are done.

        Suppose then that there exists some generator $y$ such that $\gr(x_j,y)\not\in\langle\mu_i\rangle$ for some $j$. Observe that for all but at most one $j$, $j_0$, $\gr(x_k,y)\not\in\langle\mu_j\rangle$ for any ${1\leq k\leq 2^n}$. Consider the plane $P$ containing $y$ spanned by all $\langle \mu_i\rangle_{i\in I}$ where ${I=\{1\leq i\leq n:i\neq j_0\}}$. Note that $I$ contains $n-1$ elements. Observe that if a generator $z$ from $P$ persists under the spectral sequence to $\SFH(Z_{T'-t_i'},\gamma_{T'-t_i'})\otimes V$ for some $i\in I$ then the closed components of $T'-t_i'$ would be a split unlink and $\rank(\SFH(Z_{T'},\gamma_{T'-t_i'}))=2^{n-1}$ with support in a single multi-grading. We would thus have found the remaining $2^n$ distinct generators and deduce that $\dim(P(Z_{T'},\gamma_{T'}))\leq 3$, proving the result.
        
        Suppose then that no generator in the plane $P$ that persists under any spectral sequence to $\SFH(Z_{T'-t_i'},\gamma_{T'-t_i'})$ for any $i\in I$. Then we can apply Lemma~\ref{lem:lowerrankbound} to obtain a further $2^{n-1}$ generators $\{y_i\}_{1\leq i\leq 2^{n-1}}$ of distinct multi-gradings. Observe that unless one of $\{y_i\}_{1\leq i\leq 2^{n-1}}$ differs from an element of $X$ by an element in $\langle\mu_{j_0}\rangle$, we can find a further $2^{n-1}$ distinct generators, and deduce that $\dim(P(Y_{T'},\gamma_{T'}))\leq n$.
        
        Suppose then that one of the generators $\{y_i\}_{1\leq i\leq 2^{n-1}}$ differs from a generator in $L$ by an element in $\langle\mu_{j_0}\rangle$. We can still find a further $2^{n-1}-2$ generators that do not increase the dimension of $\{y_i\}_{1\leq i\leq 2^{n-1}}\cup X$. Consider the two remaining generators $a,b$. Suppose that adding both does not increase the dimension of the polytope. Then $\dim(P(Z_{T'},\gamma_{T'}))\leq n$ and we have the desired result. Suppose adding at least one generator, say $a$, increases the dimension of the polytope. Then it differs from the previous generators in at least one grading in $H_1(Z_{T'})/\langle\mu_i\rangle$. Observe that this generator must die under the spectral sequence to $\SFH(Z_{T'-t_{1}'},\gamma_{T'-t_{1}'})$, so it follows that $\gr(a,b)\in\langle \mu_1\rangle$. This implies that $\dim(P(Z_{T'},\gamma_{T'}))\leq n+1$, as desired.

    \end{enumerate}

 Suppose now that $(Y_{T-t_i},\gamma_{T-t_i})$ is not taut for all $i$. Then under each spectral sequence from $\SFH(Z_{T'},\gamma_{T'})$ to $\SFH(Z_{T'-t'_i},\gamma_{T'-t'_i})\otimes V$ corresponding to forgetting $t'_i$, every generator has a component of a differential to or from another generator. That is, for all non-zero $x\in\SFH(Z_{T'},\gamma_{T'})$, $1\leq i\leq n$ there exists a generator $x_i$ with $gr(x_i-x)\neq 0\in\langle\mu_j\rangle\subset H_1(Z_{T'})$. By Lemma~\ref{lem:dimbound} we have that $\dim(P(Z_{T'},\gamma_{T'}))\leq n+1$.

\end{proof}

\begin{proof}[Proof of Proposition~\ref{prop:braid2ormore}]

We have proven the $n=0$ and $n=1$ cases in Proposition~\ref{prop:nocomponentlongsurface} and Proposition~\ref{Prop:atmostcomponentintheinterior}.

We have two cases: either $(Y_T,\gamma_T)$ is reducible or it is not. By the connect sum formula for sutured Floer homology, Equation~\ref{eq:connectsum}, it suffices to prove the Proposition in the case that $(Y_T,\gamma_T)$ is irreducible.

We thus assume that $n\geq 2$ and that $(Y_T,\gamma_T)$ is irreducible. We proceed towards a contradiction. Suppose $(Y_T,\gamma_T)$ is not horizontally prime. Then there exists a horizontal surface $\Sigma'$ that is not parallel $R_\pm(\gamma_T)$. Decomposing $(Y_T,\gamma_T)$ along $\Sigma'$ and removing all pairs of excess parallel sutures yields the disjoint union of two sutured manifolds $\{(Y_i,\gamma_i)\}_{1\leq i\leq 2}$. We then have that
\begin{align*} \SFH(Y_1,\gamma_1)\otimes\SFH(Y_2,\gamma_2)\cong\F^2.
    \end{align*}

    It follows, without loss of generality, that $\rank\SFH(Y_1,\gamma_1)=1$, so that $(Y_1,\gamma_1)$ is a product sutured manifold~\cite[Theorem 9.7]{juhasz2008floer} contradicting the fact that $(Y_L,\gamma_L)$ has multiple boundary components. We may thus assume that $(Y_T,\gamma_T)$ is horizontally prime.

    Suppose now that there was an essential product annulus $A$ with both boundary components on a boundary component of $(Y_T,\gamma_T)$ corresponding to a closed component of $T$. Then, since $S^3$ is atoroidal, decomposing along $A$ would yield a split sutured manifold $(Y_{T'},\gamma_{T'})\sqcup(S^3_{L'},\gamma_{L'})$ where $T'$ has $m<n$ closed components and $L'$ has $n-m+1>1$ components. Observe that: $$2^{n+1}\geq\rank(\SFH((Y_{T'},\gamma_{T'})\sqcup(S^3_{L'},\gamma_{L'}L)))\geq \rank(\SFH(Y_{T'},\gamma_{T'}))\cdot\rank(\SFH(S^3_{L'},\gamma_{L'}))>0.$$

    Now, $\rank(\SFH(Y_{T'},\gamma_{T'}))\geq 2^m$ -- which follows from the inductive hypothesis -- so we have that \[2^{n+1-m}\geq \rank(\SFH(S^3_{L'},\gamma_{L'})).\] It follows that $(S^3_{L'},\gamma_{L'})$ is as in the statement of Theorem~\ref{thm:mailinknrankclass}. Since we are assuming that $(Y_T,\gamma_T)$ is irreducible and that the annulus $A$ was essential we have that ${L'}$ is either a Hopf link with meridional sutures, a Hopf link where one component has two sutures of slope $0$. The latter case is impossible as then two closed components of $L'$ would have isotopic meridians. It follows that $m=n-1$. Since $\rank(\SFH(S^3_{L'},\gamma_{L'}))= 4$, we have that $\rank(\SFH(Y_{T'},\gamma_{T'}))\leq 2^{m}$ and $T'$ has strictly fewer closed components than $T$. In particular by induction we see that $(Y,\gamma)$ is a product sutured manifold and $T'$ a braid together with some split unlinked components. Thus $T$ is a braid together with a split unlink and a split Hopf link, as desired.

  The only remaining case is that in which there are no essential product annuli with both boundary components on a component of $\partial Y_T$ corresponding to a closed component of $T$. Since $(Y_T,\gamma_T)$ is horizontally prime, and the dimension of $P(Y_T,\gamma_T)$ does not increase under decomposition along product annuli with boundary in $R_\pm (\gamma)$ by Lemma~\ref{lem:polytopedim}, Proposition~\ref{prop:allproductannuli} implies that we can decompose $(Y_T,\gamma_T)$ along product annuli $\{A_j\}$ to obtain a sutured manifold $(Y',\gamma')$ with $\dim(P(Y',\gamma'))\leq n+1$. Moroever, $\partial Y'$ has no spherical components since $(Y_T,\gamma_T)$ is irreducible and $\partial Y'$ contains at least $n$ toroidal components since $T$ has $n$ closed components. It follows that $\partial Y'$ contains exactly $n+1$ toroidal components. Indeed, we have that all but one of the components of the boundary has meridional sutures. However, since we are assuming that $n\geq 3$, Theorem~\ref{thm:mailinknrankclass} implies that $L$ is split, contradicting our assumption that $(Y_T,\gamma_T)$ is irreducible. The result follows.

\end{proof}

\section{Links with the Floer homology of twisted Whitehead links}\label{sec:HFLW}

Our first goal in this Section is to prove that link Floer homology detects most \emph{$n$-twisted Whitehead links}, which we denote $L_n$, the family of links shown in Figure~\ref{L_n}, where $n$ indicates the number of half-twists in the box region:

\begin{theorem}\label{thm:HFLLn}
    Suppose $\widehat{\HFL}(L)\cong\widehat{\HFL}(L_n)$, for some $n \notin \{-2,-1,0,1\}$. Then $L$ is $L_n$.
\end{theorem}

Note that $n$-twisted Whitehead links are exactly clasp-braids of index $2$ with unknotted axes. Twisted Whitehead links admit a number of symmetries: reversing the orientation of either component does not change the oriented link type. On the other hand if $n\geq 0$ then $\overline{L_n}=L_{-n-1}$, as unoriented links, where $\overline{L}$ denotes the mirror of $L$.

Our strategy for proving this result is to apply the main Theorem~\ref{thm:mainbraid}, and exclude every case but that in which $L$ consists of an unknot together with a two-standed clasp-braid -- i.e. twisted Whitehead links. This relies on some work of the second author and King, Shaw, Tosun and Trace~\cite{dey2023unknotted}. The result will then follow from the fact that link Floer homology distinguishes appropriate twisted Whitehead links, a fact we verify in Section~\ref{sec:Lncomp}. Indeed, we compute the link Floer homology of all twisted Whitehead links.

We can also treat the cases excluded from the statement of Theorem~\ref{thm:HFLLn}. Note that $L_0$ is the Whitehead link, L5a1, while $L_1$ is L7n2.

\begin{theorem}\label{thm:HFKW}
    Suppose $L$ is a two component link with $\widehat{\HFK}(L)\cong\widehat{\HFK}(L_0)$. Then $L$ is the Whitehead link ($L5a1$) or $L7n2$.
\end{theorem}

Here $\widehat{\HFK}(L)$ is the knot Floer homology of $L$, a link invariant due independently to Ozsv\'ath-Szab\'o and J.Rasmussen~\cite{Rasmussen,ozsvath2005knot}. The knot Floer homology of a link $L$ can be recovered from the link Floer homology of $L$ by collapsing the multi-Alexander grading to a single Alexander grading and shifting the Maslov grading up by $\frac{n-1}{2}$, where $n$ is the number of components of $L$.

Note that knot Floer homology distinguishes the Whitehead link from its mirror. Indeed it follows from the behaviour of link Floer homology under mirroring that if $L$ is a link with the link Floer homology of the mirror of the Whitehead link then $L$ is either the mirror of the Whitehead link or the mirror of $L7n2$.

Our strategy for proving Theorem~\ref{thm:HFKW} is to first show, in Theorem~\ref{thm:HFLwhite}, that if $L$ is a link with the same link Floer homology as the Whitehead link then $L$ is the Whitehead link or $L7n2$. To prove Theorem~\ref{thm:HFKW} we then show that if a link has the same knot Floer homology as the Whitehead link then it has the same link Floer homology as the Whitehead link.

\subsection{The link Floer homology of twisted Whitehead links}\label{sec:Lncomp}

In this section we show that link Floer homology detects most twisted Whitehead links. Our first goal is to compute the link Floer homology of the $n$-twisted Whitehead links.

First recall from~\cite[Section 12]{ozsvath2008holomorphic} that for $i=0$ or $1$ we have that the summands of $\widehat{\HFL}(L_0))$ or $\widehat{\HFL}(L_1))$ are given by:

  \begin{center}

    \begin{tabular}{|c|c|c|c|}
    \hline
    \backslashbox{\!$A_2$\!}{\!$A_1$\!}
     &-1& 0&1 \\
\hline
$1$&$\F_{0}$&$\F_1^2$&$\F_2 $\\\hline
$0$&$\F_{-1}^2$&$\F_{0}^4$&$\F_1^2$\\\hline
$-1$&$\F_{-2}$&$\F_{-1}^2$&$\F_{0}$\\\hline
  \end{tabular} 
\end{center}

 Throughout this section, we let $T_n$ denote the twist knot with $n$ many half-twists. Since $T_n$ is alternating, $\widehat{\HFK}(T_n)$ can be computed from the Alexander polynomial and signature of $T_n$. For $n$ odd we have that:
\[ \widehat{\HFK}(T_n, i) = \begin{cases}
\mathbb{F}_2^{\frac{n+1}{2}} &\text{if } i =1 \\
\mathbb{F}_{1}^n &\text{if }  i=0 \\
\mathbb{F}_{0}^{\frac{n+1}{2}} & \text{if }  i=-1
\end{cases} \]
While for $n$ even we have that:
\[ \widehat{\HFK}(T_n, i) = \begin{cases}
\mathbb{F}_1^{\frac{n}{2}} &\text{if } i =1 \\
\mathbb{F}_0^{n+1} &\text{if }  i=0 \\
\mathbb{F}_{-1}^{\frac{n}{2}} & \text{if }  i=-1
\end{cases} \]

\begin{proposition}
\label{HFLLn}
  When $n>1$ is odd, the summands of $\widehat{\HFL}(L_n)$ are given by:\begin{center}
      \begin{tabular}{|c|c|c|c|}
    \hline
    \backslashbox{\!$A_2$\!}{\!$A_1$\!}
     &$-1$& $0$&$1$ \\
\hline
$1$&$\F_{0}$&$\F_1^2$&$\F_2 $\\\hline
$0$&$\F_{-1}^{\frac{n+3}{2}}\oplus \F_{0}^{\frac{n-1}{2}}$&$\F_{0}^{n+2}\oplus \F_{1}^{n-2}$&$\F_{1}^{\frac{n+3}{2}}\oplus \F_{2}^{\frac{n-2}{2}}$\\\hline
$-1$&$\F_{-2}$&$\F_{-1}^2$&$\F_{0}$\\\hline
  \end{tabular} 
\end{center}

When $n>0$ is even, the summands of $\widehat{\HFL}(L_n)$ are given by:

\begin{center}
      \begin{tabular}{|c|c|c|c|}
    \hline
    \backslashbox{\!$A_2$\!}{\!$A_1$\!}
     &$-1$& $0$&$1$ \\
\hline
$1$&$\F_{-1}$&$\F_0^2$&$\F_1 $\\\hline
$0$&$\F_{-2}^{\frac{n+2}{2}}\oplus \F_{-1}^{\frac{n-2}{2}}$&$\F_{-1}^{n+3}\oplus \F_{0}^{n-1}$&$\F_{0}^{\frac{n+2}{2}}\oplus \F_{1}^{\frac{n-2}{2}}$\\\hline
$-1$&$\F_{-3}$&$\F_{-2}^2$&$\F_{-1}$\\\hline
  \end{tabular} 
\end{center}

\end{proposition}

Note in particular that:\begin{enumerate}
    \item The maximal $A_2$ grading is $1$.
    \item $\widehat{\HFL}(L_n)$ is of rank $4$ in the maximal $A_2$ grading.
    \item $\rank(\widehat{\HFK}(L_n)) = 4n+10$ for $n>1$ while $\rank(\widehat{\HFK}(L_n)) =16$ for $-2 \leq n \leq 1$, and $\rank(\widehat{\HFK}(L_n)) =-4n+6$ for $n < -2$. 
\end{enumerate} 

To prove Proposition~\ref{HFLLn} we apply Zibrowius' machinery of 4-ended tangle invariants~\cite{zibrowiuspeculiar}. This invariant assign to each $4$-ended tangle in a $3$-ball an immersed multicurve in the $2$-sphere with $4$ punctures. The Heegaard Floer homology of the union of two four ended tangles can then be computed as the lagrangian Floer homology of the pair of corresponding multi-curves, which in practice amounts to counting intersection points between the pair of multi-curves. We borrow notation from \cite{zibrowiuspeculiar}. 

\begin{proof}

We use the ungraded version of the pairing theorem from \cite[Corollary 0.7]{zibrowiuspeculiar} for the tangle invariants corresponding to the $(2,-2)$ pretzel tangle, $S_1$, (see \cite[Example 6]{zibrowius2020manual}) and the rational tangle of slope $-\frac{2}{2n+1}$, $S_2$. \[\gamma_{S_1} = \gamma_{Q_{\frac{1}{2}}} \oplus \gamma_{Q_{-\frac{1}{2}}} \oplus s_2(0) \oplus s_2(0), \gamma_{S_2} = \gamma_{Q_{-\frac{2}{2n+1}}}, \]
where $s_2(0)$ are the special components and the remaining components are rational with the specified slopes, see Figure~\ref{HFT(P(2,-2))}.

\begin{center}
    \begin{figure}[h]
    \centering
    \includegraphics[width=6.7in]{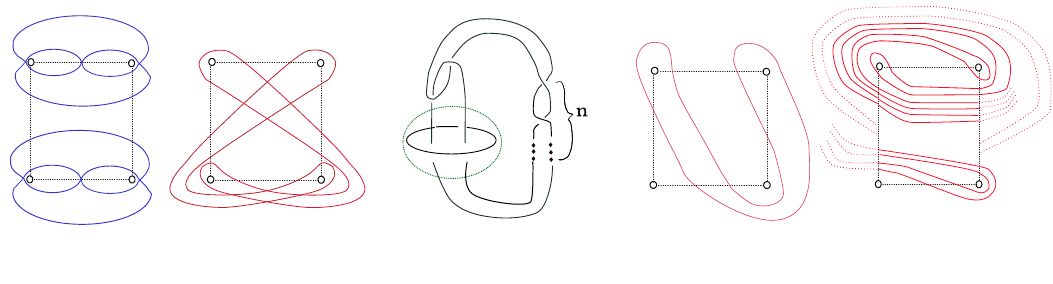}
    \caption{The two left-most figures are components of the tangle invariant of the pretzel tangle $P(2,-2)$, which is the tangle in the interior of the green sphere indicated in the picture of $L_n$. The two right-most figures are the tangle invariant of the rational tangle given by the $0-$twisted clasps and by the $n-$twisted clasps, where $n>0$, which is the tangle that lies on the exterior of the sphere.}
    \label{HFT(P(2,-2))}
\end{figure}
\end{center}

By \cite[Theorem 3.7]{zibrowiuspeculiar} we have that as ungraded vector spaces: 

\begin{align*}
    V \otimes \widehat{\HFK}(L_n) \cong  V \otimes \widehat{\HFK}(S_1 \cup S_2) &\cong \HF (mr(\gamma_{S_1}), \gamma_{S_2}) \\
=  \HF(\gamma_{Q_{\frac{1}{2}}}, \gamma_{Q_{-\frac{2}{2n+1}}}) \oplus \HF(\gamma_{Q_{-\frac{1}{2}}}, \gamma_{Q_{-\frac{2}{2n+1}}}) \\ \oplus \HF(-s_2(0),  \gamma_{Q_{-\frac{2}{2n+1}}}) \oplus  \HF(-s_2(0),  \gamma_{Q_{-\frac{2}{2n+1}}})
\end{align*}

\[V \otimes \widehat{\HFK}(L_n) = V \otimes \widehat{\HFK}(Q_{\frac{1}{2}} \cup Q_{-\frac{2}{2n+1}}) \oplus  V \otimes \widehat{\HFK}(Q_{-\frac{1}{2}} \cup Q_{-\frac{2}{2n+1}}) \oplus ((\mathbb{F}^4 \oplus \mathbb{F}^4) \otimes V) \]

Here $V$ denotes a two dimensional vector space. Thus we have, 

\begin{equation}
\label{HFTLn}
 V \otimes \widehat{\HFK}(L_n) \cong (V \otimes \widehat{\HFK}(T_{n-2})) \oplus (V \otimes \widehat{\HFK}(T_{n+2})) \oplus  ((\mathbb{F}^4 \oplus \mathbb{F}^4) \otimes V) 
\end{equation}

which follows by pairing the special component of the immersed curve invariant for $P(2,-2)$ tangle with the rational tangle, as shown in Figure~\ref{fig:HFTpairing}.

\begin{center}
\begin{figure}[h]
    \centering
    \includegraphics[width=5.4in]{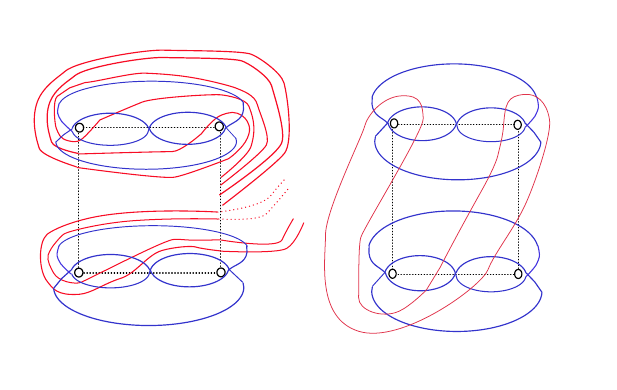}
    \caption{The left figure shows pairing the special component corresponding to $P(2,-2)$ with the rational tangle corresponding to the $n-$twisted clasp tangle. On the right is the diagram in the $n=0$ case. We only drew the portion of the pairing diagram for general $n$ in which the intersections occur.}
    \label{fig:HFTpairing}
\end{figure}

\end{center}

We now determine the Maslov gradings of the generators of $\widehat{\HFL}(L_n)$. Recall the skein exact sequence in knot Floer homology from \cite[Chapter 9]{gridhomologybook}. If a crossing $p$ in a diagram for a link $L'_+$ involves only one component of $L'_+$, then for every Alexander grading $i$ we have the following exact sequence connecting the knot Floer homology groups of $L'_+, L'_-, L'_0$:

\[\cdots \rightarrow \widehat{\HFK}(L'_-,i) \rightarrow \widehat{\HFK}(L'_0,i) \rightarrow \widehat{\HFK}(L'_+,i) \rightarrow \cdots\]

The maps to and from $L'_0$ decreases the Maslov grading by $\frac{1}{2}$, and the map from $L'_+$ to $L'_-$ preserves the Maslov grading.

We apply this Skein exact sequence to the highlighted crossing in Figure~\ref{L_n}. In this case $L'_+ = L_n, L'_-=L_{n-2}, L'_0 = Hopf^{-} \sqcup U$, where $U$ denotes the unknot.  Recall ${\widehat{\HFK}(Hopf^{-} \sqcup U)}$, is given by:

\[ \widehat{\HFK}(Hopf^{-} \sqcup U,i) = \begin{cases}
 \mathbb{F}_2 \oplus \mathbb{F}_1 & \text{if } i=1 \\  
 \mathbb{F}^2_1 \oplus \mathbb{F}_0 & \text{if } i=0 \\ 
  \mathbb{F}_0 \oplus \mathbb{F}_{-1} & \text{if } i=-1 
\end{cases}\]

Since $\widehat{\HFK}(L_0) = \widehat{\HFK}(W)$ is supported in Alexander grading $2$, and $\widehat{\HFK}(Hopf^{-} \sqcup U)$ does not have support at Alexander grading $2$, the skein exact sequence implies that for any $m \in \mathbb{Z}$,
\[\widehat{\HFK}(L_{2m}, 2) \cong \widehat{\HFK}(L_0, 2) \cong  \mathbb{F}_{\frac{3}{2}}.\]

\noindent Since both components of $L_n$ bound Euler characteristic $-1$ surfaces in the link exterior, it follows that the maximum support of $\widehat{\HFL}(L_n)$ is at $A_i = 1, i=1,2$, and also that \[\widehat{\HFL}(L_{2m}; (1,1)) \cong \mathbb{F}_1, \widehat{\HFL}(L_{2m}; (-1,-1)) \cong \mathbb{F}_{-3}, \] 
where the last equality follows from the symmetry of knot Floer homology: \[\widehat{\HFK}_d(L_n, -i) \cong \widehat{\HFK}_{d-2i}(L_n, i)\]
Consider the link $W'$ formed by reversing the orientation of the unknotted component of $L_0$ we get that $\widehat{\HFK}(W',2) \cong \mathbb{F}_{-\frac{1}{2}}$. Again using the above mentioned skein exact sequence and the symmetery mentioned, we get that 
\[\widehat{\HFL}(L_{2m}; (1,-1)) \cong \mathbb{F}_{-1}, \widehat{\HFL}(L_{2m}; (-1,1)) \cong \mathbb{F}_{-1}\]

Recall that there is a spectral sequence whose $E_1$ term in $\widehat{\HFL}(L_n)$, and whose $E_{\infty}$ term, ignoring the $A_{L_n - U}$ grading, is isomorphic to $\widehat{\HFK}(U) \otimes V$, where $V$ is a two dimensional vector space $\mathbb{F}_0 \oplus \mathbb{F}_{-1}$, since $\lk(U, L_n-U) = 0$.  We thus have that:
\begin{equation}
\label{evenHFL}
    \widehat{\HFL}(L_{2m}, A_U = 1) \cong \mathbb{F}_{1}[1,1] \oplus \mathbb{F}_0^2[0,1] \oplus \mathbb{F}_{-1}[-1,1]
\end{equation}

A similar statement holds for $\widehat{\HFL}(L_n, A_U = -1)$.

 Now notice that $T_{-n} \simeq m(T_{n-1}), n \geq 0$, where $m(K)$ denotes the mirror of $K$. Thus, starting with a specified orientation on both of the components of $L_{n-1}, n>0$ , one can obtain $L_{-n}$ by first taking mirror of $L_{n-1}$ and then changing the orientation of the unknotted component of the link obtained.  

Recall that link Floer homology satisfies the following symmetry properties:

\[\widehat{\HFL}_d(L; A_1, A_2) \cong \widehat{\HFL}_{2A_1 + 2A_2 - d + 1 - 2}(m(L); A_1, A_2) \]
\[ \widehat{\HFL}_d(L; A_1, A_2) \cong \widehat{\HFL}_{d-2A_2+\lk(L_1, L_2)}(L'; A_1, -A_2),\]
\noindent where  $m(L)$ is the mirror of $L$ and $L'$ is obtained from $L$ by reversing the orientation of one of its components $L_2$.

Using these symmetries, the skein exact sequence and Equation~\ref{evenHFL} we have that 

\begin{equation}
\label{oddHFL}
    \widehat{\HFL}(L_{2m+1}, A_U=1) \cong \mathbb{F}_{2}[1,1] \oplus \mathbb{F}_1^2[0,1] \oplus \mathbb{F}_{0}[-1,1].
\end{equation}

\noindent Indeed from Equation \ref{evenHFL} we can deduce that \[
    \widehat{\HFL}(L_{-2}, A_U = 1) \cong \mathbb{F}_{1}[1,1] \oplus \mathbb{F}_0^2[0,1] \oplus \mathbb{F}_{-1}[-1,1].\]
    \noindent Thus, \[
    \widehat{\HFL}(L_{1}, A_U = 1) \cong \mathbb{F}_{2}[1,1] \oplus \mathbb{F}_1^2[0,1] \oplus \mathbb{F}_{0}[-1,1].\]

We now determine the $A_U=0$ summand of $\widehat{\HFL}(L_n)$. Equation \ref{HFTLn} implies that 

\[\rank(\widehat{\HFL}(L_n, A_U = 0)) = \rank(\widehat{\HFK}(T_{n-2})) + \rank(\widehat{\HFK}(T_{n+2})) = 4n+2,\] 

which along with the existence of a spectral sequence to $\widehat{\HFK}(T_n)\otimes V$ and the Maslov grading information from Equations \ref{evenHFL} and \ref{oddHFL} implies that when $n>1$ is odd

\[\rank(\widehat{\HFL}(L_n; (0, \pm 1))) \geq n+1,\]
and that when $n>0$ is even
\[\rank(\widehat{\HFL}(L_n; (0, \pm 1))) \geq n. \]

We treat when $n>1$ is odd. The case when $n>0$ is even is similar.

Suppose $n>1$ is odd. Recall that $\widehat{\HFK}(T_n,1) \cong \mathbb{F}_2^{\frac{n+1}{2}}$. From the spectral sequence from $\widehat{\HFL}(L_n)$ to $\widehat{\HFL}(T_n)\otimes V$ we see that either:\begin{enumerate}
\item $\widehat{\HFL}(L_n,; (1,0)) \cong \F^{\frac{n+5}{2}}_1\oplus\F^{\frac{n+1}{2}}_2$ or;
\item $\widehat{\HFL}(L_n,; (1,0)) \cong \F^{\frac{n+3}{2}}_1\oplus\F^{\frac{n-1}{2}}_2$.
\end{enumerate}

In either case we have that $\widehat{\HFL}_i(L,(0,-1))\cong \widehat{\HFL}_{i-2}(L,(0,1))$ by the symmetry properties of $L_n$ and link Floer homology.

In case (1) we have that 
$$\rank(\widehat{\HFL}(L_n,; ( 1,0)))+\rank(\widehat{\HFL}(L_n,; (-1,0))) =2n+12,$$ hence, since $\rank(\widehat{\HFL}(L_n))=4n+12$, we have that $\rank(\widehat{\HFL}(L_n,; (0,0))) =2n-10$ and $\rank(\widehat{\HFL}(L_n,; A_1=0) =2n-6$. This is a contradiction, since $\widehat{\HFL}(L_n,; A_1=0)$ admits a spectral sequence to \[\widehat{\HFK}(T_n;0) \otimes V = \mathbb{F}_1^{n} \oplus \mathbb{F}_{0}^{n}.\]

Thus for $n>1$ odd, \[\widehat{\HFL}(L_n;(1,0)) \cong \mathbb{F}_{1}^{\frac{n+3}{2}} \oplus \mathbb{F}_{2}^{\frac{n-1}{2}}.\]
A similar argument can be used to determine $\widehat{\HFL}(L_n; (0,0))$, for $n >1$ odd, completing the computation in that case.

Finally that $\widehat{\HFL}(L_1)$ can be obtained from $\widehat{\HFL}(L_{-2})$ by using the symmetry properties of link Floer homology mentioned above. In general, for $n>0$ $\widehat{\HFL}(L_{-n})$ can also be obtained from $\widehat{\HFL}(L_{n-1})$ using the symmetry properties.

\end{proof}

To parse the following Lemma it is helpful to recall that link Floer homology detects the number of components of a link.
\begin{lemma}
\label{LnNonbraid}
    Suppose $L$ is a link with a genus one fibered component $K$ such that ${\widehat{\HFL}(L) \cong \widehat{\HFL}(L_n)}$ for some $n$. If the second component of $L$, $K'$, is isotopic to a curve in a genus one Seifert surface for the first component, $K$, then $L$ is isotopic to $L_{1}$, $L_{\pm2}$ or $L_{-3}$.
\end{lemma}

\begin{proof}

Suppose $L,K$ and $K'$ are as in the statement of the lemma. Note that since ${\widehat{\HFL}(L)\cong\widehat{\HFL}(L_n)}$, $L$ is fibered.  Moreover, $K$ is a genus one fibered knot by Theorem~\ref{thm:mainbraid}, Lemma~\ref{lem:homessentialdecomposed}. Thus, $K$ is $3_1$, $\overline{3_1}$ or $4_1$.

  \cite[Lemma 3.10]{juhasz2008floer} -- as interpreted in~\cite[Section 5]{binns2022cable}, for example --  implies that
\[\min \{A_i: \widehat{\HFL}(L,A_i) \not\cong 0 \} =\min\{ \frac{1}{2} \{-2g(S) - n(S, L-L_i): \partial S = L_1\}\},\]
where $n(S,L)$ is the minimal geometric intersection number of  $S$ with $L$. Since the minimal Alexander $A_K$ and $A_{K'}$ gradings in which $\widehat{\HFL}(L_n)$ is supported are $-1$, the Seifert genus of the knot $K'$ is 0 or 1.

Notice that Proposition \ref{HFLLn} and Corollary \ref{cor:HFLhomessential} imply that the surface framing of $K'$ with respect to $\Sigma_K$ is $\pm 1$, where $K' \subset \Sigma_K, \partial \Sigma_K = K, g(\Sigma_K) = 1$.

To study $K$ and $K'$ we use techniques developed by Dey-King-Shaw-Tosun-Trace, who study unoriented homologically essential simple closed curves on genus one Seifert surfaces of various genus one knots~\cite{dey2023unknotted}. Of course, each such curve can be endowed with two possible orientations. Unoriented unknotted curves come in two types: $(m,n)-\infty$ type curves and $(m,n)-loop$ type curves, see~\cite[Figure 2,3]{dey2023unknotted}, where in each case $\gcd(m,n)=1$ and $m,n \geq 0$. Here an $(m,n)-loop$ curve can be thought of as homologically $(m,0) + (0,n)$, and an $(m,n)-\infty$ curve is homologous to $(m,0) - (0,n)$ in an appropriate fixed basis for the homology of the Seifert surface.

First we treat the case that $g(K')=0$, i.e. that $K'$ is an unknot. The homologically essential unknotted curves on the genus one Seifert surfaces for $3_1$ and $4_1$ are characterized explicitly in~\cite[Theorem 1.1, 1.2]{dey2023unknotted}. Any genus one Seifert surface of $3_1$ contains exactly 6 unknotted curves (see~\cite[Figure 2]{dey2023unknotted}) up to isotopy fixing the Seifert surface set-wise, while any genus one Seifert surface of $4_1$ contains infinitely many, \cite[Theorem 1.1]{dey2023unknotted} up to isotopy fixing the Seifert surface set-wise. On a genus one Seifert surface of $3_1$, notice that only the $(1,0), (0,1)$ and the $(1,1)-\infty$ type unknotted curves have surface framing $\pm 1$. It can also be observed that these curves are related by the action of monodromy of $3_1$ -- $\begin{bmatrix}
    1 & -1 \\ 1 & 0 
\end{bmatrix}$  with respect to the basis $((1,0),(0,1))$ -- whence they are isotopic in the complement of $3_1$. Thus each of the links consisting of $3_1$ and one of the curves above, are isotopic to $L7n2$.

Now recall that the monodromy of $4_1$ acts on the homology of a fiber surface for $4_1$ acts by $\begin{bmatrix}
    2 & 1 \\ 1 & 1 
\end{bmatrix}$ with respect to 
the basis $((1,0), (0,1))$. The monodromy induces an isotopy of the exterior of the knot $4_1$. The unknots on the genus one Seifert surface of $4_1$ are parameterized by the Fibonacci sequence. These unknotted curves are isotopic to either $(1,0)$ or $(0,1)$ in the exterior of $4_1$, where the isotopy is given by multiple powers of the monodromy. Each of the links consisting of $4_1$ and $(1,0)$ or $(0,1)$ are isotopic to $L_2$.

We now proceed to the case that $K'$ is homologically essential and $g(K')=1$. We show that there is no essential closed curve of genus $1$ with surface framing $\pm 1$ on a genus 1 Seifert surface for $3_1$ or $4_1$. To do so we follow the case analysis in the proof of Proposition 1.2, 3.1, 3.3 and 3.4 of \cite{dey2023unknotted}.

We first treat the $3_1$ case. \cite[Proposition 3.1]{dey2023unknotted} states that homologically essential closed curves on a genus one Seifert surface of $3_1$ are isotopic to the closures of negative (classical) braids.  For $(m,n)-\infty$ and $(m,n)-loop$ curves $c$ with $m>n>0$, we have that $g(c) = \frac{m(m-n-2)+n^2+1}{2}$, and $g(c) = \frac{m(m+n-2) +n(n-2)+1}{2}$ respectively. The former is $1$ if and only if $m=3,n=1$, while the latter is $1$ if and only if $m=2, n=1$. For $(m,n)-loop$ curves $c$ with $n>m>0$, $g(c) = \frac{n(n-m-2)+m^2+1}{2}$ and for $(m,n)-\infty$ curves $c$ with $n>m>0$ we have $g(c) = \frac{n(n-1)+m(m-2)+n(m-1)+1}{2}$. The former is $1$ exactly when $m=1,n=3$, and the latter is $1$ exactly when $m=1, n=2$. The surface framing of a $(m,n)-loop$ curve on a fiber surface for $3_1$ is $-m^2-mn-n^2$ and that of a $(m,n)-\infty$ curve is $-m^2+mn-n^2$. Thus it is easy to see that neither of the above mentioned genus one curves have surface framing $\pm 1$.

We proceed now to the case of $4_1$. \cite[Proposition 3.3, 3.4]{dey2023unknotted} characterizes the homologically essential simple closed curves on a genus one Seifert surface for $4_1$. For our purpose, there are six cases to consider. The first pair of the cases are $(m,n)-loop$ curves with $m>n>0$, which are isotopic to negative braid closures,
and $(m,n)-\infty$ curves with $n>m>0$, which are isotopic to positive braid closures. For a curve $c$ of the former type, $g(c) = \frac{n(m-n)+(m-1)^2}{2}$, while for a curve of the latter type $g(c) = \frac{m(n-m)+(n-1)^2}{2}$. The former is 1 exactly when $m=2, n=1$ and the latter is 1 exactly when $n=2, m=1$. Since the surface framings of $(m,n)-loop$ type and $(m,n)-\infty$ type curves on a genus one Seifert surface for $4_1$ are $-m^2-mn+n^2$ and $-m^2+mn+n^2$ respectively, it is readily seen that the two genus one curves above do not have surface framing $\pm 1$. 

We now treat the remaining cases for $4_1$. When $m>n$, $(m,n)-\infty$ curve were divided into three cases, according to if $m-n = n, m-n >n$ or $m-n <n$. The curve is an unknot for the first case. For the second case, we have ${g(c) = \frac{n(m-2n)+(m-n-1)^2}{2}}$. Similarly for $(m,n)-loop$ curve when $n>m$ is divided into three subcase, whethere $n-m = m, n-m >m$ or $n - m<m$. The curve is an unknot for the first case. For the second case we have ${g(c) =  \frac{m(n-2m)+(n-m-1)^2}{2}}$.  The former is 1 when $m=3, n=1$, the latter is 1 when $m=1, n=3$. Again it can be seen readily that neither of the these curves have surface framings $\pm 1$.

The third subcase has further subcases divided into if $2n-m = m-n, 2n- m > m-n$ or $2n-m < m-n$, for an $(m,n)-\infty$ curve. The first and the third subcases only produce unknots, as proved in \cite[Lemma 3.5]{dey2023unknotted}, while for the second subcase an $(m,n)-\infty$ curve can be isotoped to an $(m-n, 2n-m)-\infty$ curve and is of genus ${\frac{(2n-m-1)^2 + (m-n)((2n-m)-(m-n))}{2}}$. The third subcase for $(m,n)-loop$ curve is also divided into three further subcases based on if $2m - n = n -m, 2m-n > n-m$ or $2m-n < n-m$. Again the same lemma implies that the first and the third case only produces unknots. For the second subcase, an $(m,n)-loop$ curve can be isotoped into an $(2m-n, n-m)-loop$ curve, which has genus  ${\frac{(2m-n-1)^2 + (n-m)((2m-n)-(n-m))}{2}}$. The former is $1$ when $2n-m = 2, m-n = 1$ and the latter is $1$ if $2m-n = 2, n-m =1$. It is easy to see again that none of these curves have surface framing $\pm 1$. 

For each of $3_1$ and $4_1$ there is a single non-homologically essential curve to analyse, namely a curve in the surface parallel to the boundary. However, these curves all have surface framing $0$, so we are done with this case.

Finally note that the case of $\overline{3_1}$ can be reduced to the case of $3_1$ by taking mirrors.

\end{proof}

\begin{lemma}\label{lem:twistedwhiteheaddetection}
    Suppose $L$ is a link with $\widehat{\HFL}(L)\cong\widehat{\HFL}(L_n)$ for some $n$. Then $L$ is a twisted Whitehead link.
\end{lemma}

\begin{proof}
Suppose $L$ is as in the statement of the Lemma.       Let $L = L_1 \cup L_2$. Observe that $L$ does not contain a split unknotted component since its link Floer homology polytope is non-degenerate. Without loss of generality, let \[\widehat{\HFL}(L, A_{L_1}) \cong  \mathbb{F}_m[1,1] \oplus \mathbb{F}_{m-1}^2[0,1] \oplus \mathbb{F}_{m-2}[-1,1],\] where $m \in \{1,2\}$ according to whether $n$ is odd or even. Corollary~\ref{rem:reducible} implies that $L$ could be one of three forms; a (stabilized) clasp-braid with its axis, a braid about an almost fibered knot or a genus one fibered knot with a simple closed curve in a Seifert surface. The Conway polynomial -- and hence link Floer homology -- detects the linking number of two component links by a result of Hoste \cite{hoste1985firstcoefficientoftheconwaypolynomial}. It follows that $\lk(L_1,L_2) = 0$. This rules out the case in which one of the component is an almost fibered knot and the other one its braid axis, since the linking number for such a link is non-zero. Lemma \ref{LnNonbraid} rules out the case in which, after a possible relabeling, $L_1$ is a genus one fibered knot and $L_2$ is a 
simple closed curve in a fiber surface of $L_1$. Thus $L_1$ is fibered and $L_2$ is a clasp or a stabilized clasp-braid with respect to $L_1$. $L_2$ cannot be a stabilized clasp braid with respect to $L_1$ as then the maximal non-trivial $A_{K_1}$ grading in $\widehat{\HFL}(L)$ would be at least $3$.
    
    To conclude, observe that $L_1$ must be an unknot since it has a longitudinal surface of Euler characteristic $-1$. Thus $L$ is of the form shown in Figure~\ref{L_n}. 
\end{proof}

We can now prove the main theorem of this section:

\begin{proof}[Proof of Theorem~\ref{thm:HFLLn}]
Suppose $L$ is as in the statement of the Theorem. By Lemma~\ref{lem:twistedwhiteheaddetection} $L$ is of the form $L_n$ for some $n$. Observe that $n\not\in\{-2,-1,0,1\}$ since $\widehat{\HFL}(L)$ is not of the correct rank. Now, up to mirroring and reversing the orientation of the components, the rank of $\widehat{\HFL}$ distinguishes each $L_n$, since $\rank(\widehat{\HFL}(L_n)) = 4n+10$ for $n>1$, ${\rank(\widehat{\HFL}(L_n)) = \rank(\widehat{\HFL}(L_{1-n}))=6-4n}$ for $n<-2$. Since the Maslov gradings distinguish $\widehat{\HFL}(L_{n})$ from $\widehat{\HFL}(L_{-n-1})$ this concludes the proof.
\end{proof}
 We can now deal with the remaining cases:
\begin{theorem}\label{thm:HFLwhite}
    Suppose $\widehat{\HFL}(L)\cong\widehat{\HFL}(W)$. Then $L$ is the Whitehead link or L7n2.
\end{theorem}

\begin{proof}
    Suppose $L$ is as in the statement of the Theorem. By Lemma~\ref{lem:twistedwhiteheaddetection} $L$ is a twisted Whitehead link. Observe that the only twisted Whitehead links with link Floer homology of rank 16 are the Whitehead link and $L7n2$, as desired.

\end{proof}

\subsection{Links with the same Knot Floer homology as the Whitehead link}

In this section we determine the links with the same knot Floer homology as the Whitehead link. We prove Theorem~\ref{thm:HFKW} by reducing it to reduce the detection result to a link Floer homology detection result.

\begin{proof}[Proof of Theorem~\ref{thm:HFKW}]
Let $L$ be as in the statement of the Theorem. By Theorem~\ref{thm:HFLwhite} it suffices to show that $\widehat{\HFL}(L)\cong\widehat{\HFL}(W)$.

 The Conway polynomial, and hence knot Floer homology, detects the linking number of two component links~\cite{hoste1985firstcoefficientoftheconwaypolynomial}. It follows that $\lk(L)=0$. It follows in turn that no component of $L$ is braided with respect to the other. It follows that in $\widehat{\HFL}(L)$ the rank in each maximal non-trivial Alexander grading must be at least $4$.

Recall that $\widehat{\HFK}(L)$ is $\delta$-thin. It follows that $\widehat{\HFL}(L)$ is $E_2$ collapsed. It follows from Ozsv\'ath-Szab\'o's classification of $E_2$ collapsed chain complexes that $\widehat{\HFL}(L)$ decomposes as the direct sum of one-by-one boxes~\cite[Section 12.1]{ozsvath2008holomorphic}. There must be four such boxes. Indeed, since $\widehat{\HFK}(L)\cong\widehat{\HFK}(W)$ we must have that one of these boxes is centered at $(a,1-a)$ for some $a\in\Z+\frac{1}{2}$, one of these is centred at $(b,-b)$ for some $b\in\Z+\frac{1}{2}$. The remaining two boxes are centred at $(-a,1-a)$ and $(-b,b)$ by the symmetry properties of link Floer homology. Since $\widehat{\HFL}(L)$ must be of rank at least $4$ in both maximal non-trivial Alexander gradings, we readily see that $a=\frac{1}{2}$, $b=\pm\frac{1}{2}$, concluding the proof.

\end{proof}

\section{Khovanov Homology Detection results}\label{sec:KH}

Khovanov homology is a link invariant due to Khovanov~\cite{khovanov2000categorification} which shares a number of structural properties with knot Floer homology. In this section we give two new detection results for Khovanov homology:

\begin{theorem}\label{thm:Khdetects}
    Khovanov homology detects the Whitehead link and L7n2.
\end{theorem}

These give examples of links which Khovanov homology detects but link Floer homology does not, addressing a question asked by the first author and Martin~\cite{binns2020knot}. We note in passing that it is also natural to ask the following question in the opposite direction:

\begin{question}
    Does there exist a pair of links which Khovanov homology cannot distinguish but which knot Floer homology can?
\end{question}

The authors are unaware of any such examples.

\subsection{A Review of Khovanov homology}

Let $L$ be a link, and $R$ be the ring $\Z,\Z_2$ or $\Q$. The \emph{Khovanov chain complex of $L$}, $(\CKh(L,R),\partial)$, is a finitely generated $\Z\oplus\Z$-filtered chain complex over $R$~\cite{khovanov2000categorification}.

$$\CKh(L;R):=\underset{i,j\in\Z}{\bigoplus}\CKh^{i,j}(L)$$

Here $i$ is called the \textit{homological grading}, while $j$ is called the \textit{quantum grading}. The filtered chain homotopy type of $L$ is an invariant of $L$. The $R$-module $\CKh(L,R)$ has generators corresponding to decorated resolutions of $D$, while $\partial$ is determined by a simple TQFT. The parity of the $j$ gradings in which $\Kh(L)$ has non-trivial support agrees with the parity of the number of components of $L$. The \textit{Khovanov homology of $L$} is obtained by taking the homology of $\CKh(L;R)$. A choice of basepoint $p\in L$ induces an action on $\CKh(L;R)$, which commutes with the differential. Taking the quotient of $\CKh(L;R)$ by this action and taking homology with respect to the induced differential yields a bigraded $R$-module called the~\textit{reduced Khovanov homology} of $L$, which is denoted $\widetilde{\Kh}(L,p;R)$~\cite{khovanov2003patterns}. Given a collection of points $\mathbf{p}=\{p_1,p_2,\dots,p_k\}\subset L$, there is a generalization of reduced Khovanov homology called \textit{pointed Khovanov homology}, $\Kh(L,\mathbf{p};R)$, due to Baldwin-Levine-Sarkar~\cite{baldwin2017khovanov}.

$\Kh(L;R)$ admits a number of useful spectral sequences. Suppose $L$ has two components $L_1,L_2$. If $R$ is $\Q$ or $\Z_2$ then $\Kh(L;R)$ admits a spectral sequence to $\Kh(L_1;R)\otimes\Kh(L_2;R)$, called the \emph{Batson-Seed spectral sequence}~\cite{batson2015link}. Indeed this spectral sequences respects the $i-j$ grading on $\Kh(L;R)$ in the sense that:
\begin{align}\label{GradedBatsonSeed}
    \rank^{i-j=l}(\Kh(L;R))\geq \rank^{i-j=l+2\lk(L_1,L_2)}(\Kh(L_1;R)\otimes\Kh(L_2;R))
\end{align}

 There is another spectral sequence called the \emph{Lee spectral sequence} from $\Kh(L;\Q)$ which abuts to an invariant called the \emph{Lee Homology} of $L$, $\LLL(L):=\bigoplus_i\LLL^i(L)$~\cite{lee2005endomorphism}. This spectral sequence respects the $i$ gradings in the sense that $\rank(\Kh^i(L;\Q))\geq \rank(\LLL^i(L))$. Lee showed that $\LLL(L)\cong\bigoplus_{i=1}^n\Q^2_{a_i}$ where $a_i$ are integers and $L$ has $n$ components. Indeed, if $L$ has two components $L_1,L_2$ then $\LLL(L)\cong\Q^2_{0}\oplus\Q^2_{\lk(L_1,L_2)}$~\cite[Proposition 4.3]{lee2005endomorphism}.

Finally, there is a spectral sequence due to Dowlin~\cite{dowlin2018spectral}, relating Khovanov homology and knot Floer homology. If $L$ is a link and $\mathbf{p}\subseteq L$, with exactly one element of $\mathbf{p}$ on each component of $L$, then there is a spectral sequence from $\Kh(L,\mathbf{p};\Q)$ to $\widehat{\HFK}(L;\Q)$ that respects the relative $\delta$-gradings. Here $\widehat{\HFK}(L;\Q)$ uses the coherent system of orientations given in~\cite{alishahi2015refinement}. We use this version of knot Floer homology, and the corresponding link Floer homology for the remainder of this section. As a corollary, Dowlin shows that if $L$ has $n$ components then: \begin{align}
    2^{n-1}\rank(\widetilde{\Kh}(L;\Q))\geq \rank(\widehat{\HFK}(L;\Q))\label{DowlinReduced}\end{align}

\subsection{Khovanov Homology detects the Whitehead link}
In this Section we show that Khovanov homology detects $L_0$. The idea is to reduce the classification of links with the  Khovanov homology type of the Whitehead link to the classification of  links with the knot Floer homology type of the Whitehead link, appeal to Theorem~\ref{thm:HFKW} and notice that the Whitehead link and L7n2 have distinct Khovanov homologies.

The Khovanov Homology of the Whitehead link is given as follows;
\begin{center}

    \begin{tabular}{|c|c|c|c|c|c|c|c|}
    \hline
    \backslashbox{\!$q$\!}{\!$h$\!}
     &-3& -2&-1&0&1&2 \\
\hline
4& &&&&&$\Z$\\\hline
2&&&&& &$\Z/2$\\\hline
0&&&&$\Z^2$ & $\Z$&\\\hline
-2&&&$\Z$&$\Z^2$&&\\\hline
-4& &$\Z$&$\Z/2$&&&\\\hline
-6&&$\Z\oplus\Z/2$&&&&\\\hline
-8&$\Z$&&&&&\\\hline
  \end{tabular} 
\end{center}

Observe that $\Kh(L;\Q)$ and $\Kh(L;\Z/2)$ are determined by $\Kh(L;\Z)$ by the universal coefficient theorem.
\begin{lemma}\label{whiteheadlinking}
Suppose $L$ is a link such that $\Kh(L;\Q)\cong\Kh(W;\Q)$. Then $L$ has two components with linking number $0$.
\end{lemma}

\begin{proof}
   Suppose $L$ is as in the statement of the Theorem. Consider the spectral sequence from $\Kh(L;\Q)$ to the Lee homology of $L$, $\LLL(L;\Q)$. Recall that the Lee homology of an $n$ component link is given by $\bigoplus_{1\leq i\leq n}\Q^2_{n_i}$ where $n_i$ indicates the homological grading. Note that homological gradings $-2$ and $0$ are the only two homological gradings in which $\Kh(L;\Q)$ is rank at least two. Moreover, $\Kh_{-2}(L)\cong\Q^2$, $\Kh_{0}(L)\cong\Q^4$. It follows that $L$ can have at most three components. Since the quantum grading takes value in the even integers, it follows that $L$ in fact has two components.

    To see that the two components of $L$ have linking number $0$ observe that a generator of $\Kh(L;\Q)$ of $(h,q)$-grading $(-3,-8)$ cannot persist under the spectral sequence to $\LLL(L;\Q)$. Since the Lee homology differential lowers the homological grading by one and $\rank(\Kh_{-2}(L))=2$, it follows that $\rank(\LLL_{-2}(L))=0$. Thus $\LLL(L)$ is supported in homological grading $0$, whence it follows that the linking number of $L$ is zero.
\end{proof}

\begin{lemma}\label{L5L7components}
Suppose $L$ is a two component link with $\lk(L)=0$ and ${\rank(\Kh(L;\Z/2))=16}$. Then $L$ has one unknotted component and another component which is either an unknot or a trefoil.
\end{lemma}

\begin{proof}
   Recall that the Batson-Seed spectral sequence yields a rank bound $$16=\rank(\Kh(L;\Z/2)\geq \rank(\Kh(L_1;\Z/2))\cdot\rank(\Kh(L_2;\Z/2)$$ where $L_1$ and $L_2$ are the two components of $L$. Now, $2\cdot\rank(\widetilde{\Kh}(K;\Z/2))=\rank(\Kh(K;\Z/2))$ for any $K$ by a result of Shumakovitch~\cite[Corollaries 3.2 B-C]{shumakovitch2014torsion}, so by an application of the universal coefficient theorem we find that ${4\geq\rank( \widetilde{\Kh}(L_1;\Q))\cdot\rank(\widetilde{\Kh}(L_2;\Q))}$. Applying the rank bound from Dowlin's spectral sequence we find that \[{4\geq \rank(\widehat{\HFK}(L_1;\Q))\cdot\rank(\widehat{\HFK}(L_2;\Q)}.\] But the only knots with knot Floer homology of rank less than $4$ are the trefoils and the unknot as desired.
\end{proof}

\begin{lemma}\label{lem:L5L7components2}
    Suppose $L$ is a link such that $\Kh(L;\Z)\cong\Kh(W;\Z)$. Then $L$ has unknotted components.
\end{lemma}

\begin{proof}
    Suppose $L$ is as in the statement of the Lemma. By Lemma~\ref{whiteheadlinking} $L$ has two components with linking number zero. Let $K$ denote the potentially unknotted component. Applying Lemma~\ref{L5L7components} we have that one componen is an unknot and the other component is either a trefoil or unknot.

    The refined version of Batson-Seed's spectral sequence yields the following inequality;
    $\rank^l(\Kh(L;\Q))\geq \rank^l(\Kh(U;\Q)\otimes\Kh(K;\Q))$ where $U$ is the unknot and $l$ is the grading $h-q$. Observe that if $K$ is the right handed trefoil then $\Kh(U;\Q)\otimes\Kh(K;\Q)$ has a generator in grading $-7$. However $\Kh(L;\Q)$ does not have a generator of $l$ grading $-7$ so $K$ cannot be the right handed trefoil.

    Observe that if $K$ is the left handed trefoil then $\Kh(U;\Q)\otimes\Kh(K;\Q)$ has a generator in $l$ grading $7$, while $\Kh(L;\Q)$ does not. It follows that $K$ also cannot be the left handed trefoil and is thus the unknot.

\end{proof}

\begin{lemma}\label{lem:KhWtoHFL}
Suppose $L$ is a two component link with an unknotted components $U_1$ and $U_2$ such that $\Kh(L;\Z/2)$ is delta-thin and of rank $16$. Then the maximal non-trivial $A_{i}$ gradings in $\widehat{\HFL}(L)$ are $1$ and the rank in these gradings are both $4$.
\end{lemma}

\begin{proof}
Suppose $L$ is as in the statement of the Theorem. We have that $\widehat{\HFK}(L;\Q)$ is also $\delta$-thin and of rank $16$. It follows that $\widehat{\HFL}(L;\Q)$ is $E_2$ collapsed. $\widehat{\HFL}(L;\Q)$ thus splits as a direct sum of four $1\times 1$ boxes by~\cite[Section 6]{binns2022rank}. Since the linking number of the two component of $L$ is zero, we have that neither component is braided about the other. We thus have that the rank in the maximal $A_i$ grading is at least four for each $i$ by~\cite[Proposition 1]{martin2022khovanov}. It follows that at least two of the boxes contain generators in each of the maximal $A_i$ gradings.

Suppose three or more of the boxes contain generators in the maximal $A_1$ grading. By the symmetry of link Floer homology we must have that the maximal $A_1$ grading is $\frac{1}{2}$. But the $A_1$ grading is $\Z+\frac{\lk(L_1,L_2)}{2}$-valued, so we have a contradiction since the linking number is zero.

Thus exactly two boxes contain generators in maximal $A_1$ grading. Similarly there are exactly two boxes which contain generators in the maximal $A_2$ grading. Moreover there must exist generators in $A_1$ and $A_2$ grading $0$, since $L$ has an unknotted component, the linking number is $0$ and there is a spectral sequence from $\widehat{\HFL}(L)$ to $\widehat{\HFL}(U)\otimes V$. It follows then that the maximal $A_1$ or $A_2$ gradings are both $1$, and we have the desired result.
\end{proof}

We can now conclude our proof that Khovanov homology detects $L_0$.

\begin{proof}
    Suppose $L$ is a link with $\Kh(L;\Z)\cong\Kh(W;\Z)$. $\rank(\Kh(L;\Z/2))=16$ by the universal coefficient theorem.  It follows that $L$ has unknotted components by Lemma~\ref{lem:L5L7components2}. Observe that $\Kh(L;\Z/2)$ is delta thin, so we can apply Lemma~\ref{lem:KhWtoHFL} to deduce that $\widehat{\HFL}(L)$ has rank $4$ in the maximum non-trivial Alexander grading corresponding to either component and that these gradings are both one. We also have that $L$ is not split, since the Khovanov homology of the two component unlink does not agree with $\Kh(W)$. We can thus apply Theorem~\ref{thm:mainbraid} to deduce that $L$ is of the form $L_n$ for some $n$. The only links with two unknotted components are $L_0$ and its mirror $\overline{W}$. These two links are distinguished by their Khovanov homologies, so the result follows.
\end{proof}

\subsection{Khovanov Homology Detects L7n2}

The Khovanov Homology of L7n2 is given as follows;
\begin{center}

    \begin{tabular}{|c|c|c|c|c|c|c|c|}
    \hline
    \backslashbox{\!$q$\!}{\!$h$\!}
     &-5&-4&-3& -2&-1&0 \\
\hline
0& &&&&&$\Z^2$\\\hline
-2&&&&& $\Z$&$\Z^2$\\\hline
-4&&&&$\Z$ & $\Z/2$&\\\hline
-6&&&&$\Z\oplus\Z/2$&&\\\hline
-8& &$\Z$&$\Z$&&&\\\hline
-10&&$\Z/2$&&&&\\\hline
-12&$\Z$&&&&&\\\hline

  \end{tabular} 
\end{center}
Following our notation in Section~\ref{sec:HFLW} we let $L_{-2}$ denote L7n2.

Observe that $\Kh(L;\Q)$ is determined by $\Kh(L;\Z)$ by the universal coefficient theorem.
\begin{lemma}\label{L7linking}
Suppose $L$ is a link such that $\Kh(L;\Q)\cong\Kh(L_{-2};\Q)$. Then $L$ has two components with linking number $0$.
\end{lemma}

\begin{proof}
   Suppose $L$ is as in the statement of the Theorem. Consider the spectral sequence from $\Kh(L;\Q)$ to the Lee homology of $L$, $\LLL(L;\Q)$. Observe that $\rank(\LLL(L;\Q))=2^n$ where $n$ is the number of components of $L$. Recall too that the Lee homology of an $n$ component link is given by $\bigoplus_{1\leq i\leq n}\Q^2_{n_i}$ where $n_i$ indicates the homological grading. Note that homological gradings $-2$ and $0$ are the only two homological gradings in which $\Kh(L;\Q)$ is rank at least two. Moewovwe $\Kh_{-2}(L)\cong\Q^2$, $\Kh_{0}(L)\cong\Q^4$. It follows that $L$ can have at most three components. Since the quantum grading takes value in the even integers, it follows that $L$ in fact has two components.

    We now show that the two components of $L$ have linking number $0$. Let $x$ be an element of $\Kh(L;\Q)$ in grading $(-3,-8)$. Let $\partial_*$ be the differential induced on $\Kh(L;\Q)$ by the Lee differential. Observe that $x$ cannot be a coboundary  of $\partial_*$ as there is no generator of homological grading $-3$ with quantum grading strictly less that $-8$. Since $x$ cannot persist under the spectral sequence it follows that $\partial x\neq 0$, whence $\rank(\LLL_{-2}(L;\Q))<2$, whence $\LLL(L)$ must be supported entirely in homological grading zero. It follows that the linking number is zero as required.
\end{proof}

\begin{lemma}\label{lem:Khl2components}
    Suppose $L$ is a link such that $\Kh(L;\Z)\cong\Kh(L_{-2};\Z)$. Then $L$ has an unknotted component and a left handed trefoil component.
\end{lemma}

\begin{proof}
    Suppose $L$ is as in the statment of the Lemma. By Lemma~\ref{L7linking} $L$ has two components with linking number zero.  Applying Lemma~\ref{L5L7components} we have that one component is an unknot and the other component is either a trefoil or unknot. Let $K$ denote the potentially unknotted component.

    The refined version of Batson-Seed's spectral sequence yields the following inequality;
    $\rank^l(\Kh(L;\Q))\geq \rank^l(\Kh(U;\Q)\otimes\Kh(K;\Q))$ where $U$ is the unknot and $l$ is the grading $h-q$. Observe that If $K$ is the right handed trefoil then $\Kh(U;\Q)\otimes\Kh(K;\Q)$ has a generator in grading $-7$. However $\Kh(L;\Q)$ does not have a generator of $l$ grading $-7$ so $K$ cannot be the right handed trefoil.

    Observe that If $K$ is the unknot then $\Kh(U;\Q)\otimes\Kh(K;\Q)$ has a generator in $l$ grading $2$, while $\Kh(L;\Q)$ does not. It follows that $L$ also cannot be the unknot and is thus the left handed trefoil.

\end{proof}

\begin{lemma}\label{lem:KhLtoHFL}
Suppose $L$ is a two component link with an unknotted component $U$ and a left handed trefoil component $T(2,-3)$ such that $\Kh(L;\Z/2)$ is delta-thin and of rank $16$. Then the maximal non-trivial $A_{i}$ gradings in $\widehat{\HFL}(L)$ are $1$ and the rank in these gradings are both $4$.
\end{lemma}

\begin{proof}
Suppose $L$ is as in the statement of the Theorem. We have that $\widehat{\HFK}(L;\Q)$ is also $\delta$-thin and of rank $16$. It follows that $\widehat{\HFL}(L;\Q)$ is $E_2$ collapsed. $\widehat{\HFL}(L;\Q)$ thus splits as a direct sum of four $1\times 1$ boxes by~\cite[Section 6]{binns2022rank}. Since the linking number of the two component of $L$ is zero, we have that neither component is braided about the other. We thus have that the rank in the maximal $A_i$ grading is at least four for each $i$ by~\cite[Proposition 1]{martin2022khovanov}. It follows that at least two of the boxes contain generators in each of the maximal $A_i$ gradings.

Suppose three or more of the boxes contain generators in the maximal $A_1$ grading. By the symmetry of link Floer homology we must have that the maximal $A_1$ grading is $\frac{1}{2}$. But the $A_1$ grading is $\Z+\frac{\lk(L_1,L_2)}{2}$-valued, so we have a contradiction since the linking number is zero.

Thus exactly two boxes contain generators in maximal $A_1$ grading. Similarly there are exactly two boxes which contain generators in the maximal $A_2$ grading. Moreover there must exist generators in $A_1$ and $A_2$ grading $0$, since both the unknot and the knot Floer homology of both the left handed trefoil and the unknot have generators of Alexander grading $0$, the linking number is $0$ and there are spectral sequences from $\widehat{\HFL}(L)$ to $\widehat{\HFL}(U)\otimes V$ and $\widehat{\HFL}(T(2,-3))\otimes V$. It follows then that the maximal $A_1$ and $A_2$ gradings are both $1$, and we have the desired result.
\end{proof}

We now prove that Khovanov homology detects $L7n2$
\begin{proof}
    Suppose $L$ is a link with $\Kh(L;\Z)\cong\Kh(L_{-2};\Z)$. Observe that $L$ has an unknotted component and left handed trefoil component by Lemma~\ref{lem:Khl2components}. Observe that $\Kh(L;\Z/2)$ is delta thin, so we can apply Lemma~\ref{lem:KhLtoHFL} to deduce that $\widehat{\HFL}(L)$ has rank $4$ in the maximum non-trivial Alexander grading corresponding to the unknotted component of $L$ and that this Alexander grading is one. We also have that $L$ is not split, since the Khovanov homology of the split sum of a left handed trefoil and an unknot does not agree with $\Kh(L_{-2})$. We can thus apply Theorem~\ref{thm:mainbraid} to deduce that $L$ is of the form $L_n$ for some $n$. The only link $n$ for which $L_n$ has a left handed trefoil component is $L_{-2}$, so we have the desired result.
\end{proof}

\section{Annular Khovanov homology detection results}\label{sec:annular}

Annular Khovanov homology is a version of Khovanov homology for links in the thickened annulus due to Aseda-Przytycki-Sikora~\cite{asaeda_categorification_2004}, see also~\cite{roberts_knot_2013,grigsby_sutured_2014}. It takes value in the category of $\Z\oplus\Z\oplus\Z$-graded $\Z$-modules. The first grading is called the \emph{homological grading}, the second is called the \emph{quantum grading} and the final grading is called the \emph{annular grading}. In this Section we prove the following annular Link detection results:

\begin{theorem}\label{thm:AKH}
 Annular Khovanov homology detects each of the family of annular knots shown in Figure~\ref{annularlinks} amongst annular knots.
\end{theorem}

Our strategy is to apply an adapted version of Theorem~\ref{thm:mainbraid} to deduce that if $K$ is an annular knot with the annular Khovanov homology type of $L_n$ then $K$ is of the form $L_m$ for some $m$. We will prove this in Section~\ref{sec:AKHdetectproof}. It will thus suffice to show that annular Khovanov homology can distinguish between the links $\{L_n\}_{n\in\Z}$. We will prove this by computing $\AKh(L_n)$ for all $n$ in Section~\ref{sec:AKHcomputations}.

\subsection{Annular Khovanov homology Computations}~\label{sec:AKHcomputations}
We compute the annular Khovanov homology of clasp-braids $L_n$ shown in Figure~\ref{annularlinks}. We begin by stating and, for the sake of completeness, proving the exact triangles in annular Khovanov homology. We note however, that these results are essentially the same as those given in the Khovanov homology case, and have indeed been used in the annular Khovanov homology case in~\cite{grigsby_annular_2018}, for example. Let $[a]$ denote a shift in the homological grading of a vector space by $a$ and $\{b\}$ denote a shift in the quantum grading of a vector space by $b$. For a link diagram $D$ let $n_+$ denote the number of positive crossings in $D$ and $n_-$ denote the number of negative crossings in $D$.

\begin{lemma}\label{lem:AKhexact}
   Let $c$ be a crossing in a diagram $D$ for a link $L$. Let $L_0$ and $L_1$ be the links with diagrams given by the $0$ and $1$ resolutions of $D$ respectively. If $c$ is a positive crossing then we have a short exact sequence:
\begin{center}
\begin{tikzcd}
\AKh(L)\arrow[rr]&&\AKh(L_0)\{1\}\arrow[dl]\\&\AKh(L_1)[n_-^1-n_-+1]\{3n_-^1-3n_-+2\}\arrow[ul]
\label{eq:akhskeinpos}
\end{tikzcd}
\end{center}

If $c$ is a negative crossing then we have an exact triangle:

\begin{center}
\begin{tikzcd}
    \AKh(L)\arrow[rr]&&\AKh(L_0)[n_-^0-n_-]\{3n_-^0-3n_-+1\}\arrow[dl]\\&\AKh(L_1)\{-1\}\arrow[ul]
\label{eq:akhskeinneg}
\end{tikzcd}
\end{center}
    
Here in both triangles the connecting homomorphism preserves the annular and quantum gradings and increases the homological grading by one. The other homomorphisms preserve all gradings.
    
\end{lemma}

The superscripts on the crossing numbers indicate that we are considering the crossing number in the diagram $D_i$ obtained by taking the $i=0$ or $i=1$ resolution of $D$ at $c$.

\begin{proof}
Let $c$ be a positive crossing in a diagram $D$ for a link $L$. Note that $D_0$ inherits an orientation from $D$. Let $C$ denote the pre-shift annular Khovanov chain complex. Then:
\begin{equation*}
    C(D)\cong\Cone(C(D_0)\to C(D_1)[1]\{1\})
\end{equation*}

Where the map in the mapping cone increases the homological grading by one and preserves the quantum grading. It follows that there is an exact triangle:

\begin{center}
\begin{tikzcd}
    H(D)\arrow[rr]&&H(D_0)\arrow[dl]\\&H(D_1)[1]\{1\}\arrow[ul]
\end{tikzcd}
\end{center}

We can rewrite the exact triangle above as:

\begin{center}
\begin{tikzcd}
    H(D)[-n_-]\{n_+-2n_-\}\arrow[r]&H(D_0)[-n_-]\{n_+-1-2n_-\}\{1\}\arrow[d]\\&H(D_1)[-n_-^1]\{n_+^1-2n_-^1\}[n_-^1-n_-]\{2n_-^1-n_+^1+n_+-2n_-\}[1]\{1\}\arrow[ul]
\end{tikzcd}
    \end{center}
Note that $\AKh(L):=H(D)[-n_-]\{n_+-2n_-\}$. Moroever, $n_+^0=n_+-1$, $n_-^0=n_-$, and $n_-^1+n_+^1=n_-+n_+-1$, so we have in turn have the first exact triangle in the statement of the Lemma.

On the other hand, if $L_-$ is a link with a negative crossing $c$ then $D_1$ inherits an orientation and we have that;

\begin{center}
\begin{tikzcd}
    H(D)[-n_-]\{n_+-2n_-\}\arrow[r]&H(D_0)[-n_-^0]\{n_+^0-2n_-^0\}[n_-^0-n_-]\{n_+-2n_--n_+^0+2n_-^0\}\arrow[d]\\&H(D_1)[1-n_-]\{n_+-2n_-+2\}\{-1\}\arrow[ul]
\end{tikzcd}

\end{center}

But $n_-^1=n_--1$, $n_+^1=n_+$, and $n_+^0+n_-^0=n_-+n_+-1$, so we have in turn the desired result.

\end{proof}

We apply these skein exact triangles to compute the annular Khovanov homology of $L_n$ for all $n$. Set $V_m$ to be the $(m+1)$-dimensional $(j,k)$ graded vector space, with dimension $1$ in each annular gradings between $-m$ and $m$ that shares the same parity as $m$. $V_m$ is supported in $j-k$ grading $0$. The $j$ grading of the minimum $k$ grading generator of $V_m$ is $-m$. 
\begin{lemma}\label{lem:2nearbraidcomp}
   Let  $n\geq 1$. If $n$ is even then $L_{n}$ has annular Khovanov homology given by:

  \begin{align}
\AKh^i(L_{n})\cong\begin{cases} V_2\{2n-1\}&\text{for } i=n\\
V_2\{2n-3\}\oplus V_0\{2n-3\}&\text{for } i=n-1\\
V_0\{2i+1\}\oplus V_0\{2i-1\}&\text{if } 1\leq i\leq n-2\\
V_0\{1\}\oplus V_0\{-1\}\oplus V_0\{-1\}&\text{if }i=0\\
V_0\{-1\}&\text{if }i=-1\\
V_0\{-5\}&\text{if }i=-2\\
0&\text{else.}
\end{cases}
\end{align}

If $n$ is odd then we have that:
\begin{align}
\AKh^i(L_{n})\cong\begin{cases}V_2\{2n+5\}&\text{for } i=n+2\\
V_2\{2n+3\}\oplus V_0\{2n+3\}&\text{for } i=n+1\\
V_0\{2i+1\}\oplus V_0\{2i+3\}&\text{if }2\leq i\leq n\\
V_0\{3\}&\text{if }i=1\\
V_0\{3\}\oplus V_0\{1\}&\text{if }i=0\\
0&\text{else.}\end{cases}
\end{align}
\end{lemma}

For the proof we let $\sigma^n$ denote the closure of the two stranded braid with $n$ right handed half twists.

\begin{proof}
    Suppose $n$ is even. Then the highlighted crossing in Figure~\ref{annularlinks} is negative and by Lemma~\ref{lem:AKhexact} we have an exact sequence:

\begin{center}
\begin{tikzcd}
    \AKh(L_{n})\arrow[rr]&&\AKh(\sigma^{-n-1})[n_-^0-n_-]\{3n_-^0-3n_-+1\}\arrow[dl]\\&\AKh((\sigma^{-n})')\{-1\}\arrow[ul]
\end{tikzcd}
\end{center}  

where $(\sigma^{-n})'$ is $\sigma^{-n}$ with the orientation of a strand reversed. Now, $n_-^0-n_-=n-1$, so we have:

\begin{center}
\begin{tikzcd}
    \AKh(L_{n})\arrow[rr]&&\AKh(\sigma^{-n-1})[n-1]\{3n-2\}\arrow[dl]\\&\AKh((\sigma^{-n})')\{-1\}\arrow[ul]
\end{tikzcd}
\end{center}

Now, Grigsby-Licata-Wehrli computed $\AKh(\sigma^{-n})$ for all $n$~\cite[Section 9.3]{grigsby_annular_2018}, where if $n$ is even the two strands are oriented in parallel. We recall the result here for the reader's convenience. For all $n\geq 1$ we have that:

\begin{align}\label{eq:akhsigma}
\AKh^i(\sigma^{-n})\cong\begin{cases} V_2\{-n\}&\text{for } i=0\\
V_0\{2i-n\}&\text{if }-n\leq i\leq -1\text{ and }i\text{ is odd}\\
V_0\{2i+2-n\}&\text{if }-n+1\leq i\leq -2\text{ and }i\text{ is even}\\
V_0\{2-3n\}\oplus V_0\{-3n\}&\text{if }i=-n\text{ and }n\text{ is even}\\
0&\text{else.}\end{cases}
\end{align}

Observe that if $n$ is even then Equation~\ref{eq:akhsigma} implies that:

\begin{align}
\AKh^i(\sigma^{-n-1})\cong\begin{cases} V_2\{-n-1\}&\text{for } i=0\\
V_0\{2i-n-1\}&\text{if }-n-1\leq i\leq -1\text{ and }i\text{ is odd}\\
V_0\{2i+1-n\}&\text{if }-n\leq i\leq -2\text{ and }i\text{ is even}\\
0&\text{else.}\end{cases}
\end{align}

so in turn:

\begin{align}
\AKh^i(\sigma^{-n-1})[n-1]\{3n-2\}\cong\begin{cases} V_2\{2n-3\}&\text{for } i=n-1\\
V_0\{2i-1\}&\text{if }-2\leq i\leq n-2\text{ and }i\text{ is even}\\
V_0\{2i+1\}&\text{if }-1\leq i\leq n-3\text{ and }i\text{ is odd}\\
0&\text{else.}\end{cases}
\end{align}

On the other hand, the grading shift formula in annular Khovanov homology implies that $\AKh((\sigma^{-n})')\cong\AKh(\sigma^{-n})[n]\{3n\}$, so that $\AKh^i((\sigma^{-n})')\{-1\}$ is given by:

\begin{align}\label{eq:akh}
\AKh(\sigma^{-n})[n]\{3n-1\}\cong\begin{cases} V_2\{2n-1\}&\text{for } i=n\\
V_0\{2i-1\}&\text{if }0\leq i\leq n-1\text{ and }i\text{ is odd}\\
V_0\{2i+1\}&\text{if }1\leq i\leq n-2\text{ and }i\text{ is even}\\
V_0\{1\}\oplus V_0\{-1\}&\text{if }i=0\\
0&\text{else.}\end{cases}
\end{align}

We then see that the exact sequence splits and the desired result follows.

Now consider the case in which $n$ is odd. In this case the highlighted crossing in Figure~\ref{annularlinks} is positive and by Lemma~\ref{lem:AKhexact} we have an exact triangle:

\begin{center}
\begin{tikzcd}
\AKh(L_{n})\arrow[rr]&&\AKh((\sigma^{-n-1})')\{1\}\arrow[dl]\\&\AKh(\sigma^{-n})[n_-^1-n_-]\{3n_-^1-3n_-+2\}\arrow[ul]
\end{tikzcd}
\end{center}

Now, $n_-^1-n_-=n+1$, so we have that:

\begin{center}
\begin{tikzcd}
\AKh(L_{n})\arrow[rr]&&\AKh((\sigma^{-n-1})')\{1\}\arrow[dl]\\&\AKh(\sigma^{-n})[n+2]\{3n+5\}\arrow[ul]
\end{tikzcd}
\end{center}

As before we have that:

\begin{align}\label{eq:akh}
\AKh^i(\sigma^{-n})[n+2]\{3n+5\}\cong\begin{cases} V_2\{2n+5\}&\text{for } i=n+2\\
V_0\{2i+1\}&\text{if }2\leq i\leq n+1\text{ and }i\text{ is even}\\
V_0\{2i+3\}&\text{if }3\leq i\leq n\text{ and }i\text{ is odd}\\
0&\text{else.}\end{cases}
\end{align}

while $\AKh((\sigma^{-n-1})')\cong \AKh(\sigma^{-n-1})[n+1]\{3n+3\}$, so that $\AKh((\sigma^{-n-1})')\{1\}$ is given by:

\begin{align}
 \AKh(\sigma^{-n-1})[n+1]\{3n+4\}\cong \begin{cases}V_2\{2n+3\}&\text{for } i=n+1\\
V_0\{2i+1\}&\text{if }0\leq i\leq n\text{ and }i\text{ is odd}\\
V_0\{2i+3\}&\text{if }1\leq i\leq n-1\text{ and }i\text{ is even}\\
V_0\{3\}\oplus V_0\{1\}&\text{if }i=0\\
0&\text{else.}\end{cases}
\end{align}
If follows that the connecting homomorphism vanishes and we have the desired result.

\end{proof}

The $n\leq 1$ cases follow from the symmetry properties of annular Khovanov under mirroring, noting that for $n\leq 1$ $\overline{L_n}=L_{-n-1}$. Moreover, $\overline{L_0}=L_{-1}$. Thus it remains only to compute $\AKh(L_0)$.

\begin{lemma}
 The annular Khovanov homology of $L_0$ is given by 
\begin{align}
 \AKh^i(L_0)\cong\begin{cases} 
 V_2\{-3\}\oplus V_0\{-1\}&\text{for } i=0\\
 V_2\{-3\}&\text{for } i=-1\\
V_0\{-5\}&\text{if }i=-2\\
0&\text{ else.}
\end{cases}
\end{align}
\end{lemma}

\begin{proof}
      We proceed as in the previous Lemma. Applying Lemma~\ref{lem:AKhexact} to the crossing highlighted in Figure~\ref{annularlinks} we have the following exact triangle;

\begin{center}
\begin{tikzcd}
    \AKh(L_{0})\arrow[rr]&&\AKh(\sigma^{-1})[n_-^0-n_-]\{3n_-^0-3n_-+1\}\arrow[dl]\\&\AKh(\mathbf{1})\{-1\}\arrow[ul]
\end{tikzcd}
\end{center}

Here $\mathbf{1}$ is the closure of the identity two stranded braid where the orientations of the two components disagree. Observe that $n_-^0=1$, $n_-=2$, we in turn have 

\begin{center}
\begin{tikzcd}
    \AKh(L_{0})\arrow[rr]&&\AKh(\sigma^{-1})[-1]\{-2\}\arrow[dl]\\&\AKh(\mathbf{1})\{-1\}\arrow[ul]
\end{tikzcd}
\end{center}

Now, Grigsby-Licata-Wehrli~\cite[Section 9.3]{grigsby_annular_2018}, computed $\AKh(\sigma^{-1})$, so we have that:

\begin{align}
\AKh^i(\sigma^{-1})[-1]\{-2\}\cong\begin{cases} V_2\{-3\}&\text{for } i=-1\\
V_0\{-5\}&\text{if }i=-2\\
0&\text{else.}\end{cases}
\end{align}

On the other hand it follows from the definition of annular Khovanov homology that $\AKh^0(\mathbf{1})\{-1\}\cong V_2\{-3\}\oplus V_0\{-1\}$ and $\AKh^i(\mathbf{1})\{-1\}\cong 0$ for $i\neq 0$. It follows that the exact triangle splits and we have the desired result.

\end{proof}

\subsection{Detection Result Proofs}\label{sec:AKHdetectproof}

To prove Theorem~\ref{thm:AKH} we apply Xie's spectral sequence from $\AKh(L;\C)$ to a version of instanton Floer homology called \emph{annular instanton Floer homology}~\cite{xie2021instantons}, denoted $\AHI(L;\C)$. We note that while appropriate versions of instanton Floer homology and Heegaard Floer homology are conjecturally equivalent~\cite[Conjecture 7.24]{kronheimer2010knots} there is currently no analogue of Xie's spectral sequence in the setting of Heegaard Floer homology.

We now state a limited version of Theorem~\ref{thm:mainbraid} in the context of annular instanton Floer homology. To do so recall that annular instanton Floer homology carries a version of an Alexander grading, the top non-trivial summand of which is a version of sutured instanton Floer homology by~\cite[Lemma 7.10]{xie_instanton_2019}. This should be compared with Corollary~\ref{cor:annularlinkfloer}.

\begin{lemma}\label{lem:wrappingnumber1detection}
    Suppose $L$ is an annular link such that the maximal non-trivial annular grading of $\AKh(L;\C)$ is $w\in\{1,2\}$. Then $L$ has wrapping number $w$.
\end{lemma}
\begin{proof}
    Suppose $L$ is as in the statement of the Lemma. Consider Xie's spectral sequence from $\AKh(L;\C)$ to $\AHI(L;\C)$~\cite{xie2021instantons}. It follows from~\cite[Theorem 5.16]{xie2021instantons} that $\AHI(L;\C)$ is either supported only in annular grading $0$ or the maximum non-trivial annular grading is $w$. 

Suppose towards a contradiction that $\AHI(L;\C)$ is supported only in annular grading $0$. Then~\cite[Corollary 1.16]{xie2021instantons} implies that $L$ is contained in a $3$-ball in the thickened annulus. But then $\AKh(L;\C,w)=0$, a contradiction.
\end{proof}
We now prove a version of Corollary~\ref{cor:annularlinkfloer} for annular instanton Floer homology.

\begin{proposition}\label{prop:AHI}
    Suppose $K$ is a knot with $\AHI(K;\C)$ of rank at most $2$ in the maximal non-trivial annular grading, which we assume is non-zero. Then either:
    \begin{enumerate}
        \item $K$ is a braid closure,
        \item $K$ is a clasp-braid closure or,
        \item $K$ is a stabilized clasp-braid closure.
    
    \end{enumerate}
\end{proposition}
We remark that a version of this statement could be proven for links if~\cite[Corollary 1.16]{ghosh2019decomposing} were to be generalized to a version of Corollary~\ref{cor:trivialinclusion}.

\begin{proof}[Proof of Proposition~\ref{prop:AHI}]
    The proof follows that of Corollary~\ref{cor:annularlinkfloer}, but applying results for instanton Floer homology rather than link Floer homology. Specifically, sutured instanton Floer homology behaves as sutured (Heegaard) Floer homology under sutured decomposition by~\cite[Proposition 7.11]{kronheimer2010knots}. Observe that any meridional surface~\footnote{We have called such surfaces ``longitudinal" thus far, in accordance with Martin's terminology~\cite{martin2022khovanov}, but use Xie's terminology here.} intersects $K$ since we are assuming that the maximal non-trivial Alexander grading is non-zero. It follows that the sutured manifold obtained by decomposing the exterior of $K$ in the thickened annulus has second homology group zero. Work of Ghosh-Li~\cite[Corollary 1.16]{ghosh2019decomposing} allows us to find annuli  to decompose along to obtain a reduced sutured manifold with sutured instanton Floer homology of rank at most $2$. Noting that this sutured manifold must have genus one boundary and hence be a knot exterior, we can apply the classification of knot exteriors in $S^3$ with sutured instanton Floer homology of rank $2$ given in~\cite[Theorem 1.4]{li2022seifert}. We can now re-run the topological argument given in Proposition~\ref{lem:nocomponentstep1} to conclude the result, noting that the axis is unknotted.
\end{proof}

\begin{remark}\label{rem:AKHAHI}
We return briefly to Question~\ref{q:nearlybraidedannular}. As we noted earlier, Grigsby-Ni showed that annular Khovanov homology detects braids amongst annular links~\cite{grigsby_sutured_2014}. More precisely, an annular link $L$ is a braid if and only if $\AKh(L;\Z/2)$ is of rank one in the maximal non-trivial annular grading. For annular links with annular Khovanov homology of rank $2$ in the maximal non-trivial annular grading for which Xie's spectral sequence from the maximum non-trivial annular grading of annular Khovanov homology collapses immediately, such a classification would be exactly that given in the annular instanton setting. It is less clear what would happen in the event that the spectral sequence does not collapse.
\end{remark}

We now study annular knots with annular Khovanov homology of rank $2$ in the maximal non-trivial annular grading, in the case that that grading is $0$, $1$ or $2$.

In the case that the maximum non-trivial annular grading of annular Khovanov homology is $0$ we have simply that $L$ is the unknot. To see this first note that $L$ has wrapping number zero by work of Zhang-Xie~\cite{xie_instanton_2019}. The fact that $K$ is the unknot then follows from the fact that the only link with Khovanov homology of rank two is the unknot~\cite{kronheimer2011khovanov}.

\begin{proposition}\label{prop:AKH1}
    Suppose $K$ is an annular knot. Suppose that the maximal non-trivial annular grading of $\AKh(K;\C)$ is one and that $\rank(\AKh(K;\C,1))\leq 2$. Then  $K$ is a one braid and $\rank(\AKh(K;\C,1))=1$.
    
\end{proposition}

\begin{proof}
    Suppose $K$ is as in the statement of the Proposition. By Lemma~\ref{lem:wrappingnumber1detection}, $L$ has wrapping number one. Moreover, we have that $0\leq\rank(\AHI(K;\C,1))\leq 2$, so the result now follows from Proposition~\ref{prop:AHI}.
\end{proof}

A similar statement holds in the case that the maximal non-trivial annular grading is two.

\begin{proposition}\label{prop:AKH2}
    Suppose $K$ is an annular knot. Suppose that the maximal non-trivial annular grading of $\AKh(K;\C)$ is two and that $\rank(\AKh(K;\C,2))\leq 2$. Then either:
    \begin{enumerate}
        \item $K$ is a braid, in which case $\rank(\AKh(K;\C,2))=1$ or;
        \item $K$ is a clasp-braid closure  in which case $\rank(\AKh(K;\C,2))=2$.
    \end{enumerate}
\end{proposition}

\begin{proof}
Suppose $K$ is as in the statement of the Proposition. By Lemma~\ref{lem:wrappingnumber1detection}, $K$ has wrapping number $2$. Consider Xie's spectral sequence from $\AKh(K;\C)$ to $\AHI(K;\C)$~\cite{xie2021instantons}. It follows from~\cite[Theorem 5.16]{xie2021instantons} that $\AHI(K;\C)$ has maximum non-trivial annular grading $2$ and $\AHI(K;\C)$ is of rank at most $2$ in that grading. The result now follows from Proposition~\ref{prop:AHI}, noting that the wrapping number of a stabilized clasp-braid is at least three.

\end{proof}

 We can apply Proposition~\ref{prop:AKH2} to obtain some new annular knot detection results.

\begin{proof}[Proof of Theorem~\ref{thm:AKH}]
    Suppose $K$ is an annular knot with $\AKh(K;\C)\cong\AKh(L_n)$ for some $n$. Observe that the hypothesis of Proposition~\ref{prop:AKH2} apply, so we have that $K$ is a clasp-braid closure, i.e. $K$ is of the form $L_n$ for some $n$. The result then follows from Lemma~\ref{lem:2nearbraidcomp}, which implies that if $\AKh(L_n)\cong\AKh(L_m)$ then $n=m$.
\end{proof}

\bibliographystyle{plain}
\bibliography{bibliography}
\end{document}